\documentclass{article}
\usepackage{geometry}
\usepackage{amsmath,amsthm,amssymb}
\usepackage{epsfig}
\usepackage{amssymb}
\usepackage{amsmath}
\usepackage{amssymb}
\usepackage{amsmath,amsthm}
\usepackage[latin1]{inputenc}
\usepackage[T1]{fontenc}
\usepackage{path}
\usepackage{ae,aecompl}
\usepackage{amsfonts}

\usepackage{bbm,euscript}

\usepackage[automark]{scrpage2}

\numberwithin{equation}{section}

\newcommand{\sm}{\!\setminus\!}
\newcommand{\dx}{\, dx}
\newcommand{\dy}{\, dy}

\newcommand{\dk}{\, dk}
\newcommand{\R}{\mathbb{R}}

\newcommand{\N}{{\mathbb{N}}}
\newcommand{\bs}{\boldsymbol}

\def\IntBRx0{\int_{B_R(x_0)}}

\def\IntB2R{\int_{B_{2R}}}
\def\IntBtRx0{\int_{B_{tR}(x_0)}}

\newcommand{\eqn}[1]{(\ref{#1})}

\allowdisplaybreaks

\DeclareMathOperator{\dist}{dist}

\newcommand{\Z}{\mathbb Z}

\newcounter{formel}

\newtheorem{defs}{Definition}[section]
\newtheorem{them}[defs]{THEOREM}

\newtheorem{lemma}[defs]{Lemma}
\newtheorem{corollary}[defs]{Corollary}
\newtheorem{assumption}[defs]{Assumption}
\newtheorem{remark}[defs]{Remark}
\newtheorem{props}[defs]{Proposition}
\newtheorem*{lemA}{Lemma A.1}

\newcommand{\EE}{\mathcal E}
\newcommand{\FF}{\mathcal F}

\newcommand{\II}{\mathcal I}
\newcommand{\JJ}{\mathcal J}
\newcommand{\KK}{\mathcal K}
\newcommand{\LL}{\mathcal L}
\newcommand{\MM}{\mathcal M}

\newcommand{\PP}{\mathcal P}
\newcommand{\RR}{\mathcal R}
\renewcommand{\SS}{\mathcal S}

\newcommand{\sech}{\mathrm{sech}\,}
\newcommand{\supp}{\mathrm{supp}\,}
\newcommand{\nn}{|{\mskip-2mu}|{\mskip-2mu}|}

\renewcommand{\SS}{\mathcal S}

\DeclareMathOperator{\re}{Re}

\newcounter{count}

  \title{A variational approach to solitary gravity-capillary interfacial waves with infinite depth}

\author{
D.~Breit\thanks{Heriot-Watt University, Department of Mathematics, EH14 4AS Edinburgh, UK} 
\and E. Wahl\'en\footnote{Centre for Mathematical Sciences, Lund University, PO Box 118, 22100 Lund, Sweden}
}

\date{}

\begin{document}
\maketitle

\begin{abstract}
We present an existence and stability theory for gravity-capillary solitary waves on the top surface of and interface between two 
perfect fluids of different densities, the lower one being of infinite depth.  
Exploiting a classical variational principle, we prove the existence of a minimiser of the wave energy $\EE$ subject to the constraint $\II=2\mu$, where $\II$ is the
wave momentum and $0< \mu <\mu_0$, where $\mu_0$ is chosen small enough for the validity of our calculations. Since $\EE$ and $\II$ are both conserved quantities a standard argument asserts
the stability of the set $D_\mu$ of minimisers: solutions starting near $D_\mu$ remain close to $D_\mu$ in a suitably defined energy space over their interval of existence. The solitary waves which we construct are of small amplitude and are to leading order described by the cubic nonlinear Schr\"odinger equation. They exist in a parameter region in which the `slow' branch of the dispersion relation has a strict non-degenerate global minimum and the corresponding nonlinear Schr\"odinger equation is of focussing type.  We show that the waves detected by our variational method converge (after an appropriate rescaling)
to solutions of the model equation as $\mu \downarrow 0$.
\end{abstract}

\section{Introduction}

\subsection{The model}

We consider a two-layer perfect fluid with irrotational flow subject to the forces of gravity, surface tension and interfacial tension. 
The lower layer is assumed to be of infinite depth, while the upper layer has finite asymptotic depth $\overline{h}$.
We assume that density $\underline{\rho}$ of the lower fluid is strictly greater than the density $\overline{\rho}$ of the upper fluid.
The layers are separated by a free interface $\{y=\underline{\eta}(x,t)\}$ and the upper one is bounded from above by a free surface $\{y=\overline{h}+\overline{\eta}(x,t)\}$. The fluid motion in each layer is described by the incompressible Euler equations. 
The fluid occupies the domain $\underline{\Sigma}(\underline{\eta}) \cup \overline{\Sigma}(\bs \eta)$, 
where
\begin{align*}
\underline{\Sigma}(\underline{\eta})&:=\left\{(x,y)\in \R^2\colon-\infty< y<\underline{\eta}(x,t)\right\},\\
\overline{\Sigma}(\bs \eta)&:=\left\{(x,y)\in \R^2\colon \underline{\eta}(x,t)< y< \overline{h}+ \overline{\eta}(x,t)\right\},
\end{align*}
and $\bs \eta=(\underline{\eta}, \overline{\eta})$. Since the flow is assumed to be irrotational 
in each layer, there exist velocity potentials $\underline{\phi}$ and $\overline{\phi}$ satisfying
\[
 \Delta \overline{\phi}=0 
\quad \text{in}\quad \overline{\Sigma}(\bs \eta),\qquad 
 \Delta \underline{\phi}=0 
\quad \text{in}\quad \underline{\Sigma}(\underline \eta).
\]
On the interface $\{y=\underline{\eta}\}$ we have the kinematic boundary conditions
\begin{align*}
\partial_t \underline{\eta}=\overline{\phi}_y-\underline{\eta}_x\overline{\phi}_x=(1+ \underline{\eta}_x^2)^{\frac12}  \partial_{\underline{\bs n}} \overline{\phi}, \qquad
\partial_t\underline{\eta} =\underline{\phi}_y-\underline{\eta}_x\underline{\phi}_x=(1+ \underline{\eta}_x^2)^{\frac12}\partial_{\underline{\bs n}} \underline{\phi},
\end{align*}
where 
\[
\underline{\bs n}=(1+\underline{\eta}_x^2)^{-\frac12}(1, -\underline{\eta}_x)
\]
is the upward unit normal vector to the interface. In particular, this implies that the normal component of 
the velocity is continuous across the interface. 
At the free surface $\{y=\overline{h}+\overline{\eta}\}$, the kinematic boundary condition reads
\begin{align*}
\partial_t \overline{\eta}&=\overline{\phi}_y-\overline{\eta}_x\overline{\phi}_x=(1+ \overline{\eta}_x^2)^{\frac12}  \partial_{\overline{\bs n}} \overline{\phi},
\end{align*}
where
\[
\overline{\bs n}=(1+\overline{\eta}_x^2)^{-\frac12}(1, -\overline{\eta}_x)
\]
is the outer unit normal vector at the  surface. In addition, 
we have the Bernoulli conditions
\[
\overline{\rho}\left(\partial_t\overline{\phi}+\frac{1}{2}|\nabla\overline{\phi}|^2+g\underline{\eta}\right)-\underline{\rho}\left(\partial_t\underline{\phi}+\frac{1}{2}|\nabla\underline{\phi}|^2+g\underline{\eta}\right)
=-\underline{\sigma}\left(\frac{\underline{\eta}_x}{\sqrt{1+\underline{\eta}_x^2}}\right)_x,
\]
and
\[
\overline{\rho}\left(\partial_t\overline{\phi}+\frac{1}{2}|\nabla\overline{\phi}|^2+g\overline{\eta}\right)
=\overline{\sigma}\left(\frac{\overline{\eta}_x}{\sqrt{1+\overline{\eta}_x^2}}\right)_x
\]
at the interface and surface, respectively, where $g>0$ is the acceleration due to gravity, $\overline{\sigma}>0$ the coefficient of surface tension and $\underline{\sigma}>0$ the coefficient of interfacial tension.

In order to obtain dimensionless variables we define
\begin{align*}
&(x',y'):=\frac{1}{\overline{h}}(x,y), &
&t':=\left(\frac{g}{\overline{h}}\right)^{\frac{1}{2}}t,\\
&\overline{\eta}'(x',t'):=\frac{1}{\overline{h}}\overline{\eta}(x,t), &
&\underline{\eta}'(x',t'):=\frac{1}{\overline{h}}\underline{\eta}(x,t),\\
&\overline{\phi}'(x',y',t'):=\frac{1}{(\overline{h})^{\frac{3}{2}}g^{\frac{1}{2}}}\overline{\phi}(x,y,t), & 
&\underline{\phi}'(x',y',t'):=\frac{1}{(\overline{h})^{\frac{3}{2}}g^{\frac{1}{2}}}\underline{\phi}(x,y,t),
\end{align*}
and obtain the equations (dropping the primes for notational simplicity)
\begin{align}
&\Delta\underline{\phi}=0,\qquad\qquad\qquad y<\underline{\eta}, \label{1a}\\
&\Delta\overline{\phi}=0,\qquad\qquad\qquad\underline{\eta}<y<1+\overline{\eta}\label{1b},
\end{align}
with boundary conditions
\begin{align}
\partial_t\underline{\eta}&=\underline{\phi}_y-\underline{\eta}_x\underline{\phi}_x, && y=\underline{\eta},\label{2a}\\
\partial_t\underline{\eta}&=\overline{\phi}_y-\underline{\eta}_x\overline{\phi}_x,  && y=\underline{\eta},\label{2b}\\
\partial_t\overline{\eta}&=\overline{\phi}_y-\overline{\eta}_x\overline{\phi}_x,  && y=1+\overline{\eta},\label{3b}\\
\nabla \underline{\phi}&\to 0,&&y\to -\infty,\label{3a}
\end{align}
and
\begin{align}
\rho\left(\partial_t\overline{\phi}+\frac{1}{2}|\nabla\overline{\phi}|^2+\underline{\eta}\right)
-\left(\partial_t\underline{\phi}+\frac{1}{2}|\nabla\underline{\phi}|^2+\underline{\eta}\right)
&=-\underline{\beta} \left(\frac{\underline{\eta}_x}{\sqrt{1+\underline{\eta}_x^2}}\right)_x,&& y=\underline{\eta}, \label{4a}\\
\partial_t\overline{\phi}+\frac{1}{2}|\nabla\overline{\phi}|^2+\overline{\eta}&=\overline{\beta}\left(\frac{\overline{\eta}_x}{\sqrt{1+\overline{\eta}_x^2}}\right)_x,&& y=1+\overline{\eta},\label{4b}
\end{align}
in which $\rho:=\overline{\rho}/\underline{\rho}\in(0,1)$, $\underline{\beta}:=\underline{\sigma}/(g \overline{h}^2 \underline{\rho})>0$ and $\overline{\beta}:=\overline{\sigma}/(g \overline{h}^2 \overline{\rho})>0$.
The total energy
\begin{align*}
\mathcal{E}&=\frac{1}{2}\int_{\overline{\Sigma}(\bs \eta)}\rho |\nabla\overline{\phi}|^2\,dx\, dy+\frac{1}{2}\int_{\underline{\Sigma}(\underline \eta)}|\nabla\underline{\phi}|^2\,dx\, dy\\
&\quad +\frac{1}{2}\int_{\R}(1-\rho)\underline{\eta}^2\,dx
+\frac{1}{2}\int_{\R}\rho\,\overline{\eta}^2\,dx
+\int_{\R}\underline{\beta} \left(\sqrt{1+\underline{\eta}_x^2}-1\right)\,dx
+\int_{\R}\rho\overline{\beta} \left(\sqrt{1+\overline{\eta}_x^2}-1\right)\,dx
\end{align*}
and the total horizontal momentum
\[
\mathcal{I}=\int_{\R} \underline{\eta}_x (\underline{\phi}|_{y=\underline{\eta}}-\rho \overline{\phi}|_{y=\underline{\eta}})\,dx 
+ \rho \int_{\R} \overline{\eta}_x \overline{\phi}|_{y=1+\overline{\eta}}\,dx
\]
are conserved quantities.

Our interest lies in solitary-wave solutions of \eqref{1a}--\eqref{4b}, that is, localised waves of permanent form which propagate in the negative $x$-direction with 
constant (dimensionless) speed $\nu>0$, so that $\underline{\eta}(x,t)=\underline{\eta}(x+\nu t)$, $\overline{\eta}(x,t)=\underline{\eta}(x+\nu t)$, $\underline{\phi}(x,y,t)=\underline{\phi}(x+\nu t,y)$ and $\overline{\phi}(x,y,t)=\overline{\phi}(x+\nu t,y)$, and $\underline{\eta}(x+\nu t), \overline{\eta}(x+\nu t) \to 0$ as $|x+\nu t|\to \infty$.  Figure \ref{setting} contains a sketch of the physical setting.

\begin{figure}
\centering
\includegraphics[width=6cm]{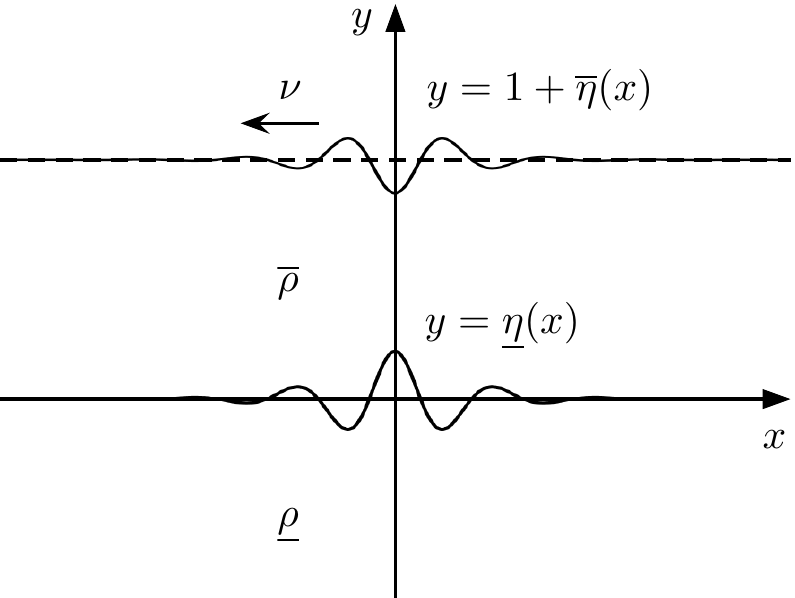}

\caption{Sketch of the physical setting and the waves obtained in this paper (in dimensionless variables).}
\label{setting}
\end{figure}

\subsection{Heuristics}
\label{Heuristics}

The existence of small-amplitude solitary waves can be predicted by studying the dispersion relation of the linearised 
version of \eqref{1a}--\eqref{4b}. Linear waves of the form $\bs \eta(x,t)=\cos k(x+\nu t) \bs v$ exist whenever 
\[
(P(k)-\nu^2 F(k))\bs v=0,
\]
where
\begin{equation}
\label{formulas for P and F}
P(k)=
\begin{pmatrix} 
1-\rho+\underline{\beta}|k|^2 & 0 \\ 
0 & \rho(1+\overline{\beta}|k|^2)
\end{pmatrix}
\quad 
\text{and}
\quad
F(k)=
\begin{pmatrix}
|k|+\rho |k| \coth |k|& -\rho\frac{|k|}{\sinh |k|}\\ 
-\rho \frac{|k|}{\sinh |k|}&  \rho |k| \coth |k| 
\end{pmatrix}
\end{equation}
(see equation \eqref{K2 and L2 formulas} below).
Equivalently, 
$\nu^2$ is an eigenvalue and $\bs v$ an eigenvector of the matrix
\begin{align*}
F(k)^{-1} P(k)=
\frac{1}{|k| \coth |k|+\rho |k|} 
\begin{pmatrix} 
(1-\rho+\underline{\beta}|k|^2)\coth |k|  & (1+\overline{\beta}|k|^2)\frac{\rho}{\sinh |k|} \\ 
(1-\rho+\underline{\beta}|k|^2)\frac{1}{\sinh |k|}& (1+\overline{\beta}|k|^2)\left(1+ \rho\coth |k|\right)
\end{pmatrix}
\end{align*}
(assuming that $k\ne 0$ so that $F(k)$ is invertible).
The eigenvalues are given by
\begin{align*}
\lambda_\pm(k) = \frac{(1-\rho+\underline{\beta}|k|^2)+(1+\overline{\beta}|k|^2)(\tanh |k|+\rho)}{2|k|(1+\rho \tanh |k|)}
\pm \frac1{2|k|(1+\rho \tanh |k|)}\sqrt{D(k)},
\end{align*}
with
\begin{align*}
 D(k)=
\left((1-\rho+\underline{\beta}|k|^2)-\left(\tanh|k|+\rho\right)(1+\overline{\beta}|k|^2)\right)^2
+\frac{4\rho}{\cosh^2|k|}(1-\rho+\underline{\beta}|k|^2)(1+\overline{\beta}|k|^2)
>0.
\end{align*}
It follows that $\lambda_-(k)<\lambda_+(k)$ for all $k\ne 0$, meaning that for each wavenumber $k\ne 0$ there is an associated `slow' speed $\sqrt{\lambda_-(k)}$ and a `fast' speed $\sqrt{\lambda_+(k)}$ (see Figure \ref{dispersion1}).
Moreover, 
\[
\lambda_+(k)
=\frac{1}{|k|}+O(1)
\quad \text{and} \quad 
\lambda_-(k)
=1-\rho-\rho(1-\rho)|k|+O(1)
\]
as $k\to 0$.
As $|k|\to \infty$ we have that
\[
\lambda_\pm(k) =
\frac{\underline{\beta}+(1+\rho)\overline{\beta}
\pm \left|\underline{\beta}-
\left(1+\rho\right)\overline{\beta}\right|
}{2(1+\rho)} |k|+O(1).
\]
Since
\[
\left|\underline{\beta}-
\left(1+\rho\right)\overline{\beta}\right|<\underline{\beta}+(1+\rho)\overline{\beta},
\]
we have that $\lambda_\pm (k) \to \infty$ as $|k|\to \infty$. In view of the behaviour at $0$, we conclude that 
$\lambda_-(k)$ is minimised at some $k=k_0>0$. 
\begin{figure}
\centering
\includegraphics[width=0.3\linewidth]{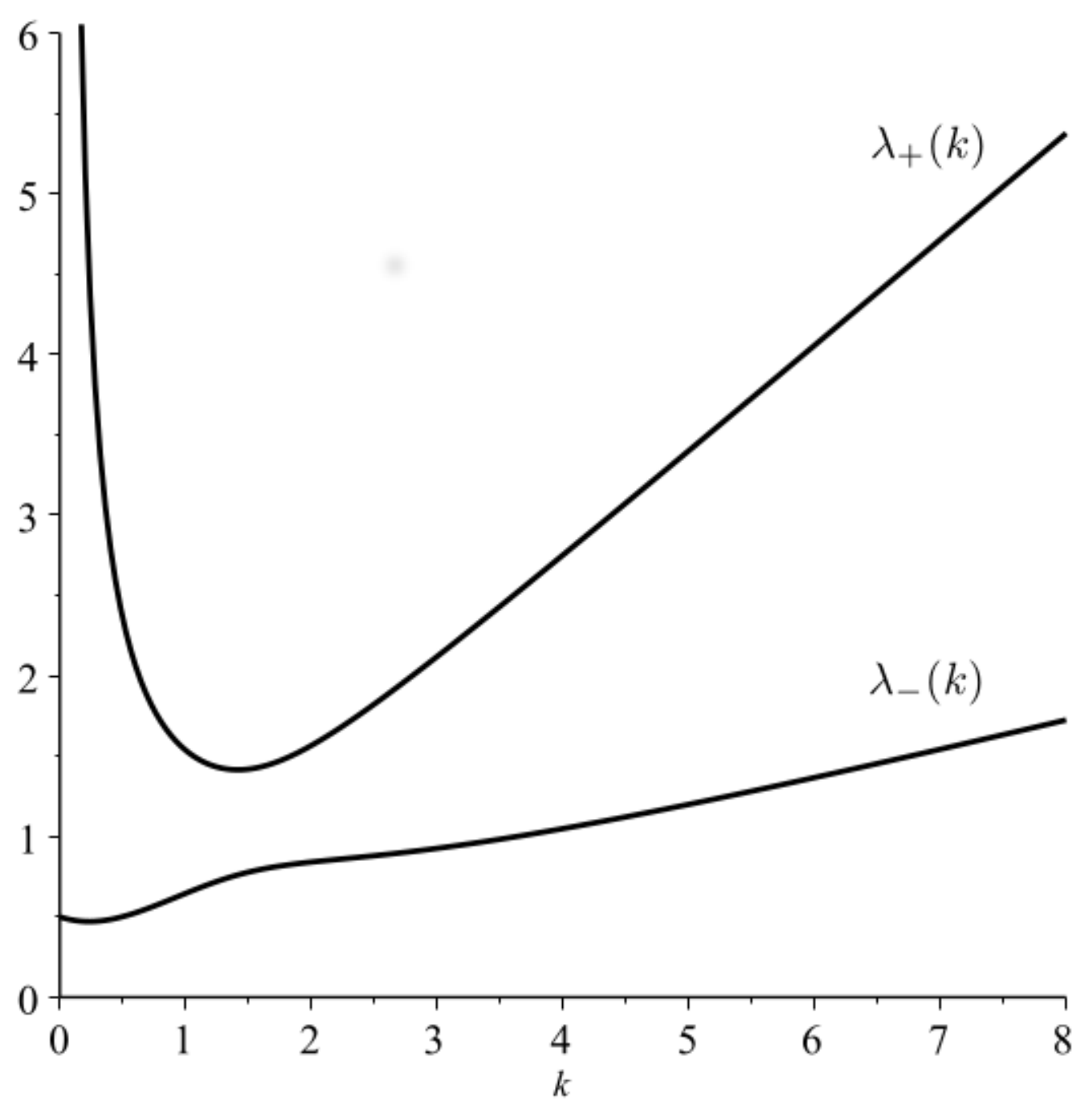}

\caption{Dispersion relation for the parameter values $\rho=0.5$, $\underline{\beta}=1$ and $\overline{\beta}=0.2$. The dispersion relation has a slow branch $\lambda_-(k)$ and  a fast branch $\lambda_+(k)$.}
\label{dispersion1}
\end{figure}

In order to find solitary waves we will assume the following non-degeneracy conditions.
\begin{assumption}
\label{assumption 1}
\begin{equation}
\label{non-degeneracy 1}
\lambda_-(k)>\lambda_-(k_0) \quad \text{for } k\ne \pm k_0
\end{equation}
and
\begin{equation}
\label{non-degeneracy 2}
\lambda_-''(k_0)>0.
\end{equation}
\end{assumption}

The first part of the assumption is introduced in order to avoid resonances. The second part is introduced in order to obtain the inequality \eqref{lower bound on g} below. This in turn dictates the choice of model equation (the cubic nonlinear Schr\"odinger equation). We note that these conditions are satisfied for generic parameter values, but that there are exceptions; see Figures \ref{dispersion2} and \ref{dispersion3}.

Set $\nu_0=\sqrt{\lambda_-(k_0)}$ and
note that $\bs v_0=(1,-a)$ is an eigenvector to the eigenvalue $\nu_0^2$ of the matrix $F(k_0)^{-1}P(k_0)$, in which
\begin{equation}
\label{formula for a}
a=
\frac{\tfrac12(1-\rho+\underline{\beta}|k_0|^2)-
\tfrac12 (1+\overline{\beta}|k_0|^2)(\tanh |k_0|+\rho)+\tfrac12 \sqrt{D(k_0)}}{(\rho+\overline{\beta}|k_0|^2)\frac{1}{\cosh |k|} }
>0.
\end{equation}
For future use we also introduce the matrix-valued function
\begin{align}
\nonumber
g(k)&:=P(k)-\nu_0^2 F(k)\\
\label{definition of g}
&=
\begin{pmatrix} 
1-\rho+\underline{\beta}|k|^2 & 0 \\ 
0 & \rho+\rho\overline{\beta}|k|^2 
\end{pmatrix}
-\nu_0^2\begin{pmatrix} |k|+\rho|k| \coth |k|& -\rho\frac{|k|}{\sinh |k|}
\\ -\rho\frac{|k|}{\sinh |k|}&  \rho|k| \coth |k| \end{pmatrix}, 
\end{align}
which satisfies $g(k_0)\bs v_0=0$ and (due to the second part of Assumption \ref{assumption 1} and evenness)
\begin{equation}
\label{lower bound on g}
g(k) \bs w \cdot \bs w \ge (\lambda_-(k)-\lambda_-(k_0)) F(k) \bs w \cdot \bs w \ge c(|k|-k_0)^2 |\bs w|^2
\end{equation}
for $||k|-k_0|\ll 1$, where $c$ is a positive constant.

\begin{figure}
\centering
\includegraphics[width=0.3\linewidth]{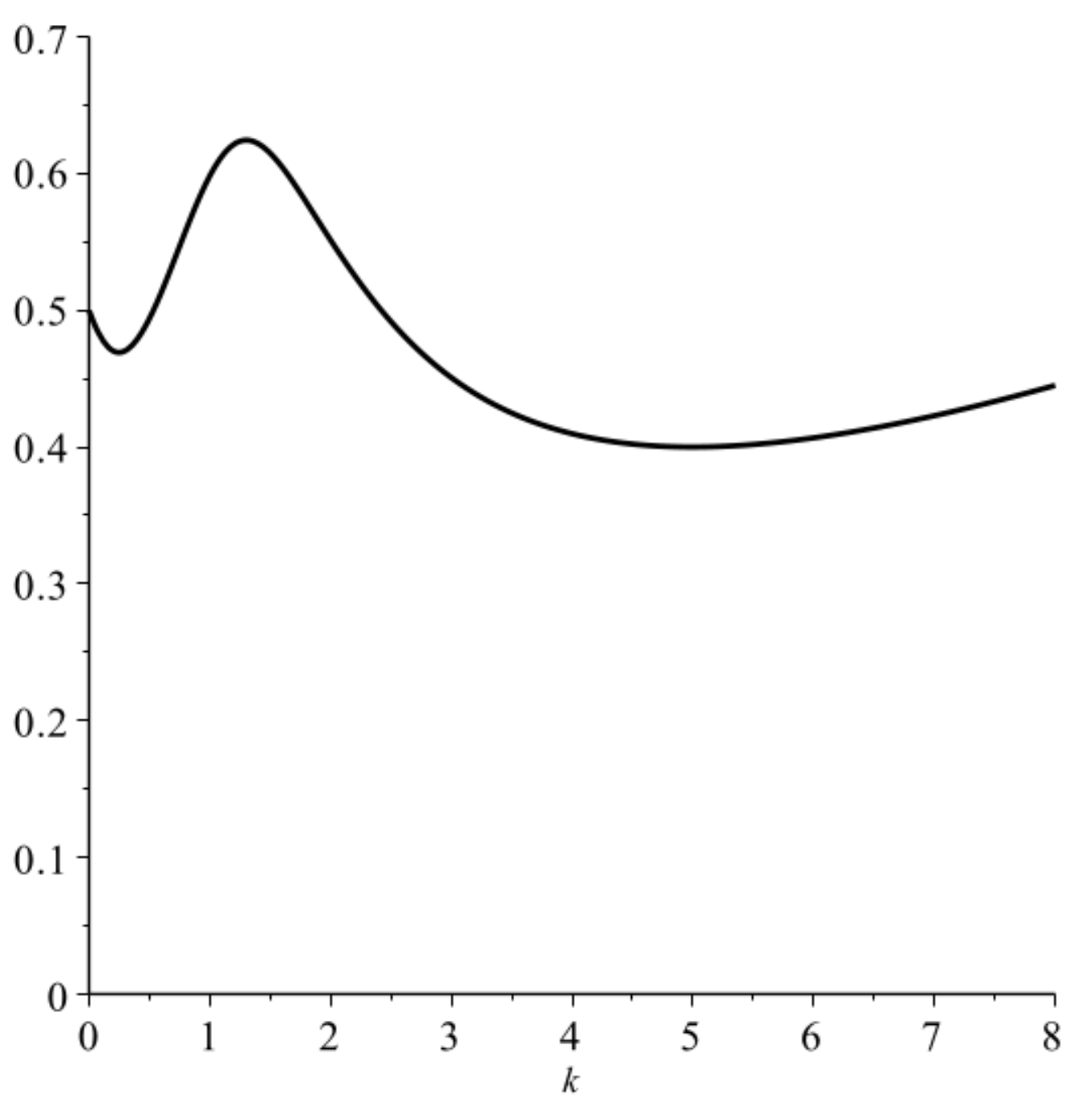} \hspace{1mm}
\includegraphics[width=0.3\linewidth]{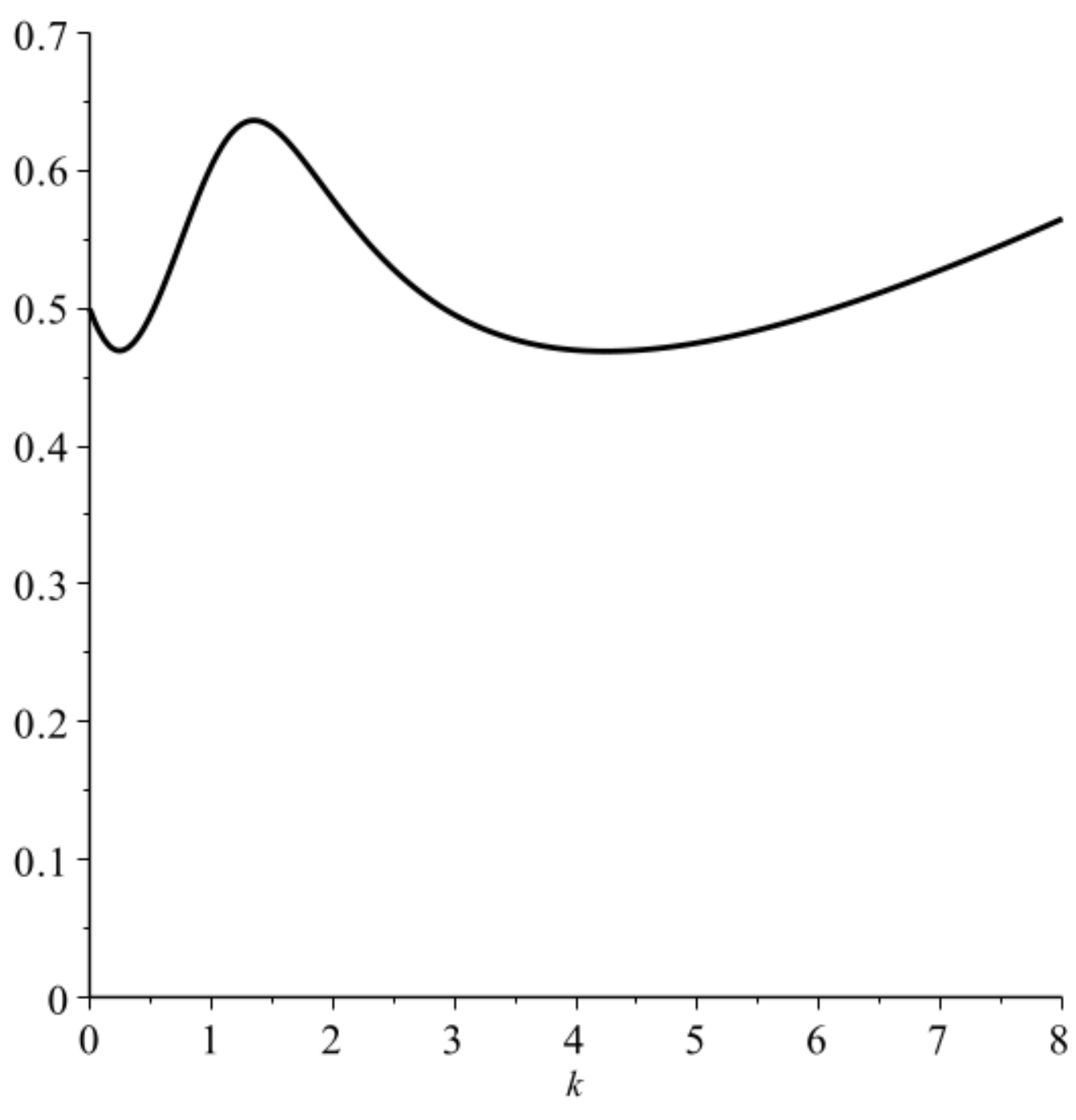} \hspace{1mm}
\includegraphics[width=0.3\linewidth]{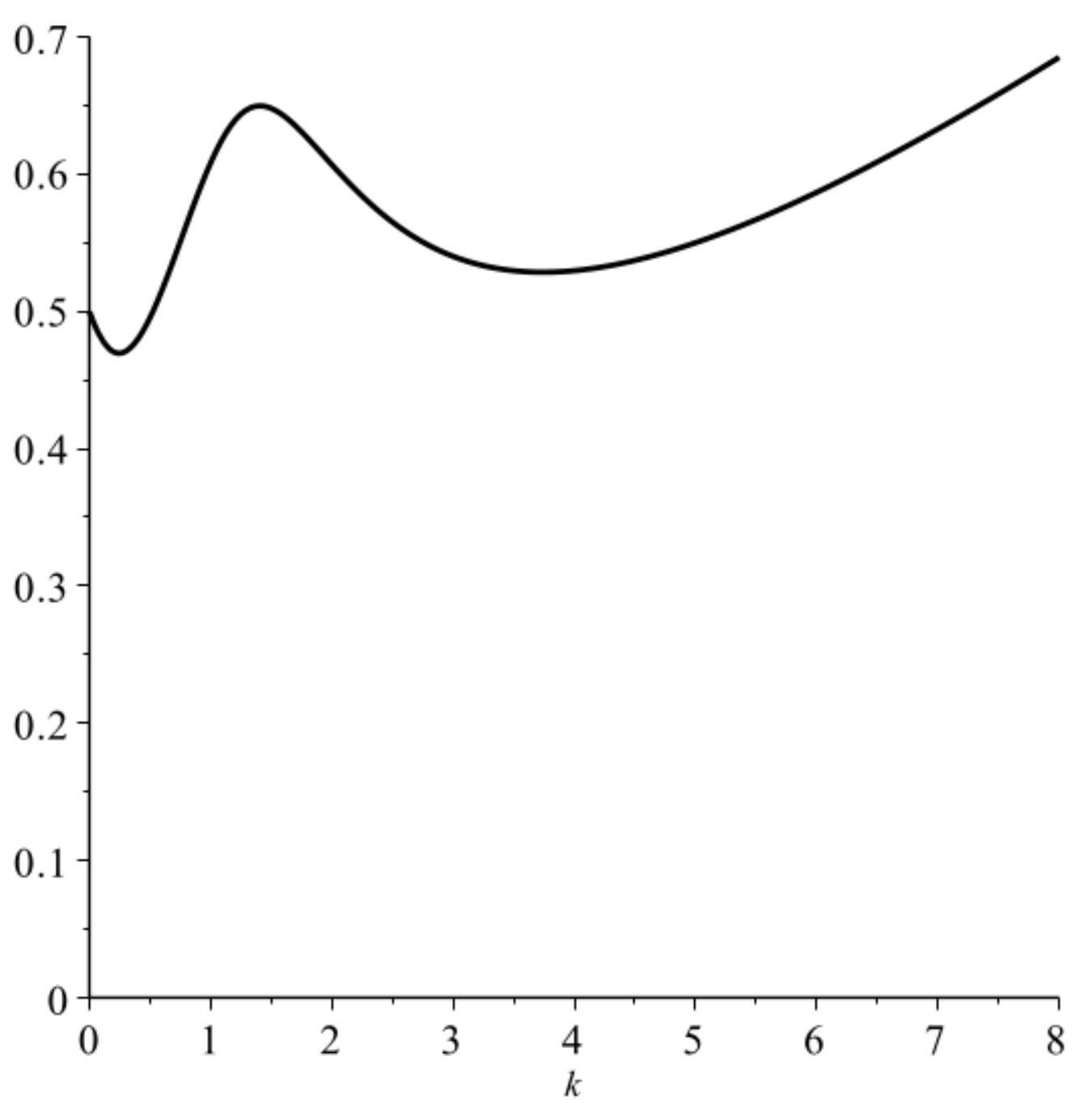}

\caption{Graphs of $\lambda_-(k)$ for $\rho=0.5$, $\underline{\beta}=1$ and three different values of $\overline{\beta}$. In the first and third cases ($\overline{\beta}=0.04$ and $\overline{\beta}=0.07$, respectively) both conditions \eqref{non-degeneracy 1} and \eqref{non-degeneracy 2} are satisfied. In the second case ($\overline{\beta}\approx 0.055$) condition \eqref{non-degeneracy 1} is violated.}
\label{dispersion2}
\end{figure}

Bifurcations of nonlinear solitary waves are expected whenever the linear group and phase speeds are equal, so that $\nu^\prime(k)=0$ (see Dias \& Kharif \cite[Section 3]{DiasKharif99}).
We therefore expect the existence of small-amplitude solitary waves with speed near $\nu_0$, bifurcating from a linear periodic wave train with frequency $k_0 \nu_0$. 
Making the Ansatz
\[
\eta=\frac{1}{2}\mu (A(X,T)\mathrm{e}^{\mathrm{i}k_0(x+\nu_0 t)}+\mathrm{c.c.}){\bs v}_0+O(\mu^2),
\]
\[
X=\mu(x+\nu_0 t),\qquad T=2 k_0 (\nu_0 F(k_0)\bs v_0\cdot \bs v_0)^{-1}  \mu^2 t,
\]
where `$\mathrm{c.c.}$' denotes the complex conjugate of the
preceding quantity,
and expanding in powers of $\mu$ one obtains the cubic nonlinear Schr\"{o}dinger equation
\begin{equation}
2\mathrm{i}A_T -\tfrac{1}{4}A_2 A_{XX} + \tfrac32\left(\tfrac12 A_3+A_4\right) |A|^2 A =0, \label{T-NLS}
\end{equation}
for the complex amplitude $A$, in which
\[
A_2=g''(k_0)\bs v_0\cdot \bs v_0
\] 
and $A_3$ and $A_4$ are functions of $\rho$, $\underline{\beta}$ and $\overline{\beta}$ which are given 
in Proposition \ref{lot in threes} and Corollary  \ref{lot in fours}.
At this level of approximation a
standing wave solution to \eqn{T-NLS} of the form
$A(X,T)=\mathrm{e}^{\mathrm{i}\nu_\mathrm{NLS} T}\phi(X)$
with $\phi(X) \rightarrow 0$ as $X \rightarrow \pm\infty$ corresponds
to a solitary water wave with speed
$$\nu=\nu_0 + 2  (\nu_0 F(k_0)\bs v_0\cdot \bs v_0)^{-1} \mu^2\nu_\mathrm{NLS}.$$

\begin{figure}
\centering
\includegraphics[width=0.3\linewidth]{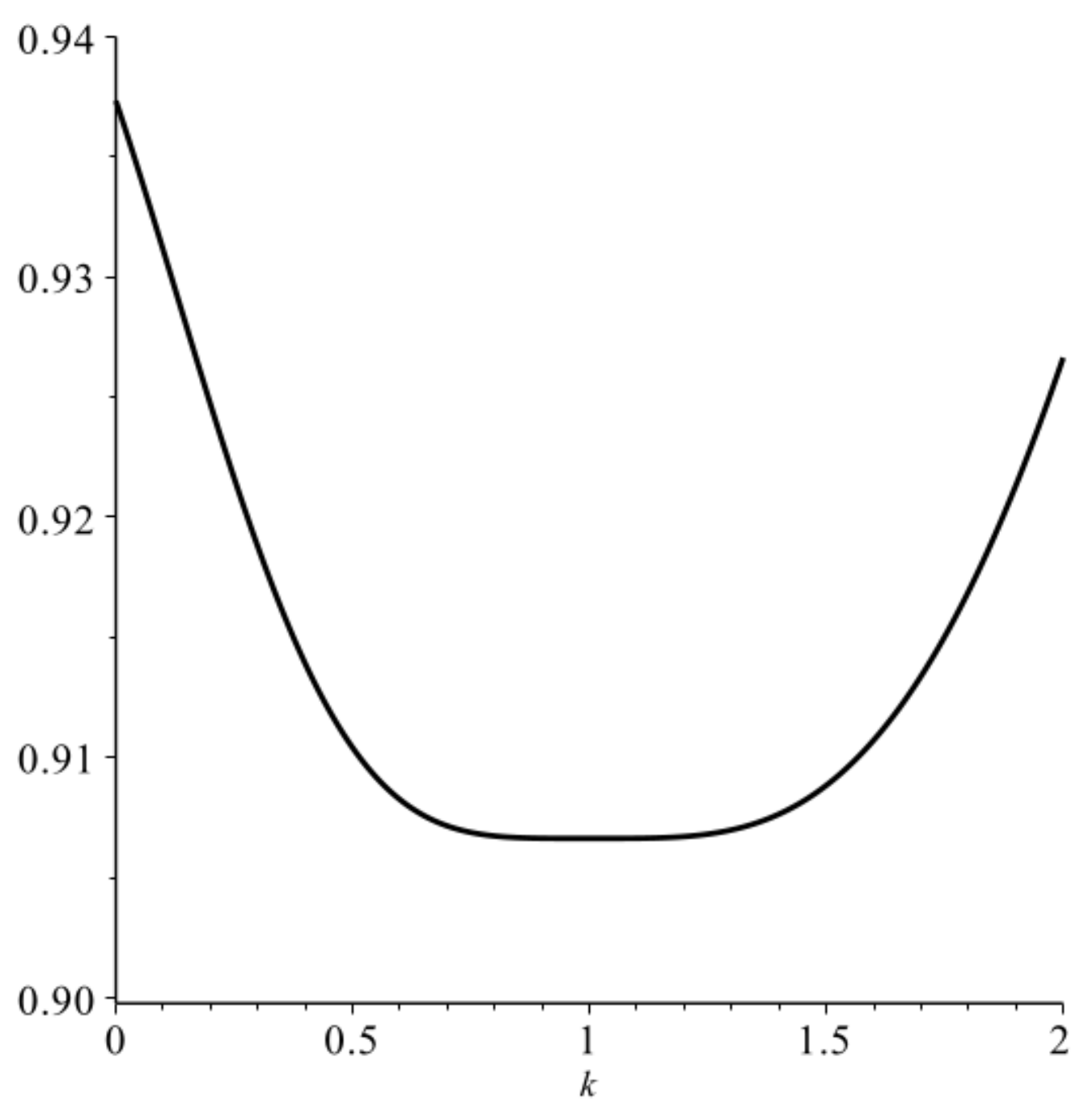} 

\caption{Numerical computations indicate that $\lambda_-(k)$ has a degenerate minimum at $k=1$ ($\lambda'(1)=\lambda''(1)=\lambda''''(1)=0$, $\lambda^{(iv)}(1)>0$) for  $\rho\approx 0.063$, $\underline{\beta}\approx 0.939$, $\overline{\beta}\approx 0.232$, in violation of condition \eqref{non-degeneracy 2}.}
\label{dispersion3}
\end{figure}

\begin{lemma}
$A_2>0$ under Assumption \ref{assumption 1}.
\end{lemma}

\begin{proof}
Let $\bs v(k)$ be a smooth curve of eigenvectors of $F(k)^{-1}P(k)$ corresponding to the eigenvalue $\lambda_-(k)$ with $\bs v(0)=\bs v_0$.
Then 
\[
(P'(k)-\lambda_-(k) F'(k))\bs v(k)+(P(k)-\lambda_-(k) F(k))\bs v'(k)=\lambda_-'(k)F(k)\bs v(k)
\]
and 
\begin{align*}
&(P''(k)-\lambda_-(k) F''(k))\bs v(k)+2(P'(k)-\lambda_-(k) F'(k))\bs v'(k)+(P(k)-\lambda_-(k) F(k))\bs v''(k)\\
&\qquad =\lambda_-''(k) F(k)\bs v(k)+2\lambda_-'(k)F'(k)\bs v(k)+2\lambda_-'(k)F(k)\bs v'(k).
\end{align*}
Evaluating the first equation at $k=k_0$ and using that $\lambda_-'(k_0)=0$, we find that
\[
g'(k_0)\bs v_0 =- g(k_0) \bs v'(k_0).
\]
Taking the scalar product of the second equation with $\bs v(k)$, evaluating at $k=k_0$ and using the previous equality, we therefore find that
\[
g''(k_0) \bs v_0\cdot \bs v_0=\lambda_-''(k_0) F(k_0)\bs v_0\cdot \bs v_0+2g(k_0)\bs v'(k_0)\cdot \bs v'(k_0),
\]
where we have also used that  $g(k_0)\bs v_0=0$. This concludes the proof since 
$\lambda_-''(k_0)>0$ and $F(k_0)$ and $g(k_0)$ are positive definite.
\end{proof}

It follows that a necessary and sufficient condition for \eqn{T-NLS} to possess solitary standing waves
is that the coefficient in front of the cubic term is negative. 

\begin{assumption}
\label{assumption 2}
\begin{equation}
\label{non-degeneracy 3}
\frac12 A_3+A_4<0.
\end{equation}
\end{assumption}

The following lemma gives a variational description of the set of such solutions (see Cazenave \cite[Section 8]{Cazenave03}).

\begin{lemma} \label{Variational NLS}
Assume that $A_2>0$ and $\frac12 A_3+A_4<0$.
The set of complex-valued solutions to the ordinary differential equation
$$
-\frac{1}{4}A_2\phi^{\prime\prime}-2\nu_\mathrm{NLS}\phi+\frac{3}{2}\left(
\frac{A_3}{2}+A_4\right)|\phi|^2\phi=0
$$
satisfying $\phi(X) \rightarrow 0$ as $X \rightarrow \infty$ is
$D_\mathrm{NLS} = \{\mathrm{e}^{\mathrm{i}\omega}\phi_\mathrm{NLS}(\cdot + y)\colon \omega \in [0,2\pi), y \in \R\},$
where
\begin{align*}
\nu_\mathrm{NLS} & = -\frac{9\alpha_\mathrm{NLS}^2}{8A_2}\left(\frac{A_3}{2}+A_4\right)^{\!\!2}, \\
\phi_\mathrm{NLS}(x) & =  \alpha_\mathrm{NLS}\left(-\frac{3}{A_2}\left(\frac{A_3}{2}+A_4\right)\right)^{\!\!\frac{1}{2}}
\sech\left(-\frac{3\alpha_\mathrm{NLS}}{A_2}\left(\frac{A_3}{2}+A_4\right)x\right).
\end{align*}
These functions are precisely the minimisers of the functional
$\EE_\mathrm{NLS}:H^1(\R) \rightarrow \R$ given by
$$
\EE_\mathrm{NLS}(\phi)=\int_{\R}\left\{\frac{1}{8}A_2|\phi^\prime|^2
+\frac{3}{8}\left(\frac{A_3}{2}+A_4\right)|\phi|^4\right\} \dx
$$
over the set $N_\mathrm{NLS} = \{\phi \in H^1(\R): \|\phi\|_0^2=2\alpha_\mathrm{NLS}\}$, where $\alpha_\mathrm{NLS}=2(\nu_0 k_0+\nu_0\rho \overline{F}(k_0)\bs v_0\cdot \bs v_0)^{-1}$;
the constant $2\nu_\mathrm{NLS}$ is the Lagrange multiplier in this constrained variational
principle and
$$
I_\mathrm{NLS}:=\inf\left\{\EE_\mathrm{NLS}(\phi)\colon \phi\in N_\mathrm{NLS}\right\}
=-\frac{3\alpha_\mathrm{NLS}^3}{4A_2}\left(\frac{A_3}{2}+A_4\right)^{\!\!2}.
$$
\end{lemma}

\subsection{Main results}

The main result of this paper is an existence theory for small-amplitude solitary-wave solutions to equations \eqref{1a}--\eqref{4b} under Assumptions \ref{assumption 1} and \ref{assumption 2}. The waves are constructed by minimising the energy functional $\mathcal{E}$ subject to the constraint of fixed horizontal momentum $\mathcal{I}$; see Theorem \ref{Result for constrained minimisation} for a precise statement.  As a consequence of the existence result we also obtain a stability result for the set of minimisers; see Theorem \ref{CES}. 

Before describing our approach in further detail, we note that the above formulation of the hydrodynamic problem has the disadvantage of being posed in a priori unknown domains.
It is therefore convenient to reformulate the problem in terms of the traces of the velocity potentials on the free surface and interface.
We denote the boundary values of the velocity potentials by $\underline{\Phi}(x):=\underline{\phi}(x,\underline{\eta}(x))$ and $\overline{\bs \Phi}(x)=(\overline{\Phi}_i(x),\overline{\Phi}_s(x))$ where $\overline{\Phi}_i(x):=\overline{\phi}(x,\underline{\eta}(x))$ and $\overline{\Phi}_s(x):=\overline{\phi}(x,1+\overline{\eta}(x))$. 
Following Kuznetsov \& Lushnikov \cite{KuznetsovLushnikov95} and Benjamin \& Bridges \cite{BenjaminBridges97} (see also \cite{CraigGroves00, CraigGuyenneKalisch05}) we set
\begin{equation}
\label{eq:canonical variables}
\underline{\xi}(x):=\underline{\Phi}(x)-\rho\overline{\Phi}_i(x),\qquad\overline{\xi}(x):=\rho\overline{\Phi}_s(x);
\end{equation}
the natural choice of canonical variables is $(\bs \eta,\bs \xi)$, where $\bs \eta=(\underline{\eta},\overline{\eta})$, $\bs\xi=(\underline{\xi},\overline{\xi})$. 
We formally define Dirichlet-Neumann operators $\underline{G}(\underline{\eta})$ and $\overline{G}(\bs \eta)$ which map (for a given $\bs \eta$) Dirichlet boundary-data of solutions of the Laplace-equation to the Neumann boundary-data, i.e.
\begin{align*}
\underline{G}(\underline{\eta})\underline{\Phi}&:=(1+\underline{\eta}_x^2)^{\frac{1}{2}}(\nabla\underline{\phi}\cdot \underline{\bs n})|_{y=\underline{\eta}},\\
\overline{G}(\bs \eta)\overline{\bs \Phi}&:=\begin{pmatrix}\overline{G}_{11}(\bs \eta)&\overline{G}_{12}(\bs \eta)\\
\overline{G}_{21}(\bs \eta)&\overline{G}_{22}(\bs \eta)\end{pmatrix}\begin{pmatrix}\overline{\Phi}_i\\\overline{\Phi}_s\end{pmatrix}:=
\begin{pmatrix}
-(1+\underline{\eta}_x^2)^{\frac{1}{2}}(\nabla\overline{\phi}\cdot \underline{\bs n})|_{y=\underline{\eta}}\\
(1+\overline{\eta}_x^2)^{\frac{1}{2}}(\nabla\overline{\phi}\cdot \overline{\bs n})|_{y=1+\overline{\eta}}\end{pmatrix};
\end{align*}
see Section \ref{FA setting} for the rigorous definition.
Note that $\underline{G}$ only depends on $\underline{\eta}$, whereas $\overline{G}$ depends on $\underline{\eta}$ and $\overline{\eta}$.
The boundary conditions \eqref{2a}--\eqref{2b} imply that
\begin{equation}
\label{eq:continuity in normal direction}
\underline{G}(\underline{\eta})\underline{\Phi}=-(\overline{G}_{11}(\bs \eta)\overline{\Phi}_i 
+\overline{G}_{12}(\bs \eta)\overline{\Phi}_s).
\end{equation}
If we define
\begin{equation}
\label{definition of B}
B(\bs \eta):=\overline{G}_{11}(\bs \eta)+\rho\underline{G}(\underline{\eta}),
\end{equation}
we can recover $\underline{\Phi}$ and $\overline{\Phi}$ from $\bs \xi$ using the formulas
\begin{align}
\label{eq:expressing Phi in terms of xi}
\begin{split}
\underline{\Phi}&=B^{-1}\overline{G}_{11}\underline{\xi}-B^{-1}\overline{G}_{12}\overline{\xi},\\
\overline{\Phi}_i&=-B^{-1}\underline{G}\underline{\xi}-\frac1{\rho} B^{-1}\overline{G}_{12}\overline{\xi},\\
\overline{\Phi}_s&=\frac1{\rho} \overline{\xi},
\end{split}
\end{align}
under assumption \eqref{eq:continuity in normal direction}.
Moreover, the total energy and horizontal momentum can be reexpressed as
\begin{equation}
\label{definition of H}
\mathcal{E}(\bs \eta, \bs \xi)=
\int_{\R} \left\{ \frac{1}{2}\bs \xi\, G(\bs \eta)\bs \xi+\frac{1-\rho}{2}\underline{\eta}^2+\frac{\rho}{2}\,\overline{\eta}^2
+\underline{\beta} \left(\sqrt{1+\underline{\eta}_x^2}-1\right)
+\rho\overline{\beta} \left(\sqrt{1+\overline{\eta}_x^2}-1\right)\right\}\,dx
\end{equation}
and
\begin{equation}
\label{Definition of I}
\mathcal{I}(\bs \eta,\bs\xi)=\int_{\R} \underline{\eta}_x \underline{\xi}\,dx 
+  \int_{\R} \overline{\eta}_x \overline{\xi}\,dx,
\end{equation}
respectively,
where we have abbreviated
\begin{equation}
\label{definition of GG}
G(\bs \eta):=\begin{pmatrix}\underline{G}(\underline{\eta})
B(\bs \eta)^{-1}\overline{G}_{11}(\bs \eta)
&-\underline{G}(\underline{\eta})B(\bs \eta)^{-1}\overline{G}_{12}(\bs \eta)\\
-\overline{G}_{21}(\bs \eta)B(\bs \eta)^{-1}\underline{G}(\underline{\eta})&\tfrac{1}{\rho}\overline{G}_{22}(\bs \eta)-\tfrac{1}{\rho}\overline{G}_{21}(\bs \eta)B(\bs \eta)^{-1}\overline{G}_{12}(\bs \eta)\end{pmatrix}.
\end{equation}
Note that
\[
G(\bs \eta)\bs \xi=\begin{pmatrix}\underline{G}(\underline{\eta})\underline{\Phi} \\
\overline{G}_{21}(\bs \eta)\overline{\Phi}_s+\overline{G}_{22}(\bs \eta)\overline{\Phi}_i\end{pmatrix}.
\]

We now give a brief outline of the variational existence method. 
We tackle the problem of finding minimisers of $\mathcal{E}(\bs \eta, \bs \xi)$ under the constraint $\mathcal{I}(\bs \eta,\bs\xi)=2\mu$ 
in two steps.
\begin{itemize}
\item[1.] Fix $\bs \eta\neq 0$ and minimise $\mathcal{E}(\bs \eta,\cdot)$ over $T_\mu:=\left\{\bs \xi \in \tilde X \colon \mathcal{I}(\bs \eta,\bs \xi)=2\mu\right\}$, where the space $\tilde X$ is defined in Section \ref{FA setting}. This problem (of minimising a quadratic functional over a linear manifold) admits a unique global minimiser $\bs \xi$.
\item[2.] Minimise $\mathcal{J}_\mu(\bs \eta):=\mathcal{E}(\bs \eta,\bs \xi_{\bs \eta})$ over $\bs \eta\in U\setminus\{0\}$ with $U:=B_M(0)\subset H^2(\R)$. Because $\bs \xi_{\bs \eta}$ minimises $\mathcal{E}(\bs \eta,\cdot)$ over $T_\mu$ there exists a Lagrange multiplier $\gamma_{\bs \eta}$ such that
\begin{align*}
G(\bs \eta)\bs \xi_{\bs \eta}=\gamma_{\bs \eta} \bs \eta_x.
\end{align*}
Hence
\begin{align*}
\bs \xi_{\bs \eta}&=\gamma_{\bs \eta}G(\bs \eta)^{-1}\bs \eta_x.
\end{align*}
Furthermore we get
\begin{align}
\gamma_{\bs \eta}=\frac{\mu}{\mathcal{L}(\bs \eta)}, \qquad
\mathcal{L}(\bs \eta)=\frac{1}{2}\int_{{\mathbb R}}\bs \eta K(\bs \eta)\bs \eta \,dx,
\label{Definition of L}
\end{align}
where
\begin{equation}
K(\bs \eta)=-\partial_x G(\bs \eta)^{-1}\partial_x
=
-\partial_x
\begin{pmatrix}
\rho\overline{N}_{11}(\bs \eta)  +\underline{N}(\underline{\eta}) & 
-\rho \overline{N}_{12}(\bs \eta)\
\\
-\rho \overline{N}_{21}(\bs \eta) & \rho\overline{N}_{22}(\bs \eta)
\end{pmatrix}\partial_x,
\end{equation}
with
$\underline{N}(\underline{\eta}):=\underline{G}(\underline{\eta})^{-1}$ and 
\[
\overline{N}(\eta)=
\begin{pmatrix}
\overline{N}_{11}(\bs \eta)&\overline{N}_{12}(\bs \eta)\\
\overline{N}_{21}(\bs \eta)&\overline{N}_{22}(\bs \eta)
\end{pmatrix}
:=
\overline{G}(\bs \eta)^{-1};
\]
see Proposition \ref{G is invertible}. 
For $\mathcal{J}_\mu(\bs \eta)$ we obtain the representation
\begin{equation}
\mathcal{J}_\mu(\bs \eta)=\KK(\bs \eta)+\frac{\mu^2}{\LL(\bs \eta)},
\label{Definition of J}
\end{equation}
where
\begin{align*}
\begin{split}
\mathcal{K}(\bs \eta)&=\underline{\mathcal{K}}(\underline{\eta})+\overline{\mathcal{K}}(\overline{\eta}),\\
\underline{\mathcal{K}}(\underline{\eta}) &= \int_{\mathbb R} \left\{\frac{(1-\rho)}{2}\underline{\eta}^2 + \underline{\beta}\sqrt{1+\underline{\eta}_x^2}-\underline{\beta}
\right\} \dx, \\
\overline{\mathcal{K}}(\overline{\eta}) &= \rho\int_{\mathbb R} \left\{\frac{1}{2}\overline{\eta}^2 +  \overline{\beta}\sqrt{1+\overline{\eta}_x^2}-\overline{\beta}
\right\} \dx.
\end{split}
\end{align*}
We address the problem of minimising $\JJ_\mu$ using the concentration-compactness method. The main difficulties are that 
the functional is quasilinear, nonlocal and nonconvex. These difficulties are partly solved by minimising over 
a bounded set in the function space, but we then have to prevent minimising sequences from converging to the boundary of this set. 
This is achieved by constructing a suitable test function and a special minimising sequence with good properties using the intuition from the nonlinear Schr\"odinger equation above.
\end{itemize}

Our approach is similar to that originally used by Buffoni \cite{Buffoni04a} to study solitary waves with strong surface tension on a single layer of fluid of finite depth, 
and later extended to deal with weak surface tension  \cite{Buffoni05, Buffoni09, GrovesWahlen10}, infinite depth \cite{Buffoni04a, GrovesWahlen11}, fully localised three-dimensional waves \cite{BuffoniGrovesSunWahlen13} and constant vorticity \cite{GrovesWahlen15}. Our main interest is in investigating the nontrivial modifications needed to deal with multi-layer flows. We give detailed explanations when needed (see in particular the discussion of the vector-valued Dirichlet-Neumann operators in the next section) and refer to the above papers for the details of the proofs when possible.

Note that we could also have considered a bottom layer with finite depth. This introduces an additional dimensionless parameter in the problem (the ratio between the depths of the two layers), which allows for other phenomena (for example, the slow speed can have a minimum at the origin).
We refer to \cite{WoolfendenParau11} for a discussion of the dispersion relation and numerical computations of solitary waves in the finite depth case.
One of the reasons why we chose to look at the infinite depth problem is that it entails some technical challenges which invalidates the use of certain methods which are widely used to find solitary waves in hydrodynamics. In particular, the idea originally due to Kirchg\"assner \cite{Kirchgaessner82} of formulating the steady water wave problem as an ill-posed evolution equation and applying a centre-manifold reduction cannot be used. The variational method that we use is less sensitive to these issues. Note however that Kirchg\"assner's method has been extended to deal with the issues due to infinite depth by several authors (see \cite{BarrandonIooss05} and references therein) and this could have been used in order to construct solitary waves also in our setting. These methods give no information about stability, however.

As far as we are aware, there are no previous existence results for solitary waves in our setting. However, Iooss \cite{Iooss99b} constructed small-amplitude periodic travelling-wave solutions of problem \eqref{1a}--\eqref{4b} in two situations. The first situation is when the parameters are chosen so that $\nu^2=\lambda_+(k)$ or $\nu^2=\lambda_-(k)$ for some wavenumber $k\ne 0$ which is not in resonance with any other wavenumber (i.e.~$\lambda_{\pm}(nk) \ne \nu^2$ for all $n\in \Z$) and $\lambda_\pm'(k) \ne 0$ (where the sign is chosen such that $\lambda_{\pm}(k)=\nu^2$). The second situation is the $1:1$ resonance, that is when $k$ is a non-degenerate critical point of $\lambda_{\pm}$. In both situations he proved the existence of small amplitude waves with period close to $2\pi/k$ using dynamical systems techniques. The second situation includes our setting, but is somewhat more general (the critical point is e.g.~not assumed to be a minimum). There are also a number of papers dealing with solitary or generalised solitary waves (asymptotic to periodic solutions at spatial infinity) in the related settings 
where either one or both of the surface and interfacial tension vanishes (see \cite{Barrandon06, BarrandonIooss05, DiasIooss03, IoossLombardiSun02, LombardiIooss03, SunShen93} and references therein). The variational method presented in this paper does not work in those settings since it requires both surface tension and interfacial tension.
Finally, let us conclude this section by mentioning that our assumptions exclude two possibilities which could be interesting for further study (by variational or other methods), that is when $\lambda_-$ has a degenerate global minimum at $k_0$ (see Figure \ref{dispersion3}) or when the minimum value is attained at two distinct wave numbers (Figure \ref{dispersion2}). Also, when Assumption \ref{assumption 1} is satisfied, but the corresponding nonlinear Schr\"odinger equation is of defocussing type (so that Assumption \ref{assumption 2} is violated), one would expect the existence of dark solitary waves.

\section{The functional-analytic setting} \label{FA setting}

The goal of this section is to introduce rigorous definitions of the Dirichlet-Neumann operators 
$\underline{G}(\underline \eta)$ and $\overline{G}(\bs \eta)$ and their inverses $\underline{N}(\underline \eta)$ and  $\overline{N}(\bs \eta)$, as well as the operators $G(\bs \eta)$ and $K(\bs \eta)$.

\subsection{Definition of operators}

\subsubsection{Lower fluid}

In order to define $\underline{G}(\underline \eta)$ and $\underline{N}(\underline \eta)$, we first introduce suitable function spaces on which these operators are well-defined. We begin by recalling the definition of the Schwartz class $\SS(\overline{\Omega})$ for an open set $\Omega\subset \R^n$:
\[
\SS(\overline \Omega):=\left\{ u\in C^\infty(\overline \Omega)\colon \sup_{\bs x \in \overline \Omega} |\bs x|^m|\partial^\alpha u(\bs x)| <\infty \text{ for all } m,\alpha \in \N_0^{n}\right\}.
\]

\begin{defs}$ $ \label{Defns of function spaces}

\begin{list}{(\roman{count})}{\usecounter{count}}
\item
Let $\dot H^{\frac12}(\R)$ be the completion of $\SS(\R)$
with respect to the norm
$$
\|u\|_{\dot H^{\frac12}(\R)}:=\left(\int_\R |k| |\hat u(k)|^2 \, dk\right)^{\frac12}.
$$
\item
Let $\dot{H}^{-\frac12}(\R)$ be the completion of $\overline{\SS}(\R)=\{u\in \SS(\R)\colon \hat u(0)=0\}$ 
with respect to the norm
$$
\|u\|_{\dot H^{-\frac12}(\R)}:=\left(\int_\R |k|^{-1} |\hat u(k)|^2\, dk\right)^{\frac12}.
$$
\item
Let $\dot H^1(\Omega)$ be the completion of $\SS(\overline{\Omega})$
with respect to the norm 
$$
\|u\|_{\dot H^1(\Omega)}:=\left(\int_{\Omega}|\nabla u|^2\, dx\, dy\right)^{\frac12}.
$$
\end{list}
\end{defs}

The following result is classical and the proof is therefore omitted (we do however present a proof of a similar result for the upper domain later; see Proposition \ref{upper trace}).

\begin{props} \hspace{1in}
\label{lower trace}
\begin{list}{(\roman{count})}{\usecounter{count}}
\item
The trace map $u \mapsto u|_{y=\underline{\eta}}$ defines a continuous map $\dot H^1(\underline{\Sigma}(\underline{\eta})) \to \dot H^{\frac12}(\R)$ and has
a continuous right inverse $\dot H^{\frac12}(\R)\to \dot H^1(\underline{\Sigma}(\underline{\eta}))$.
\item
The space $\dot{H}^{-\frac12}(\R)$ can be identified with $(\dot{H}^{\frac12}(\R))^\prime$.
\end{list}
\end{props}

\begin{defs} \label{Definition of DNO}
For $\eta \in W^{1,\infty}(\R)$, the bounded linear operator $\underline{G}(\eta)\colon \dot H^{\frac12}(\R)\to \dot H^{-\frac12}(\R)$ is defined by 
\[
\langle \underline{G}(\underline{\eta})\underline{\Phi}_1, \underline{\Phi}_2\rangle =\int_{\underline{\Sigma}(\underline{\eta})} \nabla \underline{\phi}_1 \cdot \nabla \underline{\phi}_2 \,dx\, dy, \qquad \underline{\Phi}_1, \underline{\Phi}_2 \in \dot H^{\frac12}(\R),
\]
where $\langle \cdot\,, \cdot \rangle$ denotes the $\dot H^{-\frac12}(\R) \times \dot H^{\frac12}(\R)$ pairing and
$\underline{\phi}_j$, $j=1,2$, is the unique function in $\dot H^1(\underline{\Sigma}(\underline{\eta}))$ such that $\underline{\phi}_j|_{y=\underline{\eta}}=\underline{\Phi}_j$ and
$$
\int_{\underline{\Sigma}(\underline{\eta})} \nabla \underline{\phi}_j\cdot  \nabla \underline{\psi}\, dx\, dy =0
$$
for all $\underline{\psi} \in \dot H^1(\underline{\Sigma}(\underline{\eta}))$ with $\underline{\psi}|_{y=\underline{\eta}}=0$.
\end{defs}

Using Proposition \ref{lower trace} and the definition of $\underline{G}(\underline{\eta})$, we find that 
\begin{equation}
\label{lower G coercive}
\langle \underline{G}(\underline{\eta}) \underline{\Phi}, \underline{\Phi}\rangle
=\int_{\underline{\Sigma}(\underline{\eta})} |\nabla \underline{\phi}|^2 \,dx\, dy
\ge c\|\underline{\Phi}\|_{\dot H^{\frac12}(\R)}^2
\end{equation}
for some constant $c>0$ which depends on $\|\underline{\eta}\|_{W^{1, \infty}(\R)}$. 
From this we immediately obtain the following result.

\begin{lemma}
\label{lower G invertible}
The Dirichlet-Neumann operator  $\underline{G}(\underline{\eta})\colon \dot H^{\frac12}(\R) \to \dot H^{-\frac12}(\R)$ is an isomorphism for each $\underline{\eta}\in W^{1,\infty}(\R)$. 
\end{lemma}

\begin{defs}
For $\underline{\eta}\in W^{1,\infty}(\R)$, the Neumann-Dirichlet operator $\underline{N}(\underline{\eta}) \colon  \dot H^{-\frac12}(\R)\to \dot H^{\frac12}(\R)$ is defined as the inverse of $\underline{G}(\underline{\eta})$.
\end{defs}

\subsubsection{Upper fluid}

We next discuss the same questions for the upper fluid. Here we have the additional difficulty that both boundaries are free.
Choose $h_0 \in (0,1)$. In order to prevent the boundaries from intersecting, we consider the class
\[
 W:=\{\bs \eta=(\underline{\eta}, \overline{\eta}) \in W^{1,\infty}(\R)\colon 1+\inf(\overline{\eta}-\underline{\eta})>h_0\}
\]
of surface and interface profiles.

\begin{defs}$ $
\label{definition of function spaces}

\begin{list}{(\roman{count})}{\usecounter{count}}
\item
Let $H_\star^{\frac12}(\R)$ be the completion of $\SS(\R)$
with respect to the norm
$$
\|u\|_{H_\star^{\frac12}(\R)}:=\left( \int_{\R} (1+k^2)^{-\frac{1}{2}}k^2|\hat{u}|^2\dk\right)^{\frac12},
$$
\item
Let $H_\star^{-\frac12}(\R)$ be the completion of $\overline{\SS}(\R)$ 
with respect to the norm
$$
\|u\|_{H_\star^{-\frac12}(\R)}:= \left(\int_{\R} (1+k^2)^{\frac{1}{2}}k^{-2}|\hat{u}|^2\dk\right)^{\frac12}.
$$
\item
Let $X$ be the Hilbert space
\[
\{ \overline{\bs\Phi}=(\overline{\Phi}_i, \overline{\Phi}_s)\in (H_\star^{\frac12}(\R))^2
\colon \overline{\Phi}_s-\overline{\Phi}_i \in H^{\frac12}(\R)\}
\]
equipped with the inner product
\[
\langle \overline{\bs \Phi}_1,\overline{\bs \Phi}_2\rangle_X=
\langle \overline{\bs \Phi}_1,\overline{\bs \Phi}_2\rangle_{(H_\star^{\frac12}(\R))^2}+\langle\overline{\Phi}_{1,s}-\overline{\Phi}_{1,i},\overline{\Phi}_{2,s}-\overline{\Phi}_{2,i}\rangle_{H^{\frac12}(\R)}.
\]
\item
Let $Y$ be the Hilbert space
\[
\{\overline{\bs\Psi}=(\overline{\Psi}_i, \overline{\Psi}_s) \in (H^{-\frac12}({\mathbb R}))^2\colon 
\overline{\Psi}_s+ \overline{\Psi}_i \in H_\star^{-\frac12}(\R)\}
\]
equipped with the inner product
\[
\langle \overline{\bs \Psi}_1,\overline{\bs \Psi}_2\rangle_Y=\langle \overline{\bs \Psi}_1,\overline{\bs \Psi}_2\rangle_{(H^{-\frac12}(\R))^2}+\langle\overline{\Psi}_{1,s}+\overline{\Psi}_{1,i},\overline{\Psi}_{2,s}+\overline{\Psi}_{2,i}\rangle_{H_\star^{-\frac12}(\R)}.
\]
\end{list}
\end{defs}

Note that we have the inclusions 
\[
H^{\frac12}(\R) \subset \dot H^{\frac12}(\R)\subset H_\star^{\frac12}(\R)
\quad 
\text{and} 
\quad 
H_\star^{-\frac12}(\R) \subset \dot H^{-\frac12}(\R)\subset H^{-\frac12}(\R).
\]
The reason for introducing the space $X$ is that it is the natural trace space associated with $\dot H^1(\overline{\Sigma}(\bs \eta))$.
Since this is not completely standard, we include a proof.

\begin{props}
\label{upper trace}
Fix $\bs \eta\in W$. The trace map $u \mapsto 
(u|_{y=\underline{\eta}}, u|_{y=1+\overline{\eta}})$
defines a continuous map
$\dot H^1(\overline{\Sigma}(\bs \eta))\to X$ with a continuous right inverse $X\to \dot H^1(\overline{\Sigma}(\bs \eta))$.
\end{props}

\begin{proof}
We flatten the domain using the transformation $(x,y)\mapsto (x, y^\prime(x,y))$, where
\[
y^\prime(x,y) =\frac{y-\underline{\eta}(x)}{1+\overline{\eta}(x)-\underline{\eta}(x)}.
\]
This maps the domain $\overline{\Sigma}(\bs \eta)$ onto the strip $\overline \Sigma_0=\{(x,y)\in \R^2 \colon 0<y<1\}$, 
and $\dot H^1(\overline{\Sigma}(\bs\eta))$ to  $\dot H^1(\overline{\Sigma}_0)$. 
Letting $\chi \in C_0^\infty(\R)$ be a cut-off function with support in $[-1/2,1/2]$, we find that
\[
\frac{\mathrm{d}}{\mathrm{d}y}(\chi(y) |\hat \phi(k, y)|^2)= \chi'(y)|\hat \phi(k, y)|^2 +
2\chi(y)\re (\hat \phi(k, y)\overline{\hat \phi_y(k,y)}), \qquad \phi \in C_0^\infty(\overline{\overline{\Sigma}_0})
\]
and hence
\begin{align*}
\int_{\R} \langle k\rangle^{-1}|k|^2 |\hat \phi(k, 0)|^2\dk &\le 
\int_{\Sigma_0} (|k|^2 |\hat \phi(k, y)|^2+2|k| |\hat \phi(k,y)| |\hat \phi_y(k,y)|)\dk \dy\\
&\le  2\int_{\Sigma_0} (|k|^2 |\hat \phi(k, y)|^2+|\hat \phi_y(k,y)|^2)\dk \dy\\
&=2\|\phi\|_{\dot H^1(\overline{\Sigma}_0)}^2.
\end{align*}
Moreover, 
\[
 \int_{\R} (\phi(x,1)-\phi(x,0))^2\dx =
 \int_{\R} \left(\int_{0}^{1} \phi_y\dy \right)^2 \dx
\le \|\phi\|_{\dot H^1(\overline{\Sigma}_0)}^2.
\]
It follows that $\|\phi(\cdot, 1)-\phi(\cdot, 0)\|_{H^{\frac12}(\R)}\le  c\|\phi\|_{\dot H^1(\overline{\Sigma}_0)}$, and hence 
that $\|(\phi|_{y=0}, \phi|_{y=1})\|_X\le c\|\phi\|_{\dot H^1(\overline{\Sigma}_0)}$. The continuity of the trace map 
now follows by a density argument.

Conversely, given $(\overline{\Phi}_i,\overline{\Phi}_s)$ we formally define 
$u\in \dot H^1(\overline{\Sigma}_0)$ by
\[
\widehat{u}=\frac{\sinh(k(1-y))}{\sinh (k)}\widehat{\overline{\Phi}_i}
+\frac{\sinh(ky)}{\sinh (k)}\widehat{\overline{\Phi}_s}.
\]
This means that $u$ is the element of $\dot H^1(\overline{\Sigma}_0)$ whose partial derivatives have Fourier transforms
\begin{align*}
\FF[u_x](k,y)&=\mathrm{i}k\frac{\sinh(k(1-y))}{\sinh (k)}\widehat{\overline{\Phi}_i}
+\mathrm{i}k\frac{\sinh(ky)}{\sinh (k)}\widehat{\overline{\Phi}_s},\\
\FF[u_y](k,y)&=-k\frac{\cosh(k(1-y))}{\sinh (k)}\widehat{\overline{\Phi}_i}
+k\frac{\cosh(ky)}{\sinh (k)}\widehat{\overline{\Phi}_s}.
\end{align*}
It is clear from these formulas that the map $X\ni (\overline{\Phi}_i, \overline{\Phi}_s) 
\mapsto  u\in  \dot H^1(\overline{\Sigma}_0)$ is continuous.
\end{proof}

Note that  $(H_\star^{\frac12}(\R))'$ can be identified with $H_\star^{-\frac12}(\R)$.
A straightforward argument shows that the dual space of $X$ is $Y$.

\begin{props}
\label{Y=X'}
The space $Y$ can be identified with the dual of $X$ using the duality pairing
\[
\langle \overline{\bs \Psi}, \overline{\bs \Phi}\rangle_{Y\times X}=
\langle \overline{\Psi}_s, \overline{\Phi}_s-\overline{\Phi}_i \rangle_{H^{-\frac12}(\R)\times H^{\frac12}(\R)}
+
\langle \overline{\Psi}_s+ \overline{\Psi}_i, \overline{\Phi}_i \rangle_{H_\star^{-\frac12}(\R)\times H_\star^{\frac12}(\R)}.
\]
\end{props}

\begin{defs}
\label{def:Dirichlet-Neumann}
For $\bs\eta\in W$, the bounded linear operator $\overline{G}(\bs \eta)\colon X\to Y$ is 
defined by
$$
\langle \overline{G}(\bs \eta)\overline{\bs \Phi}_1, \overline{\bs \Phi}_2 \rangle=\int_{\overline{\Sigma}(\bs \eta)} \nabla \overline{\phi}_1\cdot \nabla \overline{\phi}_2 \dx \dy, \qquad \overline{\bs \Phi}_1, \overline{\bs \Phi}_2 \in X,
$$
where $\langle \cdot, \cdot\rangle$ denotes the $Y\times X$ pairing and  $\overline{\phi}_j$, $j=1,2$, is the unique function in $\dot H^1(\overline{\Sigma}(\bs \eta))$ such that  $\overline{\phi}_j|_{y=1+\overline{\eta}}=\overline{\Phi}_{j,s}$, $\overline{\phi}_j|_{y=\underline{\eta}}=\overline{\Phi}_{j,i}$ and
 \[
 \int_{\overline{\Sigma}(\bs \eta)} \nabla \overline{\phi}_j \cdot \nabla \overline{\psi} \dx\dy=0
\]
for all $\overline{\psi} \in \dot H^1(\overline{\Sigma}(\bs \eta))$ with $\overline{\psi}|_{y=\underline{\eta}}=0$ and $\overline{\psi}|_{y=1+\overline{\eta}}=0$
\end{defs}

As in the case of the lower fluid, we obtain that
\begin{equation}
\label{upper G coercive}
\langle \overline{G}(\bs \eta) \overline{\bs \Phi}, \overline{\bs \Phi}\rangle
\ge c\|\overline{\bs \Phi}\|_{X}^2,
\end{equation}
for some constant $c>0$ which depends on $h_0$ and $\|\bs \eta\|_{W^{1,\infty}(\R)}$, and the following consequence.

\begin{lemma}
The operator  $\overline{G}(\bs \eta)\colon X \to Y$ is an isomorphism for each $\bs \eta \in W$. 
\end{lemma}

\begin{defs}
For $\bs \eta \in W$, the Neumann-Dirichlet operator $\overline{N}(\bs \eta) \colon Y \to X$ is defined as the inverse of $\overline{G}(\bs \eta)$.
\end{defs}

\subsubsection{Further operators} \label{Further operators}

We now proceed with the rigorous definition of the operators 
$G(\bs \eta)$, $N(\bs \eta)$ and $K(\bs \eta)$.
Recall that the definition of $\overline{G}(\bs \eta)$ involves various combinations of 
the components of $\overline{G}(\bs \eta)$ (cf.~\eqref{definition of GG}).
We can formally write
\[
\overline{G}(\bs \eta)\overline{\bs \Phi}=
\begin{pmatrix}
\overline{G}_{11}(\bs \eta) & \overline{G}_{12}(\bs \eta)\\
\overline{G}_{21}(\bs \eta) & \overline{G}_{22}(\bs \eta)
\end{pmatrix}
\begin{pmatrix}
\overline{\Phi}_i\\
\overline{\Phi}_s
\end{pmatrix},
\] 
but since the definition of the function space $X$ involves the 
condition $\overline{\Phi}_s-\overline{\Phi}_i \in H^{\frac12}(\R)$ which couples the components 
$\overline{\Phi}_s$ and $\overline{\Phi}_i$, the definition of the components $\overline{G}_{ij}$ requires some care.
Note however that $(H^{\frac12}(\R))^2 \subset X$, so that the components
$\overline{G}_{ij}(\bs \eta)$  define bounded operators $H^{\frac12}(\R)\to H^{-\frac12}(\R)$. 
The components $\overline{N}_{ij}(\bs \eta)$ can similarly be defined by considering the subspace
$(H_\star^{-\frac12}(\R))^2\subset Y$.

\begin{props}
The operators $\overline{G}_{ij}(\bs \eta)\colon H^{\frac12}(\R)\to H^{-\frac12}(\R)$
and $\overline{N}_{ij}(\bs \eta)\colon H_\star^{-\frac12}(\R)\to H_\star^{\frac12}(\R)$ are continuous.
\end{props}

\begin{lemma}
For each $\bs \eta \in W$, the operator $B(\bs \eta):=\overline{G}_{11}(\bs \eta)+ \rho \underline{G}(\underline{\eta})$ is an isomorphism 
$H^{\frac12}(\R)\to H^{-\frac12}(\R)$.
\end{lemma}

\begin{proof}
Recall that $\underline{G}(\underline{\eta})\colon \dot H^{\frac12}(\R)\to \dot H^{-\frac12}(\R)$ is an 
isomorphism, with 
\[
\langle \underline{G}(\underline{\eta})\underline{\Phi}, \underline{\Phi}\rangle_{\dot H^{-\frac12}(\R) \times 
\dot H^{\frac12}(\R)}
\ge c\|\underline{\Phi}\|_{\dot H^{\frac12}(\R)}^2
\] 
(cf.~\eqref{lower G coercive} and Lemma \ref{lower G invertible})
for some $c>0$.
On the other hand,
\[
\langle \overline{G}_{11}(\bs \eta)\overline{\Phi}_i, \overline{\Phi}_i\rangle_{H^{-\frac12}(\R) \times 
H^{\frac12}(\R)}
\ge c\|\overline{\Phi}_i\|_{H^{\frac12}(\R)}^2,
\]
by Definition \ref{definition of function spaces} and \eqref{upper G coercive} with $\overline{\Phi}_s=0$.
It follows that
\[
\langle B(\bs \eta) \Phi, \Phi\rangle_{H^{-\frac12}(\R)\times H^{\frac12}(\R)}
\ge c \|\Phi\|_{H^{\frac12}(\R)}^2 
\]
and hence $B(\bs \eta) \colon H^{\frac12}(\R)\to H^{-\frac12}(\R)$ is an isomorphism.
\end{proof}

Recall that we formally defined the operator $G(\bs \eta)$ by
\begin{align*}
G(\bs \eta):=\begin{pmatrix}\underline{G}(\underline{\eta})
B(\bs \eta)^{-1}\overline{G}_{11}(\bs \eta)
&-\underline{G}(\underline{\eta})B(\bs \eta)^{-1}\overline{G}_{12}(\bs \eta)\\
-\overline{G}_{21}(\bs \eta)B(\bs \eta)^{-1}\underline{G}(\underline{\eta})&\tfrac{1}{\rho}\overline{G}_{22}(\bs \eta)-\tfrac{1}{\rho}\overline{G}_{21}(\bs \eta)B(\bs \eta)^{-1}\overline{G}_{12}(\bs \eta)\end{pmatrix}.
\end{align*}
It is not difficult to see that $G(\bs \eta)$ is bounded $(H^{\frac12}(\R))^2 \to (H^{-\frac12}(\R))^2$.
However, we need to extend it to a larger space in order to define $K(\bs \eta)$. We record some lemmas which enable us to do this.

\begin{lemma}
The operators $\overline{G}_{11}(\bs \eta) B^{-1}(\bs \eta)$ and $\overline{G}_{21}(\bs \eta) B^{-1}(\bs \eta)$ are  bounded on $\dot H^{-\frac12}(\R)$.
\end{lemma}

\begin{proof}
The first part follows from the facts that
$\overline{G}_{11}(\bs \eta)B^{-1}(\bs \eta)=I-\rho \underline{G}(\underline{\eta})B^{-1}(\bs \eta)$ as well as $\underline{G}(\underline{\eta})B^{-1}(\bs \eta) 
\in \mathcal{L}(H^{-\frac12}(\R), \dot H^{-\frac12}(\R))$. The second part now follows from the 
fact that $\overline{G}_{11}(\bs \eta)+\overline{G}_{21}(\bs \eta)\in \mathcal{L}(H^{\frac12}(\R), H_\star^{-\frac12}(\R))$.
\end{proof}

\begin{corollary}
\label{cor:extension by duality}
The maps  $B^{-1}(\bs \eta) \overline{G}_{11}(\bs \eta)$ and $B^{-1}(\bs \eta) \overline{G}_{12}(\bs \eta)$ extend 
to bounded mappings on $\dot H^{\frac12}(\R)$ by duality.
\end{corollary}

Recall that $\bs \xi$ is defined in terms of $\underline{\Phi}$ and $\overline{\bs \Phi}$ through 
\eqref{eq:canonical variables}. Conversely, we can formally recover 
$\underline{\Phi}$ and $\overline{\bs \Phi}$ from $\bs \xi$ under the assumption 
\eqref{eq:continuity in normal direction} through \eqref{eq:expressing Phi in terms of xi}.
We now investigate these relations in more detail. We begin defining appropriate function spaces 
for $\bs \xi$ and $G(\bs \eta) \bs \xi$.

\begin{defs}$ $
\label{definition of function spaces 2}

\begin{list}{(\roman{count})}{\usecounter{count}}
\item
Let $\tilde X$ be the Hilbert space 
\[
\{\bs \xi=(\underline{\xi}, \overline{\xi})\in (H_\star^{\frac12}(\R))^2  \colon 
\underline{\xi}+\overline{\xi} \in \dot H^{\frac12}(\R)\}\]
equipped with the inner product
\[
\langle \bs \xi_1,\bs \xi_2\rangle_{\tilde X}=\langle \bs \xi_1, \bs \xi_2\rangle_{(H_\star^{\frac12}(\R))^2}+\langle\overline{\xi}_{1}+\underline{\xi}_{1},\overline{\xi}_{2}+\underline{\xi}_{2}\rangle_{\dot H^{\frac12}(\R)}.
\]
\item
Let $\tilde Y$ be the Hilbert space
\[
\{\bs \zeta=(\underline{\zeta}, \overline{\zeta}) \in (\dot H^{-\frac12}(\R))^2 \colon \overline{\zeta}-\underline{\zeta}\in H_\star^{-\frac12}(\R)\}
\]
equipped with the inner product
\[
\langle \bs \zeta_1, \bs \zeta_2\rangle_{\tilde Y}=\langle \bs \zeta_1,\bs \zeta_2 \rangle_{(\dot H^{-\frac12}(\R))^2}+\langle\overline{\zeta}_{1}-\underline{\zeta}_{1},\overline{\zeta}_{2}-\underline{\zeta}_{2}\rangle_{H_\star^{-\frac12}(\R)}.
\]
\end{list}
\end{defs}

An argument similar to Proposition \ref{Y=X'} shows that $\tilde Y$ is dual to $\tilde X$.

\begin{lemma}
\label{xi to Phi}
The equations \eqref{eq:expressing Phi in terms of xi} define bounded linear operators
\[\bs \xi \mapsto \underline{\Phi}\colon \tilde X \to \dot H^{\frac12}(\R)\] and 
\[\bs\xi \mapsto \overline{\bs \Phi}\colon \tilde X \to X\] with
\[
\underline{G}(\underline{\eta}) \underline{\Phi}=-(\overline{G}_{11}(\bs \eta)\overline{\Phi}_i+ \overline{G}_{12}(\bs \eta)\overline{\Phi}_s).
\]
\end{lemma}

\begin{proof}
By definition we have that 
\begin{align*}
\underline{\Phi}&=B^{-1}(\bs \eta)\overline{G}_{11}(\bs \eta)\underline{\xi}-B^{-1}(\bs \eta)\overline{G}_{12}(\bs \eta)\overline{\xi}\\
&=B^{-1}(\bs \eta)\overline{G}_{11}(\bs \eta)(\underline{\xi}+\overline{\xi})
-B^{-1}(\bs \eta)(\overline{G}_{11}(\bs \eta)+\overline{G}_{12}(\bs \eta))\overline{\xi}.
\end{align*}
This defines an element of $\dot H^{\frac12}(\R)$ by Corollary \ref{cor:extension by duality} 
and the continuity of 
\[
(\overline{G}_{11}(\bs \eta)+\overline{G}_{12}(\bs \eta)) \colon H_\star^{\frac12}(\R) \to 
H^{-\frac12}(\R).
\]
Similarly,  
\begin{align*}
\overline{\Phi}_i&=-B^{-1}(\bs \eta)\underline{G}(\underline{\eta})\underline{\xi}-\frac1{\rho} B^{-1}(\bs \eta)\overline{G}_{12}(\bs \eta)\overline{\xi}\\
&=-B^{-1}(\bs \eta)\underline{G}(\underline{\eta})(\underline{\xi}+\overline{\xi})+
\frac{1}{\rho} B^{-1}(\bs \eta)(\rho\underline{G}(\underline{\eta})+\overline{G}_{11}(\bs \eta))\overline{\xi}-\frac1{\rho} B^{-1}(\bs \eta)(\overline{G}_{11}(\bs \eta)+\overline{G}_{12}(\bs \eta))\overline{\xi}\\
&=
-B^{-1}(\bs \eta)\underline{G}(\underline{\eta})(\underline{\xi}+\overline{\xi})+
\frac{1}{\rho} \overline{\xi}-\frac1{\rho} B^{-1}(\bs \eta)(\overline{G}_{11}(\bs \eta)+\overline{G}_{12}(\bs \eta))\overline{\xi}
\in H_\star^{\frac12}(\R).
\end{align*}
It is obvious that  $\overline{\Phi}_s=\frac1{\rho} \overline{\xi} \in H_\star^{\frac12}(\R)$.
To see that $\overline{\bs \Phi}\in X$, we note that 
\begin{align*}
\overline{\Phi}_s-\overline{\Phi}_i
=B^{-1}(\bs \eta)\underline{G}(\underline{\eta})(\underline{\xi}+\overline{\xi}) +\frac1{\rho} B^{-1}(\bs \eta)(\overline{G}_{11}(\bs \eta)+\overline{G}_{12}(\bs \eta))\overline{\xi}\in H^{\frac12}(\R).
\end{align*}
It is easily seen that all of the involved operators are bounded. The final formula follows by straightforward algebraic manipulations.
\end{proof}

\pagebreak

\begin{props}
The operator $G(\bs \eta)$ is bounded $\tilde X 
\to \tilde Y$. 
\end{props}

\begin{proof}
Assume that $\bs \xi\in \tilde X$.
A direct computation then shows that
\begin{align*}
\underline{G}(\underline{\eta})B(\bs \eta)^{-1}\overline{G}_{11}(\bs \eta)\underline{\xi}-\underline{G}(\underline{\eta})B(\bs \eta)^{-1}\overline{G}_{12}(\bs \eta)\overline{\xi}
&=
\underline{G}(\underline{\eta})(B(\bs \eta)^{-1}
\overline{G}_{11}(\bs \eta)\underline{\xi}
-B(\bs \eta)^{-1}\overline{G}_{12}(\bs \eta)\overline{\xi})\\
&
=\underline{G}(\underline{\eta})\underline{\Phi}\in \dot H^{-\frac12}(\R),
\end{align*}
where we have used Lemma \ref{xi to Phi}.
Similarly,
\begin{align*}
&-\overline{G}_{21}(\bs \eta)B(\bs \eta)^{-1}\underline{G}(\underline{\eta})\underline{\xi}
+
(\tfrac{1}{\rho}\overline{G}_{22}(\bs \eta)-\tfrac{1}{\rho}\overline{G}_{21}(\bs \eta)B(\bs \eta)^{-1}\overline{G}_{12}(\bs \eta))\overline{\xi}\\
&\qquad=\overline{G}_{21}(\bs \eta)(-B(\bs \eta)^{-1}\underline{G}(\underline{\eta})\underline{\xi}
-\tfrac{1}{\rho} B(\bs \eta)^{-1}\overline{G}_{12}(\bs \eta)\overline{\xi})
+\overline{G}_{22}(\bs \eta)\overline{\Phi}_s\\
&\qquad=\overline{G}_{21}(\bs \eta)\overline{\Phi}_i+\overline{G}_{22}(\bs \eta)\overline{\Phi}_s
\in H^{-\frac12}(\R).
\end{align*}
We have to show that the last expression is actually an element of $\dot H^{-\frac12}(\R)$. 
To see this, we note that
\begin{equation}
\label{useful inclusion}
\overline{G}_{11}(\bs \eta)\overline{\Phi}_i+\overline{G}_{12}(\bs \eta)\overline{\Phi}_s+
\overline{G}_{21}(\bs \eta)\overline{\Phi}_i+\overline{G}_{22}(\bs \eta)\overline{\Phi}_s
\in H_\star^{-\frac12}(\R),
\end{equation}
by the definition of $Y$ and Definition \ref{def:Dirichlet-Neumann}.
On the other hand 
\[
\overline{G}_{11}(\bs \eta)\overline{\Phi}_i+\overline{G}_{12}(\bs \eta)\overline{\Phi}_s
=-\underline{G}(\underline{\eta})\underline{\Phi} \in \dot H^{-\frac12}(\R).
\]
This shows that $\overline{G}_{21}(\bs \eta)\overline{\Phi}_i+\overline{G}_{22}(\bs \eta)\overline{\Phi}_s\in 
\dot H^{-\frac12}(\R)$. The fact that
$(G(\bs \eta)\bs \xi)_2-(G(\bs \eta)\bs \xi)_1\in H_\star^{-\frac12}(\R)$ follows from \eqref{useful inclusion}.
The boundedness of $G(\bs \eta)$ follows from the above formulas and Lemma \ref{xi to Phi}.
\end{proof}

Define
\[
N(\bs \eta):=
\begin{pmatrix}
\rho\overline{N}_{11}(\bs \eta)  +\underline{N}(\underline{\eta}) & 
-\rho \overline{N}_{12}(\bs \eta)\
\\
-\rho \overline{N}_{21}(\bs \eta) & \rho\overline{N}_{22}(\bs \eta)
\end{pmatrix}.
\]

\begin{props}
\label{G is invertible}
$G(\bs \eta)\colon \tilde X
\to \tilde Y$ is invertible with 
\[
G(\bs \eta)^{-1}=
N(\bs \eta).
\]
\end{props}

\begin{proof}
We begin by showing that $N(\bs \eta)$ defines an operator $\tilde Y\to \tilde X$.
Indeed, if $\bs \zeta=(\underline{\zeta}, \overline{\zeta})\in \tilde Y$, then $(\underline{\zeta}, -\overline{\zeta})\in Y$ whence
\[
\overline{N}_{11}(\bs \eta)\underline{\zeta} -\overline{N}_{12}(\bs \eta)\overline{\zeta}
=(\overline{N}(\eta) (\underline{\zeta}, -\overline{\zeta}))_1\in H_\star^{\frac12}(\R),
\]
\[
-\overline{N}_{21}(\bs \eta)\underline{\zeta} +\overline{N}_{12}(\bs \eta)\overline{\zeta}
=-(\overline{N}(\eta) (\underline{\zeta}, -\overline{\zeta}))_2\in H_\star^{\frac12}(\R),
\]
and
\[
\underline{N}(\underline{\eta})\underline{\zeta} \in \dot H^{\frac12}(\R).
\]
Finally,
\begin{align*}
&(\overline{N}_{11}(\bs \eta)\underline{\zeta}-\overline{N}_{12}(\bs \eta)\overline{\zeta}) 
+(-\overline{N}_{21}(\bs \eta)\underline{\zeta} +\overline{N}_{12}(\bs \eta)\overline{\zeta})) \\
&= (\overline{N}_{11}(\bs \eta)\underline{\zeta} -\overline{N}_{12}(\bs \eta)\overline{\zeta}) 
-(\overline{N}_{21}(\bs \eta)\underline{\zeta} -\overline{N}_{12}(\bs \eta)\overline{\zeta})) 
 \in H^{\frac12}(\R),
\end{align*}
which implies that
\[
\rho((\overline{N}_{11}(\bs \eta)\underline{\zeta}-\overline{N}_{12}(\bs \eta)\overline{\zeta}) 
+(-\overline{N}_{21}(\bs \eta)\underline{\zeta} +\overline{N}_{12}(\bs \eta)\overline{\zeta}))
+\underline{N}(\underline{\eta})\overline{\zeta} \in \dot H^{\frac12}(\R).
\]
The equation $G(\bs \eta) \bs\xi=\bs \zeta\in \tilde Y$ can equivalently be written 
\begin{align*}
\overline{G}_{11}(\bs \eta)\overline{\Phi}_i+\overline{G}_{12}(\bs \eta)\overline{\Phi}_s&=-\underline{\zeta},\\
\overline{G}_{21}(\bs \eta)\overline{\Phi}_i+\overline{G}_{22}(\bs \eta)\overline{\Phi}_s&=\overline{\zeta},
\end{align*}
with the unique solution 
\[
\overline{\bs \Phi}=\overline{N}(\bs \eta)(-\underline{\zeta}, \overline{\zeta}).
\]
On the other hand, we also have $\underline{G}(\underline{\eta})\underline{\Phi}
=\underline{\zeta}$, so that 
$\underline{\Phi}=\underline{N}(\underline{\eta}) \underline{\zeta}$. 
It follows that $G(\bs \eta) \bs \xi=\bs \zeta$ if and only if 
\[
\underline{\xi}=\underline{\Phi}-\rho \overline{\Phi}_i
=(\rho\overline{N}_{11}(\bs \eta)  +\underline{N}(\underline{\eta}))\underline{\zeta} 
-\rho \overline{N}_{12}(\bs \eta)\overline{\zeta}
\]
and
\[
\overline{\xi}=\rho \overline{\Phi}_s
=-\rho\overline{N}_{21}(\bs \eta)\underline{\zeta} 
+\rho \overline{N}_{22}(\bs \eta)\overline{\zeta}.
\]
Hence $N(\bs \eta)$ is the inverse of $G(\bs \eta)$.
\end{proof}

We are now finally ready to discuss the operator $K(\bs \eta)$.

\begin{defs}$ $
\label{definition of function spaces 3}

\begin{list}{(\roman{count})}{\usecounter{count}}
\item
Let $\breve X$ be the Hilbert space 
\[
\{\bs \xi=(\underline{\xi}, \overline{\xi})\in (\dot H^{\frac12}(\R))^2  \colon 
\underline{\xi}-\overline{\xi} \in H^{\frac12}(\R)\}
\]
equipped with the inner product
\[
\langle \bs \xi_1,\bs \xi_2\rangle_{\check X}=\langle \bs \xi_1, \bs \xi_2\rangle_{(\dot H^{\frac12}(\R))^2}+\langle\overline{\xi}_{1}-\underline{\xi}_{1},\overline{\xi}_{2}-\underline{\xi}_{2}\rangle_{H^{\frac12}(\R)}.
\]
\item
Let $\breve Y$ be the Hilbert space
\[
\{\bs \zeta=(\underline{\zeta}, \overline{\zeta}) \in (H^{-\frac12}(\R))^2 \colon \overline{\zeta}+\underline{\zeta}\in \dot H^{-\frac12}(\R)\}
\]
equipped with the inner product
\[
\langle \bs \zeta_1,\bs \zeta_2\rangle_{\check X}=\langle \bs \zeta_1, \bs \zeta_2\rangle_{(H^{-\frac12}(\R))^2}+\langle\overline{\zeta}_{1}+\underline{\zeta}_{1},\overline{\zeta}_{2}+\underline{\zeta}_{2}\rangle_{\dot H^{-\frac12}(\R)}.
\]
\end{list}
\end{defs}

Note, $\breve X$ and $\breve Y$ are each other's duals and that $(H^{\frac12}(\R))^2 \hookrightarrow \breve X$, $\breve Y\hookrightarrow (H^{-\frac12}(\R))^2$.

\begin{props}
\label{K isomorphism}
The formula $K(\bs \eta)=-\partial_x N(\bs \eta) \partial_x$ defines an isomorphism
$\breve X\to \breve Y$
with
\[
\langle K(\bs\eta)\bs\xi, \bs\xi\rangle_{\breve Y\times \breve X} \ge c\|\bs \xi\|_{\breve X}^2. 
\]
\end{props}

\begin{proof}
The fact that $K(\bs \eta)$ is a bounded operator from $\breve Y$ to $\breve X$ follows by noting that 
$\partial_x$ is an isomorphism from $\breve X$ to $\tilde Y$ and from $\tilde X$ to $\breve Y$.
The lower bound follows by setting $\tilde{\bs \xi}=(\underline{\xi}, -\overline{\xi})$ and noting that
\begin{align*}
&\langle K(\bs \eta)\bs \xi, \bs\xi\rangle_{\breve Y \times \breve X} \\
&\qquad =\rho \langle \overline{N}(\bs \eta)\partial_x \tilde{\bs \xi}, \partial_x\tilde{\bs \xi}\rangle_{(H_\star^{\frac12}(\R))^2 \times 
(H_\star^{-\frac12}(\R))^2}+
\langle \underline{N}(\underline{\eta}) \partial_x \underline{\xi}, \partial_x \underline{\xi}
\rangle_{\dot H^{\frac12}(\R)\times \dot H^{-\frac12}(\R)}\\
&\qquad \ge c(\|\partial_x  \tilde{\bs \xi}\|_Y^2+\|\partial_x \underline{\xi}\|_{\dot H^{-\frac12}(\R)}^2)\\
&\qquad \ge c(\|\bs\xi\|_{(\dot H^{\frac12}(\R))^2}^2 
+\|\underline{\xi}-\overline{\xi}\|_{H^{\frac12}(\R)}^2).
\end{align*}
This also shows that $K(\bs \eta)$ is an isomorphism.
\end{proof}

It will be useful to write $K(\bs \eta)$ in the form
\[
K(\bs \eta) =
\begin{pmatrix}
\underline{K}(\underline{\eta}) & 
0
\\
0 & 0
\end{pmatrix}
+ \rho
\overline{K}(\bs \eta),
\]
where
$\underline{K}(\underline{\eta}):=-\partial_x \underline{N}(\underline{\eta})\partial_x$ and 
\[
\overline{K}(\bs\eta):=
- \partial_x\begin{pmatrix}
\overline{N}_{11}(\bs \eta)  & 
- \overline{N}_{12}(\bs \eta)\
\\
- \overline{N}_{21}(\bs \eta) & \overline{N}_{22}(\bs \eta)
\end{pmatrix}\partial_x.
\]

\subsection{Analyticity and higher regularity}

In this section we discuss the analyticity of the operators $\overline{K}(\bs \eta)$ and $\underline{K}(\underline\eta)$ as functions of $\bs \eta$ and $\underline\eta$ respectively.
We also discuss how they act on higher order Sobolev spaces, assuming that $\bs \eta$ is sufficiently regular.
We begin by considering the second operator using the method explained by Groves \& Wahl\'en \cite{GrovesWahlen15}.
First note that $\underline{N}(\underline{\eta})$ is 
given by 
\[
\underline{N}(\underline{\eta})\underline{\Psi}=\underline{\phi}|_{y=\underline{\eta}},
\]
where $\underline{\phi} \in \dot H^1(\underline{\Sigma}(\underline{\eta}))$ is a weak solution of the boundary-value problem
\begin{equation}
\label{eq:Neumann BVP lower}
\begin{aligned}
&\Delta \underline{\phi}=0, &&y<\underline{\eta},\\
&\underline{\phi}_y - \underline{\eta}^\prime\underline{\phi}_x = \underline{\Psi}, &&y=\underline{\eta},
\end{aligned}
\end{equation}
that is,
\[
\int_{\underline{\Sigma}(\underline{\eta)}} \nabla \underline{\phi} \cdot \nabla \underline{\psi} \dx \dy=\int_{\R} \underline{\Psi} \psi|_{y=\underline{\eta}} \dx
\]
for all $\psi\in  \dot H^1(\underline{\Sigma}(\underline{\eta}))$.

We study the dependence of $\underline{N}(\underline{\eta})$ by transforming  this boundary-value problem into an equivalent one in the fixed domain $\underline{\Sigma}_0 :=\underline{\Sigma}(0)$.
For this purpose we make the change of variables
\[
y'=y-\underline{\eta}
\]
and define
$$\underline{F}(x,y^\prime)=(x, y^\prime+\underline{\eta}(x))$$
and
\[
\underline{u}(x,y')=\underline{\phi}(\underline{F}(x,y'))
\]
This change of variable transforms the boundary-value problem \eqref{eq:Neumann BVP lower} into
\begin{align*}
&\nabla\cdot((I+\underline{Q})\nabla \underline{u})=0 &&  y <0,\\
& (I+\underline{Q})\nabla \underline{u} \cdot (0, 1) = \underline{\Psi}, && y=0,
\end{align*}
where
\[
\underline{Q}=
\begin{pmatrix}
0 & - \underline{\eta}_x \\ 
-  \underline{\eta}_x &  \underline{\eta}_x^2
\end{pmatrix}
\]
and the primes have been dropped for notational simplicity. The weak form of this problem is
\[
\int_{\underline{\Sigma}} (1+\underline{Q})\nabla \underline{u} \cdot \nabla w \dx \dy=\int_{\R} \underline{\Psi}\,w|_{y=0} \dx
\]
for all $w\in \dot H^1(\underline{\Sigma}_0)$.
Fix $\underline{\eta}_0$ and write $\tilde{\underline{\eta}}=\underline{\eta}-\underline{\eta}_0$ and
\[
\underline{Q}(x,y)=\sum_{n=0}^\infty \underline{Q}^n(x,y), \qquad 
\underline{Q}^n =\tilde m^n(\tilde \eta^{(n)}),
\]
where
$\tilde m^n \in \LL_\mathrm{s}^n(W^{1, \infty}(\R), (L^\infty(\R^2))^{2\times 2})$ (in which $\LL_\mathrm{s}^n$ denotes the set of bounded, symmetric, $n$-linear operators).
We seek a solution of the above boundary value problem of the form
\[
\underline{u}(x,y)=\sum_{n=0}^\infty \underline{u}^n(x,y), \qquad \underline{u}^n=m^n(\{\tilde{\underline{\eta}}\}^{(n)}),
\]
where $m^n \in \LL_\mathrm{s}^n(W^{1, \infty}(\R), \dot H^1(\underline{\Sigma}_0))$ is 
linear in $\underline{\Psi}$.
Substituting this Ansatz into the equations, one finds that
\begin{eqnarray*}
& & \parbox{90mm}{$\nabla\cdot((I+\underline{Q}^0)\nabla \underline{u}^0)=0$}y<0, \\
& & \parbox{90mm}{$(I+\underline{Q}^0)\nabla \underline{u}^0 \cdot (0, 1) =\underline{\Psi},$}y=0, \\
\end{eqnarray*}
and
\begin{eqnarray*}
& & \parbox{90mm}{$\nabla\cdot((I+\underline{Q}^0)\nabla u^n)=\nabla \cdot \underline{F}^n$}y <0, \\
& & \parbox{90mm}{$(I+\underline{Q}^0)\nabla \underline{u}^n \cdot (0, 1) = \underline{F}^n\cdot (0,1),$}y=0, \\
\end{eqnarray*}
where
\[
\underline{F}^n=-\sum_{k=1}^n \underline{Q}^k\nabla \underline{u}^{n-k}.
\]
These equations can be solved recursively. Estimating the solutions we obtain the following result.

\begin{lemma}
The Dirichlet-Neumann operator $\underline{G}(\cdot) \colon W^{1,\infty}(\R) \to \LL(\dot H^{\frac12}(\R),\dot H^{-\frac12}(\R))$
and the Neumann-Dirichlet operator  $\underline{N}(\cdot) \colon W^{1,\infty}(\R) \to \LL(\dot H^{-\frac12}(\R),\dot H^{\frac12}(\R))$
are analytic.
\end{lemma}

The upper domain can be treated in a similar way.
Set
\[
y'=\frac{y-\underline{\eta}}{1+\overline{\eta}-\underline{\eta}},
\]
so that
\begin{equation}
\label{definition of f}
y=y'+f(x,y'), \qquad f(x,y')=\underline{\eta}+(\overline{\eta}-\underline{\eta})y',
\end{equation}
and let $\overline{F}(x,y')=(x, y'+f(x,y'))$.
The function $\overline{u}(x,y)=\overline{\phi}(F(x,y))$ then solves the boundary value problem 
\begin{eqnarray*}
& & \parbox{90mm}{$\nabla\cdot((I+\overline{Q})\nabla \overline{u})=0$}0 < y <1, 
\label{BVP for NDO 1 u}  \\
& & \parbox{90mm}{$(I+\overline{Q})\nabla \overline{u} \cdot (0, 1) = \overline{\Phi}_s,$}y=1, \label{BVP for NDO 2 u} \\
& & \parbox{90mm}{$(I+\overline{Q})\nabla \overline{u} \cdot (0,-1)=\overline{\Phi}_i,$}y=0, \label{BVP for NDO 3 u} 
\end{eqnarray*}
where
\[
\overline{Q}=
\begin{pmatrix}
f_y&  -f_x\\
 -f_x &
 \frac{-f_y+f_x^2}{1+f_y}
\end{pmatrix}.
\]
Proceeding as before, we obtain the following result.

\begin{lemma}
The Dirichlet-Neumann operator $\overline{G}(\cdot) \colon W \to \LL(X, Y)$
and the Neumann-Dirichlet operator  $\overline{N}(\cdot) \colon W \to \LL(Y,X)$
are analytic.
\end{lemma}

The next theorem follows from the above lemmas and the definitions of the involved operators.

\begin{them}
The operators $K(\cdot), \overline{K}(\cdot)\colon W \to \LL((H^{\frac12}(\R))^2, (H^{-\frac12}(\R))^2)$ and\\ $\underline K(\cdot)
\colon W^{1,\infty}(\R)\to  \LL((H^{\frac12}(\R)), (H^{-\frac12}(\R)))$ are analytic.
\end{them}

It is also possible to study these operators in spaces with more regularity. A straightforward modification of the techniques in \cite{GrovesWahlen15, Lannes13} results in the following theorem.

\begin{them}
\label{higher regularity of K operators}
The operators $K(\cdot), \overline{K}(\cdot) \colon (H^{s+\frac32}(\R))^2\cap W \to \LL((H^{s+\frac32}(\R))^2, (H^{s+\frac12}(\R))^2)$ and  $\underline{K}(\cdot)
\colon H^{s+\frac32}(\R) \to  \LL((H^{s+\frac32}(\R)), (H^{s+\frac12}(\R)))$ are analytic for each $s>0$.
\end{them}

\subsection{Variational functionals}

In this section we study the functionals
\begin{equation}
\label{definition of L}
\underline{\LL}(\underline{\eta})=\frac12 \int_{\R} \underline{\eta} \underline{K}(\eta)\underline{\eta} \dx,
\end{equation}
\begin{equation}
\label{definition of upper L}
\overline{\LL}(\bs \eta)=\frac{1}2 \int_{\R} \bs\eta \overline{K}(\bs \eta)\bs \eta \dx.
\end{equation}
and
\begin{align}
\label{definition of K}
\mathcal{K}(\bs \eta)&= \int_{\mathbb R} \left\{\frac{(1-\rho)}{2}\underline{\eta}^2 +\frac{\rho}{2}\overline{\eta}^2+ \underline{\beta}\sqrt{1+\underline{\eta}_x^2}-\underline{\beta}+ \rho \overline{\beta}\sqrt{1+\overline{\eta}_x^2}-\rho\overline{\beta}
\right\} \dx.
\end{align}

As a direct consequence of the above formulas and Theorem \ref{higher regularity of K operators} we obtain the following result.

\begin{lemma}
Equations \eqref{definition of L}, \eqref{definition of upper L} and \eqref{definition of K} define analytic functionals $\underline{\mathcal{L}} \colon  H^{s+\frac32}(\R)\to \R$, $\overline{\mathcal{L}} \colon (H^{s+\frac32}(\R))^2\cap W \to \R$ and 
$\KK  \colon  (H^{s+\frac32}(\R))^2\to \R$ for each $s>0$.
\end{lemma}

In particular, this lemma implies that $\mathcal{J}_\mu\in C^\infty(U\setminus\{0\}, \R)$, where $U=B_M(0)\subset H^2(\R)$ with $M$ sufficiently small.

We turn now to the construction of the gradients $\underline{\mathcal{L}}'(\underline{\eta})$ and
 $\overline{\mathcal{L}}'(\bs\eta)$ in $L^2(\R)$ and $(L^2(\R))^2$ respectively. The following results are proved using the methods explained in \cite[Section 2.2.1]{GrovesWahlen15}.

\begin{lemma}
The gradient $\underline{\mathcal{L}}'(\underline{\eta})$ in $L^2(\R)$ exists for each 
$\underline{\eta}\in   H^{s+\frac32}(\R)\to \R$ and is given by the formula
\[
\underline{\mathcal{L}}'(\underline{\eta})=\frac{1}{2}\Big( (1+\underline{\eta}_x^2)u_y^2-u_x^2\Big)-u_x\Big|_{y=0}.
\]
This formula defines an analytic function $\underline{\mathcal{L}}'\colon  H^{s+\frac32}(\R)\to 
H^{s+\frac12}(\R)$.
\end{lemma}

We also find that
\begin{align*}
 \underline{\LL}_2'(\underline{\eta})&=\underline{K}^0\underline{\eta}=\mathcal{F}^{-1}(|k|\widehat{\underline{\eta}}),\\
\underline{\LL}'_3(\underline{\eta})
&=-\frac{1}{2}(\underline{K}^0\underline{\eta})^2-\frac{1}{2}\underline{\eta}_x^2-\underline{\eta} \underline{\eta}_{xx}-\underline{K}^0(\underline{\eta} \underline{K}^0\underline{\eta}),\\
\underline{\LL}'_4(\underline{\eta})
&=\underline{\eta}\underline{\eta}_{xx}\underline{K}^0\underline{\eta}+\underline{K}^0\underline{\eta}\, \underline{K}^0(\underline{\eta} \underline{K}^0\underline{\eta})+\underline{K}^2(\underline{\eta})\underline{\eta},
\end{align*}
and
\begin{align}
\label{lower L2}
\underline{\LL}_2(\underline{\eta})&=\frac12 \int_{\R} \underline{\eta} \underline{K}^0 \underline{\eta} \, dx,\\ 
\label{lower L3}
 \underline{\LL}_3(\underline{\eta})
&=\frac12\int_{\R} \Big((\underline{\eta}_x)^2-(\underline{K}^0\underline{\eta})^2\Big)\,\underline{\eta} \, dx,\\
\label{lower L4}
 \underline{\LL}_4(\underline{\eta})
&=\frac12\int_{\R} \Big(\underline{\eta}^2 \underline{\eta}_{xx}K^ 0\underline{\eta}+\underline{\eta} K^ 0\underline{\eta} \, \underline{K}^0(\underline{\eta} \underline{K}^0 \underline{\eta})\Big) \, dx
\end{align}
where $\underline{\LL}_k(\underline{\eta})$, $k=2, 3, \ldots$, are the terms in the power series expansion of $\underline{\LL}(\underline{\eta})$ at the origin.

\begin{lemma}
The gradient $\overline{\mathcal{L}}'(\bs \eta)$ in $(L^2(\R))^2$ exists for each 
$\bs \eta\in  (H^{s+\frac32}(\R))^2\cap W \to \R$ and is given by the formula
\[
\overline{\mathcal{L}}'(\bs \eta)=\left(\frac12\Big(u_x^2-\frac{1+\overline{\eta}_x^2}{(1+\overline{\eta}-\underline{\eta})^2} u_y^2\Big)\Big|_{y=0}+u_x|_{y=0}, -\frac12\Big(u_x^2-\frac{1+\underline{\eta}_x^2}{(1+\overline{\eta}-\underline{\eta})^2} u_y^2\Big)\Big|_{y=1}-u_x|_{y=1}\right).
\]
This formula defines an analytic function $\overline{\mathcal{L}}'\colon  (H^{s+\frac32}(\R))^2\cap W \to 
(H^{s+\frac12}(\R))^2$.
\end{lemma}

The first few terms in the power series expansion of $\overline{\LL}'$ are given by
\begin{align*}
\overline{\LL}_2'(\bs \eta)&=\overline{K}^0\bs \eta=
\FF^{-1}[\overline F(k)\hat{\bs \eta}],\\
\overline{\LL}_3'(\bs \eta)&=
\begin{pmatrix}
\frac12 \underline{\eta}_x^2+\underline{\eta}_{xx}\underline{\eta}
+\frac12(\overline{K}_{11}^0\underline{\eta}+\overline{K}_{12}^0 \overline{\eta})^2
+\overline{K}_{11}^0 (\underline{\eta} (\overline{K}_{11}^0\underline{\eta}+\overline{K}_{12}^0 \overline{\eta}))
-\overline{K}_{21}^0 (\overline{\eta} (\overline{K}_{21}^0\underline{\eta}+\overline{K}_{22}^0 \overline{\eta}))
 \\
-\frac12 \overline{\eta}_x^2-\overline{\eta}_{xx}\overline{\eta}
-\frac12(\overline{K}_{21}^0\underline{\eta}+\overline{K}_{22}^0 \overline{\eta})^2
-\overline{K}_{22}^0 (\overline{\eta} (\overline{K}_{21}^0\underline{\eta}+\overline{K}_{22}^0 \overline{\eta}))
+\overline{K}_{12}^0 (\underline{\eta} (\overline{K}_{11}^0\underline{\eta}+\overline{K}_{12}^0 \overline{\eta}))
\end{pmatrix},\\
\overline{\LL}'_4(\bs \eta)
&=  
\begin{pmatrix}
(\overline{K}_{11}^0\underline{\eta}+\overline{K}_{12}^0 \overline{\eta})
(\underline{\eta}_{xx}\underline{\eta}
+\overline{K}_{11}^0 (\underline{\eta} (\overline{K}_{11}^0\underline{\eta}+\overline{K}_{12}^0 \overline{\eta}))
-\overline{K}_{21}^0 (\overline{\eta} (\overline{K}_{21}^0\underline{\eta}+\overline{K}_{22}^0 \overline{\eta})))
\\
(\overline{K}_{21}^0\underline{\eta}+\overline{K}_{22}^0 \overline{\eta})
(\overline{\eta}_{xx}\overline{\eta}
+\overline{K}_{22}^0 (\overline{\eta} (\overline{K}_{21}^0\underline{\eta}+\overline{K}_{22}^0 \overline{\eta}))
-\overline{K}_{12}^0 (\underline{\eta} (\overline{K}_{11}^0\underline{\eta}+\overline{K}_{12}^0 \overline{\eta})))
\end{pmatrix} +
\overline{K}^2(\bs \eta) \bs \eta,
\end{align*}
where
\begin{equation}
\label{definition of F}
\overline F(k)=
\begin{pmatrix}
|k| \coth |k|& -\frac{|k|}{\sinh |k|}\\ 
-\frac{|k|}{\sinh |k|}&  |k| \coth |k| 
\end{pmatrix}. 
\end{equation}
and
\begin{align}
\label{upper L2}
\overline{\LL}_2(\bs \eta)&=\frac12\int_{\R} \bs \eta \overline{K}^0 \bs \eta \, dx,\\
\label{upper L3}
\overline{\LL}_3(\bs \eta)&=
\frac12 \int_{\R} \left\{-
\Big((\underline{\eta}_x)^2-
(\overline{K}_{11}^0\underline{\eta}+\overline{K}_{12}^0 \overline{\eta})^2\Big)\underline{\eta}
+
\Big((\overline{\eta}_x)^2-
(\overline{K}_{21}^0\underline{\eta}+\overline{K}_{22}^0 \overline{\eta})^2\Big)\overline{\eta}\right\}
\dx,\\
\nonumber
 \overline{\LL}_4(\bs \eta)
&=
\frac12 \int_{\R} \left\{
(\overline{K}_{11}^0\underline{\eta}+\overline{K}_{12}^0 \overline{\eta})\underline{\eta}_{xx}\underline{\eta}^2
+(\overline{K}_{21}^0\underline{\eta}+\overline{K}_{22}^0 \overline{\eta})\overline{\eta}_{xx}\overline{\eta}^2
\right.
\\
\nonumber
&\qquad  \qquad
+\underline{\eta}
(\overline{K}_{11}^0\underline{\eta}+\overline{K}_{12}^0 \overline{\eta}) 
\overline{K}_{11}^0 (\underline{\eta} (\overline{K}_{11}^0\underline{\eta}+\overline{K}_{12}^0 \overline{\eta}))
-2
\underline{\eta}(\overline{K}_{11}^0\underline{\eta}+\overline{K}_{12}^0 \overline{\eta})
\overline{K}_{21}^0 (\overline{\eta} (\overline{K}_{21}^0\underline{\eta}+\overline{K}_{22}^0 \overline{\eta}))
\\
\label{upper L4}
&\qquad \qquad \left.
+
\overline{\eta}(\overline{K}_{21}^0\underline{\eta}+\overline{K}_{22}^0 \overline{\eta}) \overline{K}_{22}^0 (\overline{\eta} (\overline{K}_{21}^0\underline{\eta}+\overline{K}_{22}^0 \overline{\eta}))
\right\}\dx.
\end{align}

The first terms in the power series expansion of $\KK(\bs \eta)$ will also be needed later (the corresponding gradients are readily obtained from these expressions):
\begin{equation}
\label{K2 and K4}
\begin{aligned}
\mathcal{K}_2(\bs\eta)&=\frac12 \int_{\mathbb R} \left\{(1-\rho)\underline{\eta}^2 +\rho\overline{\eta}^2 + 
\rho \overline{\beta} \overline{\eta}_x^2 + \underline{\beta} \underline{\eta}_x^2
\right\} \dx,\\
\mathcal{K}_4(\bs\eta)&=-\frac1{8} \int_{\R} \left\{\rho\overline{\beta} \overline{\eta}_x^4 +\underline{\beta}\underline{\eta}_x^4\right\}\dx.
\end{aligned}
\end{equation}
Note in particular that
\begin{equation}
\label{K2 and L2 formulas}
\KK_2(\bs \eta)=\frac12 \int_{\R} P(k) \hat{\bs \eta}\cdot \hat{\bs \eta} \, dk
\quad \text{and} \quad 
\LL_2(\bs \eta)=\frac12 \int_{\R} F(k) \hat{\bs \eta} \cdot \hat{\bs \eta} \, dk,
\end{equation}
where $P(k)$ and $F(k)$ are given by equation \eqref{formulas for P and F}.

We end this section by recording some useful inequalities.
\begin{props}
\label{useful inequalities}
The estimates
\[
\KK(\bs \eta) \ge c\|\bs \eta\|_1^2, \qquad c\|\bs \eta\|_{\breve X}^2 \le \, \LL_2(\bs \eta), \LL(\bs \eta)\,
\le c\|\bs \eta\|_{\breve X}^2
\]
hold for each $\bs \eta \in U$.
\end{props}

\begin{proof}
The first estimate is immediate from the form of $\KK(\bs \eta)$. The estimates for $\LL(\bs \eta)$ follow 
directly from Proposition \ref{K isomorphism}, while those for $\LL_2(\bs \eta)$ follow from \eqref{K2 and L2 formulas}.
\end{proof}

\section{Existence and stability}

This section contains the main results of the paper. We begin by proving that the functional $\JJ_\mu$ has a minimiser in $U\sm \{0\}$. This is done by using concentration-compactness and penalisation methods as in \cite{Buffoni04a, Buffoni05, Buffoni09, BuffoniGrovesSunWahlen13, GrovesWahlen11, GrovesWahlen15} and we refer to those papers for the details of some of the proofs. The outcome is the following result.

\begin{them} \label{Key minimisation theorem}
Suppose that Assumptions \ref{assumption 1} and \ref{assumption 2} hold.
\hspace{1in}
\begin{list}{(\roman{count})}{\usecounter{count}}
\item
The set $C_\mu$ of minimisers of $\JJ_\mu$ over $U \sm \{0\}$ is non-empty.
\item
Suppose that $\{\bs \eta_n\}$ is a minimising sequence for $\JJ_\mu$ on $U\sm\{0\}$ which satisfies
\begin{equation}
\sup_{n\in{\mathbb N}} \|\bs \eta_n\|_2 < M. \label{Sup less than M}
\end{equation}
There exists a sequence $\{x_n\} \subset {\mathbb R}$ with the property that
a subsequence of\linebreak $\{\bs\eta_n(x_n+\cdot)\}$ converges
in $(H^r(\mathbb R))^2$, $0 \leq r < 2$ to a function $\bs\eta \in C_\mu$.
\end{list}
\end{them}

The first statement of the theorem is a consequence of the second statement, once the existence of a minimising sequence satisfying \eqref{Sup less than M} has been established. 
The existence of such a sequence can be proved using a penalisation method \cite{Buffoni04a, BuffoniGrovesSunWahlen13, GrovesWahlen11, GrovesWahlen15}. A key part of the proof is the existence of a suitable `test function'  $\bs \eta_\star$ which satisfies the inequality
\[
\JJ_\mu(\bs\eta_\star)<2\nu_0\mu-c\mu^3.
\]
This implies in particular that {\em any} minimising sequence $\{\bs\eta_n\}$ satisfies this property for $n$ sufficiently large.
We construct such a test function in the appendix. Once the existence of the test function has been proved, the remaining steps in the construction of the special minimising sequence satisfying \eqref{Sup less than M} are similar to \cite{Buffoni04a, BuffoniGrovesSunWahlen13, GrovesWahlen11, GrovesWahlen15}, to which we refer for further details. In fact, this special minimising sequence satisfies further properties which will be used below (note that a general minimising sequence satisfies the weaker estimate $\|\bs \eta_n\|_1^2 \le c\mu$ by Proposition \ref{useful inequalities}).

\begin{them} \label{Special MS theorem}
Suppose that Assumptions \ref{assumption 1} and \ref{assumption 2} hold.
There exists a minimising sequence $\{\tilde{\bs\eta}_n\}$ for $\JJ_\mu$ over $U\sm\{0\}$
with the properties that $\|\tilde{\bs\eta}_n\|_2^2 \leq c \mu$ and $\JJ_\mu(\tilde{\bs \eta}_n) < 2\nu_0\mu -c\mu^3$ for each $n \in {\mathbb N}$, 
and
$\lim_{n \rightarrow \infty}\|\JJ_\mu^\prime(\tilde{\bs\eta}_n)\|_0=0$.
\end{them}

The second statement of Theorem \ref{Key minimisation theorem} is proved by applying the concentration-compactness principle (Lions \cite{Lions84a,Lions84b}) (a form suitable for the present situation can be found in \cite[Theorem 3.7]{GrovesWahlen15}) to a minimising sequence satisfying \eqref{Sup less than M}. The key step is to show that the function
\[
\mu\mapsto I_\mu:=\inf \{ \JJ_\mu(\bs \eta) \colon \bs \eta \in U\sm \{0\}\}
\]
is strictly sub-additive.

\begin{them}\label{thm:subadd}
Suppose that Assumptions \ref{assumption 1} and \ref{assumption 2} hold.
The number $I_\mu$ has the strict sub-additivity property
\begin{align*}
I_{\mu_1+\mu_2}<I_{\mu_1}+I_{\mu_2},\quad0<\mu_1,\mu_2,\mu_1+\mu_2<\mu_0.
\end{align*}
\end{them}

Theorem \ref{thm:subadd} is obtained using a careful analysis of the special minimising sequence from Theorem \ref{Special MS theorem}, which is postponed to the end of this section.

The next step is to relate the above result to our original problem of
finding minimisers of $\mathcal{E}(\bs \eta,\bs \xi)$ subject to the constraint $\mathcal{I}(\bs \eta,\bs\xi) = 2\mu$,
where $\mathcal{E}$ and $\mathcal{I}$ are defined in equations \eqn{definition of H} and
\eqn{Definition of I}. The following result is obtained using the argument explained in \cite[Section 5.1]{GrovesWahlen15}.

\begin{them} \label{Result for constrained minimisation} 
Suppose that Assumptions \ref{assumption 1} and \ref{assumption 2} hold.
\hspace{1cm}
\begin{list}{(\roman{count})}{\usecounter{count}}
\item
The set $D_\mu$ of minimisers of $\mathcal{E}$ over the set
$$S_\mu=\{(\bs \eta,\bs\xi) \in U \times \tilde X\colon \mathcal{I}(\bs\eta,\bs\xi)=2\mu\}$$
is non-empty.
\item
Suppose that $\{(\bs\eta_n,\bs\xi_n)\} \subset S_\mu$ is a minimising sequence for $\mathcal{E}$ with
the property that
$$
\sup_{n\in{\mathbb N}} \|\bs \eta_n\|_2 < M.
$$
There exists a sequence $\{x_n\} \subset \mathbb R$ with the property that
a subsequence of\linebreak $\{\bs\eta_n(x_n+\cdot),\bs\xi_n(x_n+\cdot)\}$
converges  in
$(H^r(\mathbb R))^2 \times \tilde X$, $0 \leq r < 2$ to a function in $D_\mu$.
\end{list}
\end{them}

We obtain a stability result as a corollary of
Theorem \ref{Result for constrained minimisation} using the argument given by Buffoni 
\cite[Theorem 19]{Buffoni04a}. Recall that the usual informal interpretation of the statement that a set $V$ of solutions to an initial-value problem is `stable' is that a solution which begins close to a solution in $V$ remains close to a solution in $V$ at all subsequent times. The precise meaning of a solution in the theorem below is irrelevant, as long 
as  it conserves the functionals $\mathcal{E}$ and $\mathcal{I}$ over some time interval $[0,T]$ with $T>0$.

\begin{them} \label{CES}
Suppose that Assumptions \ref{assumption 1}  and \ref{assumption 2} hold and
that $(\bs \eta,\bs \xi)\colon [0,T] \rightarrow U \times \tilde X$
has the properties that
$$\mathcal{E}(\bs\eta(t),\bs\xi(t)) = \mathcal{E}(\bs\eta(0),\bs\xi(0)),\ \mathcal{I}(\bs\eta(t),\bs\xi(t))=\mathcal{I}(\bs\eta(0),\bs\xi(0)), \qquad t \in [0,T]$$
and 
$$\sup_{t \in [0,T]} \|\bs\eta(t)\|_2 < M.$$
Choose $r \in [0,2)$, and let `$\dist$' denote the distance in $(H^r(\R))^2 \times \tilde X$.
For each
$\varepsilon>0$ there exists $\delta>0$ such that
$$\dist((\bs\eta(0),\bs\xi(0)), D_\mu) < \delta \quad \Rightarrow \quad
\dist((\bs\eta(t),\bs\xi(t)), D_\mu)<\varepsilon$$
for $t\in[0,T]$.
\end{them}

This result is a statement of
the \emph{conditional, energetic stability of the set $D_\mu$}. Here
\emph{energetic} refers to the fact that the
distance in the statement of stability is measured in the `energy space'
$(H^r(\R))^2 \times \tilde X$, while \emph{conditional}
alludes to the well-posedness issue. At present there is no global well-posedness theory for interfacial water waves (although there is a large and growing body of literature concerning well-posedness issues for water-wave problems in general).
The solution $t \mapsto (\bs\eta(t),\bs \xi(t))$ may exist in a smaller space over the interval $[0,T]$, at each instant of which it remains close
(in energy space) to a solution in $D_\mu$. Furthermore, Theorem \ref{CES}
is a statement of the stability of the \emph{set} of constrained minimisers $D_\mu$;
establishing the uniqueness of the constrained minimiser would imply
that $D_\mu$ consists of translations of a single solution, so that the statement
that $D_\mu$ is stable is equivalent to classical orbital stability of this unique
solution.

Finally, we can also confirm the heuristic argument given in Section \ref{Heuristics}.
\begin{them} \label{Convergence}
Under assumptions \ref{assumption 1} and \ref{assumption 2}, the set $D_\mu$ of minimisers of ${\mathcal E}$ over $S_\mu$ satisfies
$$\sup_{(\bs\eta,\bs \xi) \in D_\mu} \inf_{\omega \in [0,2\pi], x \in {\mathbb R}} \|{\bs \phi}_{\bs \eta}-e^{i\omega}\phi_\mathrm{NLS}(\cdot+x)\bs v_0\|_1 \rightarrow 0$$
as $\mu \downarrow 0$, where we write
${\bs \eta}_1^+(x) = \tfrac{1}{2}\mu {\bs \phi}_{\bs \eta}(\mu x)e^{i k_0 x}$
and $\bs \eta_1^+={\mathcal F}^{-1}[\chi_{[k_0-\delta_0,k_0+\delta_0]}\hat{\bs \eta}]$ with $\delta_0 \in (0,\tfrac{1}{3}k_0)$.
Furthermore, the speed $\nu_\mu$ of the corresponding solitary
wave satisfies
\[
\nu_\mu = \nu_0 + 2(\nu_0F(k_0)\bs v_0\cdot \bs v_0)^{-1}\nu_\mathrm{NLS}\mu^2 + o(\mu^2)
\]
uniformly over $(\bs \eta,\bs \xi) \in D_\mu$.
\end{them}

Note in particular that since $\bs v_0=(1,-a)$ with $a>0$ (cf.~eq.~\eqref{formula for a}) the surface profile $\overline{\eta}$ is to leading order a scaled and inverted copy of the interface profile $\underline{\eta}$ (cf.~Figure \ref{setting}). The fact that we don't know if the minimiser is unique up to translations is reflected by the lack of control over $\omega$; for the model equation, the minimiser is in fact not unique up to translations (see Lemma \ref{Variational NLS}). Using dynamical systems methods (see e.g.~\cite{BarrandonIooss05}), we expect that one can prove the existence of two solutions corresponding to $\omega=0$ and $\omega=\pi$ above, but without any knowledge of stability. Since the proof of Theorem \ref{Convergence} follows \cite[Section 5.2]{GrovesWahlen15} closely, we shall omit it.

The goal of the rest of this section is to prove Theorem \ref{thm:subadd}, which follows directly from the strict sub-homogeneity  of $I_\mu$ (see Corollary \ref{cor:subhom}).
This property is established
by considering a `near minimiser'
of $\JJ_\mu$ over $U\sm\{0\}$, that is a function in $U \sm \{0\}$ with
$$\|\tilde{\bs \eta}\|_2^2 \leq c\mu, \quad \JJ_\mu(\tilde{\bs \eta}) < 2\nu_0\mu-c\mu^3, \quad \|\JJ_\mu^\prime(\tilde{\bs \eta})\|_0 \leq \mu^N,$$
for some $N\geq3$.
Hence we have $\LL(\tilde{\bs \eta}), \LL_2(\tilde{\bs \eta})>c\mu$ (by Proposition \ref{useful inequalities} and the inequality $\mu^2/\LL(\tilde{\bs\eta})<2\nu_0 \mu$) and can identify the dominant term in the `nonlinear' part 
\[
\MM_\mu(\tilde{\bs \eta}):=\JJ_\mu(\tilde{\bs \eta})-\frac{\mu^2}{\LL_2(\tilde{\bs\eta})}-\KK_2(\tilde{\bs\eta})
\]
of $\JJ_\mu(\tilde{\bs \eta})$.
The existence of near minimisers is a consequence of Theorem \ref{Special MS theorem}.
Note that we will work under Assumptions \ref{assumption 1} and \ref{assumption 2} throughout the rest of the section, without explicitly mentioning when they are needed. One of the main tools that we will use is the weighted norm
\begin{align*}
\nn \bs\eta \nn_\alpha^2:=\int_{\R} (1+\mu^{-4\alpha}(|k|-k_0)^4) |\hat{\bs\eta}(k)|^2 \dk
\end{align*}
and a splitting of $\bs\eta$ in view of the expected frequency distribution. In fact we split each $\bs\eta \in U$ into the sum of a function $\bs \eta_1$ with
spectrum near $k=\pm k_0$ and a function $\bs\eta_2$ whose spectrum is bounded away from these points.
To this end we write the equation
\begin{eqnarray*}
\JJ_\mu^\prime(\bs\eta)
& = & \KK_2^\prime(\bs\eta) + \KK_\mathrm{nl}^\prime(\bs\eta)
-\left(\frac{\mu}{\LL(\bs\eta)}\right)^{\!\!2}\LL_2^\prime(\bs\eta)
-\left(\frac{\mu}{\LL(\bs\eta)}\right)^{\!\!2} \LL_\mathrm{nl}^\prime(\eta)\\
& = & \KK_2^\prime(\bs\eta)  - \nu_0^2 \LL_2^\prime(\bs\eta) + \KK_\mathrm{nl}^\prime(\bs\eta) -\left(\left(\frac{\mu}{\LL(\bs\eta)}\right)^2 - \nu_0^2\right)\LL_2^\prime(\eta)
-\left(\frac{\mu}{\LL(\bs\eta)}\right)^{\!\!2} \LL_\mathrm{nl}^\prime(\bs\eta)
\end{eqnarray*}
in the form
\begin{eqnarray*}
g(k)\hat{\bs\eta}
= \FF\left[ \JJ_\mu^\prime(\bs\eta) - \KK_\mathrm{nl}^\prime(\bs\eta)
+
\left(\left(\frac{\mu}{\LL(\bs\eta)}\right)^2  - \nu_0^2\right)\LL_2^\prime(\bs\eta)
+\left(\frac{\mu}{\LL(\bs\eta)}\right)^{\!\!2} \LL_\mathrm{nl}^\prime(\bs\eta)\right],
\end{eqnarray*}
where $g(k)$ is given by \eqref{definition of g}.
We decompose it into two coupled equations by defining 
$\bs\eta_2 \in (H^2(\R))^2$ by the formula
\begin{eqnarray*}
\bs\eta_2 = \FF^{-1}\!\!\left[(1-\chi_S(k)) g(k)^{-1}  \FF\left[ \JJ_\mu^\prime(\bs\eta) - \KK_\mathrm{nl}^\prime(\bs\eta)
+
\left(\left(\frac{\mu}{\LL(\bs\eta)}\right)^2- \nu_0^2\right)\LL_2^\prime(\bs\eta)
+\left(\frac{\mu}{\LL(\bs\eta)}\right)^{\!\!2} \LL_\mathrm{nl}^\prime(\bs\eta)\right]\!\right]
\end{eqnarray*}
and $\bs\eta_1 \in (H^2(\R))^2$ by $\bs\eta_1=\bs\eta-\bs\eta_2$, so that $\hat{\bs\eta}_1$ has support in $S:=[-k_0-\delta_0,-k_0+\delta_0] \cup [k_0-\delta_0,k_0+\delta_0]$, where $\delta_0\in (0,k_0/3)$. Here we have used the fact that
\[
\bs f \mapsto \FF^{-1}\left[(1-\chi_S(k)) g(k)^{-1} \hat{\bs f}(k)\right]
\]
is a bounded linear operator $(L^2(\R))^2 \rightarrow (H^2(\R))^2$.

It will also be useful to express vectors $\bs w=(\underline w, \overline w)$ in the basis $\{\bs v_0, \bs v_0^\sharp\}$, where $\bs v_0$ is the zero eigenvector of the matrix $g(k_0)$ (see Section \ref{Heuristics}) and $\bs v_0^\sharp \not \parallel \bs v$.
The exact choice of the complementary vector $\bs v_0^\sharp$ is unimportant, but in order to simplify the notation later on we choose $\bs v_0^\sharp=(0,1)$. This implies that
\begin{align*}
\bs w=c_1 \bs v_0+c_2 \bs v_0^\sharp:=\bs w^{\bs v_0}+\bs w^{\bs v_0^\sharp},
\end{align*}
where $c_1=\underline w$ and $c_2=\overline w+a\underline w$.

The following propositions are used to estimate the special minimising sequence.
The proofs follow  
\cite[Section 4.1]{GrovesWahlen15} and are omitted.

\begin{props} \label{Estimates for nn}
\quad
\begin{list}{(\roman{count})}{\usecounter{count}}
\item
The estimates
$\|\bs\eta\|_{1,\infty} \leq c \mu^\frac{\alpha}{2}\nn \bs\eta \nn_\alpha$, $\|\underline{K}^0\bs\eta\|_\infty \leq c \mu^\frac{\alpha}{2}\nn \bs\eta \nn_\alpha$, $\|\overline{K}^0_{ij}\bs\eta\|_\infty \leq c \mu^\frac{\alpha}{2}\nn \bs\eta \nn_\alpha$
hold for each $\bs\eta \in H^2(\R)$.
\item
The estimates
$$\|\bs\eta^{\prime\prime}+k_0^2\bs\eta\|_0 \leq c \mu^{\alpha}\nn \bs\eta\nn_\alpha,$$
and
$$\| (K^0\bs\eta)^{(n)}\|_\infty \leq \mu^\frac{\alpha}{2}\nn \bs\eta \nn_\alpha, \quad n=0,1,2,\ldots,$$
hold for each $\bs\eta \in (H^2(\R))^2$ with $\supp \hat{\bs\eta} \subseteq S$.
\end{list}
\end{props}

\begin{props} \label{Speed estimate}
Any near minimiser $\tilde{\bs\eta}$ satisfies the inequalities
$$
\RR_1(\tilde{\bs\eta}) \leq \frac{\mu}{\LL(\tilde{\bs\eta})} -\nu_0 \leq \RR_2(\tilde{\bs\eta}),
$$
and
$$
\RR_1(\tilde{\bs\eta}) -\tilde{\MM}_\mu(\tilde{\bs\eta}) \leq \frac{\mu}{\LL_2(\tilde{\bs\eta})} -\nu_0 \leq \RR_2(\tilde{\bs\eta}) -\tilde{\MM}_\mu(\tilde{\bs\eta})
$$
where
\begin{eqnarray*}
\RR_1(\tilde{\bs\eta}) & = & -\frac{\langle \JJ_\mu^\prime(\tilde{\bs\eta}),\tilde{\bs\eta}\rangle}{4\mu}
+ \frac{1}{4\mu}\big(\langle \MM_\mu^\prime(\tilde{\bs\eta}),\tilde{\bs\eta}\rangle+4\mu\tilde{\MM}_\mu(\tilde{\bs\eta})\big), \nonumber \\
\RR_2(\tilde{\bs\eta}) & = & -\frac{\langle \JJ_\mu^\prime(\tilde{\bs\eta}),\tilde{\bs\eta}\rangle}{4\mu} 
+ \frac{1}{4\mu}\big(\langle \MM_\mu^\prime(\tilde{\bs\eta}),\tilde{\bs\eta}\rangle+4\mu\tilde{\MM}_\mu(\tilde{\bs\eta})\big)
-\frac{\MM_\mu(\tilde{\bs\eta})}{2\mu},
\end{eqnarray*}
and
$$
\tilde{\MM}_\mu(\tilde{\bs\eta}) = \frac{\mu}{\LL(\tilde{\bs\eta})}-\frac{\mu}{\LL_2(\tilde{\bs\eta})}.
$$
\end{props}

\begin{props} \label{Size of the functionals}
The estimates
\begin{eqnarray*}
 |\LL_3(\bs \eta)| 
& \leq & c\|\bs\eta\|_2^2 (\|\bs\eta\|_{1,\infty} + \|\bs\eta^{\prime\prime}+k_0^2\bs\eta\|_0), \\
\begin{Bmatrix} |\KK_4(\bs\eta)| \\ |\LL_4(\bs\eta)| \end{Bmatrix}
& \leq & c \|\bs\eta\|_2^2 (\|\bs\eta\|_{1,\infty} + \|\bs\eta^{\prime\prime}+k_0^2\bs\eta\|_0)^2, \\
\begin{Bmatrix}  |\KK_\mathrm{r}(\bs\eta)| \\ |\LL_\mathrm{r}(\bs\eta)| \end{Bmatrix}
& \leq & c\|\bs \eta\|_2^3 (\|\bs\eta\|_{1,\infty} + \|\bs\eta^{\prime\prime}+k_0^2\bs\eta\|_0)^2,
\end{eqnarray*}
hold for each $\bs\eta \in U$.
\end{props}

\begin{props} \label{Size of the gradients}
The estimates
\begin{eqnarray*}
\|\LL_3^\prime(\bs\eta)\|_0
& \leq & c\|\bs\eta\|_2 (\|\bs\eta\|_{1,\infty} + \|\bs\eta^{\prime\prime}+k_0^2\bs\eta\|_0 +  \|\underline{K}^0 \underline{\eta}\|_{\infty}+\|\overline{K}^0\bs\eta\|_{\infty}), \\
\begin{Bmatrix} \|\KK_4^\prime(\bs\eta)\|_0 \\ \|\LL_4^\prime(\bs\eta)\|_0 \end{Bmatrix}
& \leq & c\|\bs \eta\|_2 (\|\bs\eta\|_{1,\infty} + \|\bs\eta^{\prime\prime}+k_0^2\bs\eta\|_0 +\|\underline{K}^0 \underline{\eta}\|_{\infty}+\|\overline{K}^0\bs\eta\|_{\infty})^2, \\
\begin{Bmatrix}  \|\KK_\mathrm{r}^\prime(\bs\eta)\|_0 \\ \|\LL_\mathrm{r}^\prime(\bs \eta)\|_0 \end{Bmatrix}
& \leq & c\|\bs\eta\|_2^2  (\|\bs\eta\|_{1,\infty} + \|\bs\eta^{\prime\prime}+k_0^2\bs\eta\|_0)^2
\end{eqnarray*}
hold for each $\bs\eta \in U$.
\end{props}

It is also helpful to write
$$
\underline{\LL}_3^\prime(\underline{\eta}) = \underline{m}(\underline{\eta},\underline{\eta}),
\quad
\overline{\LL}_3^\prime(\bs\eta) = \overline{m}(\bs\eta,\bs\eta),\quad\LL_3'(\bs\eta)=m(\bs\eta,\bs\eta)=\begin{pmatrix} \underline{m}(\underline{\eta},\underline{\eta})\\0\end{pmatrix}+\rho\overline{m}(\bs\eta,\bs\eta),
$$
where $\underline{m}\in {\mathcal L}_\mathrm{s}^2(H^2(\R), L^2(\R))$ 
and $\overline{m}\in {\mathcal L}_\mathrm{s}^2((H^2(\R))^2, (L^2(\R))^2)$ are defined by
\begin{eqnarray*}
\underline{m}(\underline u_1,\underline u_2) & = & -\frac{1}{2}\underline{K}^0(\underline u_1 \underline{K}^0 \underline u_2) - \frac{1}{2}\underline{K}^0(\underline u_2 \underline{K}^0 \underline u_1)  \\
& & \qquad\mbox{} - \frac{1}{2}\underline{K}^0 \underline u_1 \underline{K}^0 \underline u_2 -\frac{1}{2}\underline u_{1x}\underline u_{2x} - \frac{1}{2}\underline u_{1xx}\underline u_2 - \frac{1}{2}\underline u_1\underline u_{2xx},\\
\overline{m}(\bs u_1,\bs u_2) 
&=& 
\begin{pmatrix}
\frac12\underline{u}_{1x} \underline{u}_{2x}+\frac12\underline{u}_{1xx}\underline{u}_2
+\frac12\underline{u}_{2xx}\underline{u}_1
+\frac12(\overline{K}_{11}^0\underline{u}_1+\overline{K}_{12}^0 \overline{u}_1)
(\overline{K}_{11}^0\underline{u}_2+\overline{K}_{12}^0 \overline{u}_2)
\\
-\frac12\overline{u}_{1x}\overline{u}_{2x}
-\frac12\overline{u}_{1xx}\overline{u}_2
-\frac12\overline{u}_{2xx}\overline{u}_1
-\frac12(\overline{K}_{21}^0\underline{u}_1+\overline{K}_{22}^0 \overline{u}_1)
(\overline{K}_{21}^0\underline{u}_2+\overline{K}_{22}^0 \overline{u}_2)
\end{pmatrix}\\
&& \qquad\mbox{} + \frac12
\begin{pmatrix}
\overline{K}_{11}^0 (\underline{u}_1 (\overline{K}_{11}^0\underline{u}_2+\overline{K}_{12}^0 \overline{u}_2))
+\overline{K}_{11}^0 (\underline{u}_2 (\overline{K}_{11}^0\underline{u}_1+\overline{K}_{12}^0 \overline{u}_1))
\\
-\overline{K}_{22}^0 (\overline{u}_1 (\overline{K}_{21}^0\underline{u}_2+\overline{K}_{22}^0 \overline{u}_2))
-\overline{K}_{22}^0 (\overline{u}_2 (\overline{K}_{21}^0\underline{u}_1+\overline{K}_{22}^0 \overline{u}_1))
\end{pmatrix}\\
& &\qquad\mbox{} + \frac12\begin{pmatrix}
-\overline{K}_{21}^0 (\overline{u}_1 (\overline{K}_{21}^0\underline{u}_2+\overline{K}_{22}^0 \overline{u}_2))
-\overline{K}_{21}^0 (\overline{u}_2 (\overline{K}_{21}^0\underline{u}_1+\overline{K}_{22}^0 \overline{u}_1))
 \\
\overline{K}_{12}^0 (\underline{u}_1 (\overline{K}_{11}^0\underline{u}_2+\overline{K}_{12}^0 \overline{u}_2))
+\overline{K}_{12}^0 (\underline{u}_2 (\overline{K}_{11}^0\underline{u}_1+\overline{K}_{12}^0 \overline{u}_1))\end{pmatrix},
\end{eqnarray*}
and similarly
$$ \underline{\LL}_3(\underline{\eta}) = \underline{n}(\underline{\eta},\underline{\eta},\underline{\eta}),
\quad  \overline{\LL}_3(\bs\eta) = \overline{n}(\bs\eta,\bs\eta,\bs\eta),\quad \LL_3(\bs\eta)=n(\bs\eta,\bs\eta,\bs\eta)=\underline{n}(\underline{\eta},\underline{\eta},\underline{\eta})+\rho\overline{n}(\bs\eta,\bs\eta,\bs\eta),$$
where $n_j \in {\mathcal L}_\mathrm{s}^3(H^2(\R), \R)$, $j=1,2,3$, are defined by
\begin{align*}
\underline{n}(\underline u_1,\underline u_2,\underline u_3) & =  \frac{1}{6}\int_{\R}\PP[\underline u_1^\prime \underline u_2^\prime \underline u_3] \dx
- \frac{1}{6}\int_{\R}\PP[(\underline{K}^0\underline u_1)(\underline{K}^0\underline u_2)\underline u_3] \dx\\
\overline{n}(\bs u_1,\bs u_2,\bs u_3)&=
\frac16 \int_{\R} 
\mathcal{P}[\overline{u}_1'\overline{u}_2'\overline{u}-\underline{u}_1'\underline{u}_2'\underline{u}_3]\dx\\
&\quad +\frac16\int_{\R} \mathcal{P}\left[
(\overline{K}_{11}^0\underline{u}_1+\overline{K}_{12}^0 \overline{u}_1)
(\overline{K}_{11}^0\underline{u}_2+\overline{K}_{12}^0 \overline{u}_2)\underline{u}_3\right]\dx\\
&\quad-\frac16\int_{\R} \mathcal{P}\left[
(\overline{K}_{21}^0\underline{u}_1+\overline{K}_{22}^0 \overline{u}_1)
(\overline{K}_{21}^0\underline{u}_2+\overline{K}_{22}^0 \overline{u}_2)
\overline{u}_3
\right]
\dx.
\end{align*}
The symbol $\PP[\cdot]$ denotes the sum of all distinct expressions resulting from permutations of the variables 
appearing in its argument.

Arguing as in \cite[Proposition 4.6 and Lemma 4.7]{GrovesWahlen15} we obtain the following estimates.

\begin{props} \label{Estimates for mj and nj}
The estimates
\begin{align*}
\|\underline{m}(\underline{\eta}_1,\underline u_2)\|_0 &\leq c(\|\underline{\eta}_1\|_{1,\infty} + \|\underline{\eta}_1^{\prime\prime}+k_0^2\underline{\eta}_1\|_0 + \|\underline{K}^0 \underline{\eta}_1\|_{1,\infty})
\|\underline u_2\|_2,\\
\|\overline{m}(\bs\eta_1,\bs u_2)\|_0 &\leq c(\|\bs\eta_1\|_{1,\infty} + \|\bs\eta_1^{\prime\prime}+k_0^2\bs\eta_1\|_0 + \|\underline{K}^0 \underline{\eta}_1\|_{1,\infty}+\|\overline{K}^0\bs\eta_1\|_{1,\infty})
\|\bs u_2\|_2,\\
|n(\bs\eta_1,\bs u_2,\bs u_3)| &\leq c(\|\bs\eta_1\|_{1,\infty} + \|\bs\eta_1^{\prime\prime}+k_0^2\bs\eta_1\|_0 + \|\underline{K}^0 \underline{\eta}_1\|_{1,\infty}+\|\overline{K}^0\bs\eta_1\|_{1,\infty}) \|\bs u_2\|_2 \|\bs u_3\|_2,
\end{align*}
hold for each $\bs\eta \in U$ and $\bs u_2$, $\bs u_3 \in (H^2(\R))^2$.
\end{props}

\begin{lemma} \label{Formulae for M}
The estimates
\begin{eqnarray*}
 \MM_\mu(\bs\eta) & = &   - \nu_0^2 \LL_3(\bs\eta)
+\KK_4(\bs\eta) - \nu_0^2 \LL_4(\bs\eta) \\
& & \mbox{}-\left(\frac{\mu}{\LL_2(\bs\eta)}-\nu_0\right)\!\!\left(\frac{\mu}{\LL_2(\bs\eta)}+\nu_0\right)(\LL_3(\bs\eta)+\LL_4(\bs\eta)) \\
& & \mbox{}+\frac{\mu^2}{(\LL_2(\bs\eta))^3}(\LL_3(\bs\eta))^{2}
+ O(\mu^\frac{3}{2}(\|\bs\eta\|_{1,\infty} + \|\bs\eta^{\prime\prime}+k_0^2 \bs\eta\|_0)^2),
\end{eqnarray*}
\begin{eqnarray*}
\lefteqn{\langle \MM_\mu'(\bs\eta),\bs\eta \rangle+4 \mu \tilde{\MM}_\mu(\bs\eta)} \qquad\\
& = & - 3\nu_0^2 \LL_3(\bs\eta)
 +4(\KK_4(\bs \eta)-\nu_0^2 \LL_4(\bs\eta)) \\
& & \mbox{}-\left(\frac{\mu}{\LL_2(\bs\eta)}-\nu_0\right)\!\!\left(\frac{\mu}{\LL_2(\bs\eta)}+\nu_0\right)(3\LL_3(\bs\eta)+4\LL_4(\bs\eta)) \\
& & \mbox{}+\frac{4\mu^2}{(\LL_2(\bs\eta))^3} (\LL_3(\bs\eta))^{2} 
+ O(\mu^\frac{3}{2}(\|\bs\eta\|_{1,\infty} + \|\bs\eta^{\prime\prime}+k_0^2 \bs\eta\|_0)^2)
\end{eqnarray*}
and
$$\tilde{\MM}_\mu(\bs\eta) = -\mu^{-1}\left(\frac{\mu}{\LL_2(\bs\eta)}\right)^2 \!(\LL_3(\bs\eta)+\LL_4(\bs\eta)) + O(\mu^\frac{1}{2}(\|\bs\eta\|_{1,\infty} + \|\bs\eta^{\prime\prime}+k_0^2\bs\eta\|_0)^2)$$
hold for each $\bs\eta \in U$ with $\|\bs\eta\|_2 \leq c\mu^\frac{1}{2}$ and $\LL_2(\bs\eta)>c\mu$.
\end{lemma}

The following proposition is an immediate consequence of the definition of $\bs \eta_1$.

\begin{props}
\label{props4.7}
The identity
$$\chi_S \FF\left[\begin{Bmatrix}\LL_3^\prime(\bs\eta_1) \end{Bmatrix} \right] =0$$
holds for each $\bs\eta \in U$.
\end{props}

As a consequence, $\bs\eta_1$ satisfies the equation
\begin{equation}
g(k)\hat{\bs\eta}_1 = \chi_S(k) \FF[ \SS(\bs\eta)], \label{eta1 equation}
\end{equation}
where
\[
\SS(\bs\eta) = \JJ_\mu^\prime(\bs\eta) - \KK_\mathrm{nl}^\prime(\bs\eta)
+\left(\left(\frac{\mu}{\LL(\bs\eta)}\right)^2 -\nu_0^2\right)\LL_2^\prime(\bs\eta)+\left(\frac{\mu}{\LL(\bs\eta)}\right)^{\!\!2} (\LL_\mathrm{nl}^\prime(\bs\eta)-\LL_3^\prime(\bs\eta_1)).
\]
In keeping with equation \eqn{eta1 equation} we write the equation for $\bs\eta_2$ in the form
\begin{equation}
\underbrace{\bs\eta_2+H(\bs\eta)}_{\displaystyle :=\bs\eta_3} = \FF^{-1}\left[(1-\chi_S(k))g(k)^{-1} \FF[\SS(\bs\eta)]\right],
\label{eta3 equation}
\end{equation}
where
\begin{equation}
H(\bs\eta) = -\FF^{-1} \left[g(k)^{-1} \FF \left[ \left(\frac{\mu}{\LL(\bs\eta)}\right)^{\!\!2} \LL_3^\prime(\bs\eta_1)\right]\right];
\label{Definition of H}
\end{equation}
the decomposition $\bs\eta=\bs\eta_1-H(\bs\eta)+\bs\eta_3$ forms the basis of the calculations presented below.
An estimate on the size of $H(\bs\eta)$ is obtained from \eqn{Definition of H} and Proposition \ref{Estimates for mj and nj}.

\begin{props} \label{Estimate for H}
The estimate
$$\|H(\bs\eta)\|_2 \leq c (\|\bs\eta_1\|_{1,\infty} + \|\bs\eta_1^{\prime\prime}+k_0^2\bs\eta_1\|_0 + \|\underline{K}^0 \underline{\eta}_1\|_{1,\infty} 
+ \|\overline{K}^0 \bs\eta_1\|_{1,\infty} 
+ \|\bs\eta_3\|_2) \|\bs\eta_1\|_2$$
holds for each $\bs\eta \in U$.
\end{props}

The above results may be used to derive estimates for the gradients of the cubic parts of the functionals
which are used in the analysis below.

\begin{props} \label{Three gradients eta to eta1}
Any near minimiser $\tilde{\bs\eta}$ satisfies the estimates
$$
\|\LL_3^\prime(\tilde{\bs\eta})-\LL_3^\prime(\tilde{\bs\eta}_1)\|_0
\leq c \mu^\frac{1}{2}((\|\tilde{\bs\eta}_1\|_{1,\infty} + \|\tilde{\bs\eta}_1^{\prime\prime}+k_0^2\tilde{\bs\eta}_1\|_0 + \|\underline{K}^0 \underline{\tilde{\eta}}_1\|_{1,\infty} 
+ \|\overline{K}^0 \tilde{\bs\eta}_1\|_{1,\infty}  )^2+ \|\tilde{\bs\eta}_3\|_2).
$$
\end{props}
{\bf Proof.} Observe that
$$
\LL_3^\prime(\bs\eta)-\LL_3^\prime(\bs\eta_1)=m(H(\bs\eta),H(\bs\eta))+m(\bs\eta_3,\bs\eta_3) - 2m(\bs\eta_1,H(\bs\eta))
-2m(\bs\eta_3,H(\bs\eta))+2m(\bs\eta_1,\bs\eta_3)$$
and estimate the right-hand side of this equation using Propositions \ref{Estimates for mj and nj}
and \ref{Estimate for H}.\qed\\\

An estimate for $\LL_3(\tilde{\bs\eta})$ is obtained in a similar fashion using Propositions \ref{Estimates for mj and nj}, \ref{props4.7}, and \ref{Estimate for H}.

\begin{props} \label{Threes}
Any near minimiser $\tilde{\bs\eta}$ satisfies the estimates
$$
|\LL_3(\tilde{\bs\eta})| 
\leq c\big(\mu (\|\tilde{\bs\eta}_1\|_{1,\infty} + \|\tilde{\bs\eta}_1^{\prime\prime}+k_0^2\tilde{\bs\eta}_1\|_0 +\|\underline K^0 \tilde{\underline{\eta}}_1\|_{1,\infty}+ \|\overline K^0 \tilde{\bs\eta}_1\|_{1,\infty})^2 + \mu\|\tilde{\bs\eta}_3\|_2\big) .
$$
\end{props}

Estimating the right-hand sides of the inequalities
\begin{eqnarray*}
\|\LL_\mathrm{nl}^\prime(\tilde{\bs\eta})-\LL_3^\prime(\tilde{\bs\eta}_1)\|_0 & \leq &
\|\LL_\mathrm{r}^\prime(\tilde{\bs\eta})\|_0 + \|\LL_4^\prime(\tilde{\bs\eta})\|_0 + \|\LL_3^\prime(\tilde{\bs\eta})-\LL_3^\prime(\tilde{\bs\eta}_1)\|_0, \\
|\LL_\mathrm{nl}(\tilde{\bs\eta})| & \leq & |\LL_\mathrm{r}(\tilde{\bs\eta})|+|\LL_4(\tilde{\bs\eta})| + |\LL_3(\tilde{\bs\eta})|,
\end{eqnarray*}
(together with the corresponding inequalities for $\KK$ and $\LL$).
Using Propositions \ref{Size of the functionals} and \ref{Size of the gradients}, the calculation
\begin{eqnarray}
\lefteqn{ \|\bs\eta\|_{1,\infty} + \|\bs\eta^{\prime\prime}+k_0^2\bs\eta\|_0 +  \|\underline{K}^0 \underline{\eta}\|_{\infty} 
+ \|\overline{K}^0 \bs\eta\|_{\infty} } 
\quad \nonumber \\
& & \leq c(\|\bs\eta_1\|_{1,\infty} + \|\bs\eta_1^{\prime\prime}+k_0^2\bs\eta_1\|_0 +  \|\underline{K}^0 \underline{\eta}_1\|_{\infty} 
+ \|\overline{K}^0 \bs\eta_1\|_{\infty}  + \|H(\bs\eta)\|_2 + \|\bs\eta_3\|_2) \nonumber \\
& & \leq c(\|\bs\eta_1\|_{1,\infty} + \|\bs\eta_1^{\prime\prime}+k_0^2\bs\eta_1\|_0 +  \|\underline{K}^0 \underline{\eta}_1\|_{1,\infty} 
+ \|\overline{K}^0 \bs\eta_1\|_{1,\infty}  + \|\bs\eta_3\|_2). \label{General estimate pt 1}
\end{eqnarray}
and Propositions \ref{Three gradients eta to eta1} and \ref{Threes}
yields the following estimates for the `nonlinear' parts of the functionals.

\begin{lemma} \label{Estimates for nl}
Any near minimiser $\tilde{\bs\eta}$ satisfies the estimates
\begin{align*}
& \begin{Bmatrix}
\|\KK_\mathrm{nl}^\prime(\tilde{\bs\eta})\|_0 \\
\|\LL_\mathrm{nl}^\prime(\tilde{\bs\eta})-\LL_3^\prime(\tilde{\bs\eta}_1)\|_0
\end{Bmatrix}
\leq c\big(\mu^\frac{1}{2}(\|\tilde{\bs\eta}_1\|_{1,\infty} + \|\tilde{\bs\eta}_1^{\prime\prime}+k_0^2\tilde{\bs\eta}_1\|_0 + \|\underline{K}^0 \tilde{\underline{\eta}}_1\|_{1,\infty} 
+ \|\overline{K}^0 \tilde{\bs\eta}_1\|_{1,\infty})^2\\
&\qquad\qquad \qquad \qquad\qquad\qquad
+ \mu^\frac{1}{2}\|\tilde{\bs\eta}_3\|_2\big), \\
& \begin{Bmatrix}
|\KK_\mathrm{nl}(\tilde{\bs\eta})| \\
|\LL_\mathrm{nl}(\tilde{\bs\eta})|
\end{Bmatrix}
\leq c\big(\mu(\|\tilde{\bs\eta}_1\|_{1,\infty} + \|\tilde{\bs\eta}_1^{\prime\prime}+k_0^2\tilde{\bs\eta}_1\|_0 +  \|\underline{K}^0 \tilde{\underline{\eta}}_1\|_{1,\infty} 
+ \|\overline{K}^0 \tilde{\bs\eta}_1\|_{1,\infty})^2
+ \mu\|\tilde{\bs\eta}_3\|_2\big).
\end{align*}
\end{lemma}

We now have all the ingredients necessary to estimate the wave speed and the
quantity $\nn \tilde{\bs\eta}_1 \nn_\alpha$.

\begin{props} \label{Weak ST speed estimate}
Any near minimiser $\tilde{\bs\eta}$ satisfies the estimates
$$
\begin{Bmatrix}
\displaystyle \left|\frac{\mu}{\LL(\tilde{\bs\eta})} - \nu_0\right| \\
\\
\displaystyle \left|\frac{\mu}{\LL_2(\tilde{\bs\eta})} - \nu_0\right|
\end{Bmatrix}
\leq c\big((\|\tilde{\bs\eta}_1\|_{1,\infty} + \|\tilde{\bs\eta}_1^{\prime\prime}+k_0^2\tilde{\bs\eta}_1\|_0 + \|\underline{K}^0 \tilde{\underline{\eta}}_1\|_{1,\infty} 
+ \|\overline{K}^0 \tilde{\bs\eta}_1\|_{1,\infty} )^2+ \|\tilde{\bs\eta}_3\|_2
 + \mu^{N-\frac{1}{2}}\big).
$$
\end{props}
{\bf Proof.} Combining Lemma \ref{Formulae for M}, inequality \eqn{General estimate pt 1}
and Lemma \ref{Estimates for nl}, one finds that
\begin{align*}
|\MM_\mu(\tilde{\bs\eta})|,\ |\langle \MM_\mu^\prime(\tilde{\bs\eta}),\tilde{\bs\eta}\rangle+4\mu\tilde{\MM}_\mu(\tilde{\bs\eta})| &\leq
c\big(\mu(\|\tilde{\bs\eta}_1\|_{1,\infty} + \|\tilde{\bs\eta}_1^{\prime\prime}+k_0^2\tilde{\bs\eta}_1\|_0 + \|\underline{K}^0 \tilde{\underline{\eta}}_1\|_{1,\infty} 
+ \|\overline{K}^0 \tilde{\bs\eta}_1\|_{1,\infty} )^2\\
&\qquad+ \mu\|\tilde{\bs\eta}_3\|_2\big),
\end{align*}
$$|\tilde{\MM}_\mu(\tilde{\bs\eta})| \leq c\big((\|\tilde{\bs\eta}_1\|_{1,\infty} + \|\tilde{\bs\eta}_1^{\prime\prime}+k_0^2\tilde{\bs\eta}_1\|_0)^2 +\|\underline{K}^0 \tilde{\underline{\eta}}_1\|_{1,\infty} 
+ \|\overline{K}^0 \tilde{\bs\eta}_1\|_{1,\infty} )^2
+ \|\tilde{\bs\eta}_3\|_2\big),$$
from which the given estimates follow by Proposition \ref{Speed estimate}.\qed

\begin{lemma} \label{Weak ST nn estimate}
Any near minimiser $\tilde{\bs\eta}$ satisfies
$\nn \tilde{\bs\eta}_{1} \nn_\alpha^2 \leq c\mu$,
$\|\tilde{\bs\eta}_{1}^{\bs v_0^\sharp} \|_0^2\leq c\mu^{3+2\alpha}$, 
$\|\tilde{\bs\eta}_3\|_2^2 \leq c\mu^{3+2\alpha}$ and $\|H(\tilde{\bs\eta})\|_2^2 \leq c\mu^{2+\alpha}$ for $\alpha<1$.
\end{lemma}

\begin{proof}
Lemma \ref{Estimates for nl} and Proposition \ref{Weak ST speed estimate} assert that
$$\|\SS(\tilde{\bs\eta})\|_0 \leq c\big( \mu^\frac{1}{2} (\|\tilde{\bs\eta}_1\|_{1,\infty} + \|\tilde{\bs\eta}_1^{\prime\prime}+k_0^2\tilde{\bs\eta}_1\|_0 +\|\underline{K}^0 \tilde{\underline{\eta}}_1\|_{1,\infty} 
+ \|\overline{K}^0 \tilde{\bs\eta}_1\|_{1,\infty} )^2+ \mu^\frac{1}{2}\|\tilde{\bs\eta}_3\|_2
+ \mu^N\big),$$
which shows that
$$\|\tilde{\bs\eta}_3\|_2 \leq c\big( \mu^\frac{1}{2} (\|\tilde{\bs\eta}_1\|_{1,\infty} + \|\tilde{\bs\eta}_1^{\prime\prime}+k_0^2\tilde{\bs\eta}_1\|_0 
+\|\underline{K}^0 \tilde{\underline{\eta}}_1\|_{1,\infty} + \|\overline{K}^0 \tilde{\bs\eta}_1\|_{1,\infty} )^2+ \mu^\frac{1}{2}\|\tilde{\bs\eta}_3\|_2
+ \mu^N\big)$$
and therefore
\begin{equation}
\|\tilde{\bs\eta}_3\|_2 \leq c\big( \mu^\frac{1}{2} (\|\tilde{\bs\eta}_1\|_{1,\infty} + \|\tilde{\bs\eta}_1^{\prime\prime}+k_0^2\tilde{\bs\eta}_1\|_0 +\|\underline{K}^0 \tilde{\underline{\eta}}_1\|_{1,\infty} + \|\overline{K}^0 \tilde{\bs\eta}_1\|_{1,\infty})^2 + \mu^N\big),
\label{Weak ST Gives result for eta3}
\end{equation}
and
\begin{align}\label{eq:}
\begin{aligned}
\int_{\R} |g(k)\FF[\tilde{\bs\eta}_1]|^2 \dk
& \leq  c\big( \mu(\|\tilde{\bs\eta}_1\|_{1,\infty} + \|\tilde{\bs\eta}_1^{\prime\prime}+k_0^2\tilde{\bs\eta}_1\|_0 +\|\underline{K}^0 \tilde{\underline{\eta}}_1\|_{1,\infty} + \|\overline{K}^0 \tilde{\bs\eta}_1\|_{1,\infty})^4+ \mu\|\tilde{\bs\eta}_3\|_2^2
+ \mu^{2N}\big) \\
& \leq  c\big( \mu(\|\tilde{\bs\eta}_1\|_{1,\infty} + \|\tilde{\bs\eta}_1^{\prime\prime}+k_0^2\tilde{\bs\eta}_1\|_0 +\|\underline{K}^0 \tilde{\underline{\eta}}_1\|_{1,\infty} + \|\overline{K}^0 \tilde{\bs\eta}_1\|_{1,\infty})^4
+ \mu^{2N}\big).
\end{aligned}
\end{align}
Multiplying the above inequality by $\mu^{-4\alpha}$ and adding $\|\tilde{\bs\eta}_1\|_0^2 \leq \|\tilde{\bs\eta}\|_0^2 \leq c \mu$, one finds that
\begin{eqnarray}
\nn \tilde{\bs\eta}_1 \nn_\alpha^2 & \leq & c\big( \mu^{1-4\alpha}(\|\tilde{\bs\eta}_1\|_{1,\infty} + \|\tilde{\bs\eta}_1^{\prime\prime}+k_0^2\tilde{\bs\eta}_1\|_0 + \|\underline{K}^0 \tilde{\underline{\eta}}_1\|_{1,\infty} + \|\overline{K}^0 \tilde{\bs\eta}_1\|_{1,\infty} )^4
+ \mu\big) \label{Weak ST Gives result for eta1}\\
& \leq & c(\mu^{1-2\alpha}\nn \tilde{\bs\eta}_1 \nn_\alpha ^4 +\mu) \nonumber
\end{eqnarray}
where Proposition \ref{Estimates for nn} and the inequality
\begin{align}\label{eq:gk}
|g(k) \bs w|^2\ge c(||k|-k_0|^4 |\bs w^{\bs v_0}|^2 + |\bs w^{\bs v_0^\sharp}|^2)\geq c||k|-k_0|^4|\bs w|^2
\end{align}
for $k \in S$ have also been used. The latter follows from \eqref{lower bound on g} and the fact that 
$g(k) \bs v_0^\sharp \ne 0$ for $k\in S$.

The estimate for $\tilde{\bs\eta}_1$ follows from the previous inequality
using the argument given by Groves \& Wahl\'{e}n \cite[Theorem 2.5]{GrovesWahlen10}, while those for
$\tilde{\bs\eta}_3$ and $H(\tilde{\bs\eta})$ are derived by estimating $\nn \tilde{\bs\eta}_1 \nn_\alpha^2 \leq c\mu$
in equation \eqn{Weak ST Gives result for eta3} and Proposition \ref{Estimate for H}.
Finally, as a consequence of \eqref{eq:}--\eqref{eq:gk} we obtain  the inequality
\begin{align*}
\|\tilde{\bs\eta}_{1}^{\bs v_0^\sharp}\|_0^2\leq\,c\,\int_{\R} |g(k)\FF[\tilde{\bs\eta}_1]|^2 \dk\leq\,c(\mu^{1+2\alpha}\nn \tilde{\bs\eta}_1 \nn_\alpha ^4 +\mu^{2N})\leq \,c\,\mu^{3+2\alpha}
\end{align*}
using that $\nn \tilde{\bs\eta}_1\nn ^2_\alpha\leq c\,\mu$.
\end{proof}

The next step is to identify the dominant terms in the formulas for $\MM_\mu(\tilde{\bs\eta})$ and\linebreak
$\langle \MM_\mu^\prime(\tilde{\bs\eta}), \tilde{\bs\eta} \rangle + 4\mu \tilde{\MM}_\mu(\tilde{\bs\eta})$
given in Lemma \ref{Formulae for M}. We begin by examining the quantities $\KK_4(\tilde{\bs\eta})$ and $\LL_4(\tilde{\bs\eta})$ using 
a lemma which allows us to replace Fourier-multiplier operators acting on functions with spectrum localised around certain wavenumbers by multiplication   by constants. The result is a straightforward modification of \cite[Proposition 4.13]{GrovesWahlen11} and \cite[Lemma 4.23]{GrovesWahlen15} and the proof is therefore omitted.

\begin{lemma} \label{Approximate operators}
Assume that $u, v \in H^2(\R)$ with $\supp \hat u, \supp \hat v\subseteq S$ and $\nn u\nn_\alpha, \nn v\nn_\alpha \leq c\mu^\frac{1}{2}$ for some $\alpha<1$ 
and let $u^+:=\FF^{-1}[\chi_{[0,\infty)}\hat u]$, $v^+:=\FF^{-1}[\chi_{[0,\infty)}\hat v]$  and
$u^-:=\FF^{-1}[\chi_{(-\infty,0]}\hat u]$, $v^-:=\FF^{-1}[\chi_{(-\infty,0]}\hat v]$ (so that $u^-=\overline{u^+}$ and $v^-=\overline{v^+}$). 
Then $u$ and $v$ satisfy the estimates 
\begin{list}{(\roman{count})}{\usecounter{count}}
\item
$L(u^{\pm}) =  m(k_0)u^\pm + \underline{O}(\mu^{\frac{1}{2}+\alpha})$,
\item
$L(u^+ v^+)= m(2k_0) u^+ v^+
 + \underline{O}(\mu^{1+\frac{3\alpha}{2}})$,
 \item
$L(u^- v^-)= m(-2k_0) u^- v^-
 + \underline{O}(\mu^{1+\frac{3\alpha}{2}})$,
 \item
$L(u^+ v^-) =m(0)u^+v^-+ \underline{O}(\mu^{1+\frac{3\alpha}{2}})$,
\end{list}
where $L=\FF^{-1}[m(k) \cdot]$ is a Fourier-multiplier operator whose symbol $m$ is locally Lipschitz continuous, and
$\underline{O}(\mu^p)$ denotes a quantity whose Fourier transform has compact support and whose $L^2(\R)$-norm (and hence $H^s(\R)$-norm for $s \geq 0$) is $O(\mu^p)$.
\end{lemma}

\begin{remark}
Note in particular that we can take $L\in \{\partial_x, \underline{K}_0, \overline{K}^0_{ij}\}$ in  estimates (i)--(iv) in Lemma \ref{Approximate operators}  and that 
we can take $m(k)=(g(k)^{-1})_{ij}$ in  (ii)--(iv) since $(g(k)^{-1})_{ij}$ is locally Lipschitz on $\R \setminus S$. 
\end{remark}

Using the formulas  \eqref{lower L4}, \eqref{upper L4}, \eqref{K2 and K4}, Lemmas \ref{Weak ST nn estimate}  and \ref{Approximate operators} (with $\alpha$ sufficiently close to $1$), and the identity 
$\tilde{\bs \eta}_1^{\bs v_0}= \tilde{\underline{\eta}}_1 (1,-a)$  we now obtain the following estimates.

\begin{props}
\label{Quartic terms 1}
Any near minimiser $\tilde{\bs\eta}$ satisfies the estimates
$$
\begin{Bmatrix}
\KK_4(\tilde{\bs\eta}) \\
\LL_4(\tilde{\bs\eta})
\end{Bmatrix}
=
\begin{Bmatrix}
\KK_4(\tilde{\bs\eta}_{1}^{\bs v_0}) \\
\LL_4(\tilde{\bs\eta}_{1}^{\bs v_0})
\end{Bmatrix}
+o(\mu^3).
$$
\end{props}

\begin{props}
\label{Quartic terms 2}
Any near minimiser $\tilde{\bs\eta}$ satisfies the estimates
\begin{align*}
 \KK_4(\tilde{\bs\eta}_{1}^{\bs v_0}) &= A_4^1 \int_{\R}\tilde{\underline{\eta}}_1^4 \dx+ o(\mu^3), \qquad A_4^1= -\frac{1}{8}(\underline{\beta} +\rho\overline{\beta}a^4)k_0^4,\\  \LL_4(\tilde{\bs\eta}_{1}^{\bs v_0}) &= A_4^2 \int_{\R}\tilde{\underline{\eta}}_1^4 \dx+ o(\mu^3), \qquad A_4^2 = \underline A_4^2+\rho \overline A_4^2,\\
\underline A_4^{2}&=-\frac{1}{6}k_0^3  ,\\
\overline A_4^2 &= -\frac{1}{2}\bigg(
 ( \overline{F}_{11}(k_0)-a\overline{F}_{12}(k_0))-a^3
( \overline{F}_{21}(k_0)-a\overline{F}_{22}(k_0)) \bigg)k_0^2,\\
&\quad + \frac{1}{6}
\Big(  \overline{F}_{11}(k_0)-a\overline{F}_{12}(k_0)\Big)^2 \Big(2\overline{F}_{11}(0)+\overline{F}_{11}(2k_0)\Big) \\
&\quad +  \frac{1}{6}
a^2\Big( \overline{F}_{21}(k_0)-a\overline{F}_{22}(k_0)\Big)^2 \Big(2\overline{F}_{22}(0)+\overline{F}_{22}(2k_0)\Big)\\
& \quad -\frac{1}{3}
 a\Big( \overline{F}_{11}(k_0)-a\overline{F}_{12}(k_0)\Big)\Big( \overline{F}_{21}(k_0)-a\overline{F}_{22}(k_0)\Big)
\Big(2\overline{F}_{21}(0)+\overline{F}_{21}(2k_0)\Big).
\end{align*}
\end{props}

\begin{corollary}  \label{lot in fours}
Any near minimiser $\tilde{\bs\eta}$ satisfies the estimate
$$\KK_4(\tilde{\bs\eta}) - \nu_0^2 \LL_4(\tilde{\bs\eta})
= A_4 \int_{\R}\tilde{\underline{\eta}}_1^4 \dx + o(\mu^3),$$
where
$$A_4=A_4^1-\nu_0^2A_4^2.$$
\end{corollary}

We now turn to the corresponding result for $\LL_3(\tilde{\bs\eta})$.
The following result is obtained by writing
\[
\LL_3(\tilde{\bs\eta}) =n(\tilde{\bs\eta}_1^{\bs v_0}+\tilde{\bs\eta}_1^{\bs v_0^\sharp}-H(\tilde{\bs\eta}) +\tilde{\bs \eta}_3,\tilde{\bs\eta}_1^{\bs v_0}+\tilde{\bs\eta}_1^{\bs v_0^\sharp}-H(\tilde{\bs\eta}) +\tilde{\bs \eta}_3, \tilde{\bs\eta}_1^{\bs v_0}+\tilde{\bs\eta}_1^{\bs v_0^\sharp}-H(\tilde{\bs\eta}) +\tilde{\bs \eta}_3),
\]
expanding the right hand side and estimating the terms using Propositions \ref{Estimates for nn} and \ref{Estimates for mj and nj},  Lemma \ref{Weak ST nn estimate} and the identity $n(\tilde{\bs\eta}_1, \tilde{\bs\eta}_1, \tilde{\bs\eta}_1)=0$.

\begin{props} \label{lot in threes step 1}
Any near minimiser $\tilde{\bs\eta}$ satisfies the estimate
$$
\LL_3(\tilde{\bs\eta}) 
= - \int_{\R}
\LL_3^\prime(\tilde{\bs\eta}_{1}^{\bs v_0}) \cdot
H(\tilde{\bs\eta}) \dx+ o(\mu^3).
$$
\end{props}

\begin{props} \label{lot in threes step 2}
Any near minimiser $\tilde{\bs\eta}$ satisfies the estimate
$$H(\tilde{\bs\eta}) = - \nu_0^2\FF^{-1} \left[g(k)^{-1} \FF[ \LL_3^\prime(\tilde{\bs\eta}_1^{\bs v_0})]\right]
+\underline{o}(\mu^3).$$
\end{props}

\begin{proof}
Noting that
\begin{align*}
\left| \frac{\mu}{\LL(\tilde{\bs\eta})} - \nu_0 \right|\ &\leq\ c(\mu^\alpha \nn \tilde{\bs\eta}_1 \nn_\alpha^2 + \|\tilde{\bs\eta}_3\|_2 + \mu^{N-\frac{1}{2}})\ =\ O(\mu^{1+\alpha}),\\
\|\LL_3^\prime(\tilde{\bs\eta}_1)\|_0
 &\leq\ c\mu^\frac{\alpha}{2}\nn \tilde{\bs\eta}_1 \nn_\alpha \|\tilde{\bs\eta}_1\|_2\ =\ O(\mu^{1+\frac{\alpha}{2}}),
\end{align*}
(see Propositions \ref{Estimates for nn} and \ref{Size of the gradients}, Corollary \ref{Weak ST speed estimate} and Lemma \ref{Weak ST nn estimate})  one finds that
\begin{align*}
H(\tilde{\bs\eta})& = - \nu_0^2\FF^{-1} \left[g(k)^{-1} \FF[  \LL_3^\prime(\tilde{\bs\eta}_1)]\right]+O(\mu^{1+\alpha})\underline{O}(\mu^{1+\frac{\alpha}{2}})\\
&=- \nu_0^2\FF^{-1} \left[g(k)^{-1} \FF[  \LL_3^\prime(\tilde{\bs\eta}_1)]\right]+ \underline{o}(\mu^3)
\end{align*}
recalling the definition of $H$ in \eqref{Definition of H}.
The proof is concluded by estimating
\[
\LL_3(\tilde{\bs \eta}_1)-\LL_3(\tilde{\bs \eta}_1^{\bs v_0})=\underline{o}(\mu^3)
\]
(cf.~Propositions \ref{Estimates for nn} and \ref{Estimates for mj and nj},  and Lemma \ref{Weak ST nn estimate}).
\end{proof}

Combining Propositions \ref{lot in threes step 1} and \ref{lot in threes step 2}, one finds that
\begin{equation} \label{lcomb threes}
\begin{aligned}
\LL_3(\tilde{\bs\eta})&= \nu_0^2\int_{\R} g(k)^{-1} \FF[\LL_3'(\tilde{\bs \eta}_1^{\bs v_0})] \cdot \overline{\FF[\LL_3'(\tilde{\bs \eta}_1^{\bs v_0})] } \dk
+o(\mu^3).
\end{aligned}
\end{equation}
Expanding the right hand side using Lemma \ref{Approximate operators} we then obtain the following result.

\begin{props} \label{lot in threes}
Any near minimiser $\tilde{\bs\eta}$ satisfies
\[
- \nu_0^2\LL_3(\tilde{\bs\eta}) = A_3 \int_{\R}\tilde{\underline{\eta}}_1^4 \dx + o(\mu^3),
\]
where
\begin{eqnarray*}
A_3 & = &-\frac{1}{3}g(2k_0)^{-1} \bs A_3^1\cdot \bs A_3^1-\frac{2}{3}g(0)^{-1} \bs A_3^2\cdot \bs A_3^2,\\
\bs A_3^1 &=&\rho \nu_0^2\begin{pmatrix}
\frac32 k_0^2  - \frac12( \overline{F}_{11}(k_0)-a\overline{F}_{12}(k_0))^2
-\overline{F}_{11}(2k_0)  (  \overline{F}_{11}(k_0)-a\overline{F}_{12}(k_0))
 \\
-\frac32 k_0^2a^2
+\frac12( \overline{F}_{21}(k_0)+\overline{F}_{22}(k_0))^2
-a \overline{F}_{22}(2k_0) ( \overline{F}_{21}(k_0)-a\overline{F}_{22} (k_0))
\end{pmatrix}\\
&&+\rho \nu_0^2 \begin{pmatrix}
-a \overline{F}_{21}(2k_0)( \overline{F}_{21}(k_0)-a\overline{F}_{22}(k_0))
 \\
-\overline{F}_{12}(2k_0)  ( \overline{F}_{11}(k_0)-a\overline{F}_{12}(k_0))\end{pmatrix}+\begin{pmatrix}
\nu_0^2 k_0^2  \\0
\end{pmatrix},\\
\bs A_3^2 &=&\rho \nu_0^2
\begin{pmatrix}
\frac12 k_0^2 -
\frac12( \overline{F}_{11}(k_0)-a\overline{F}_{12}(k_0))^2-\overline{F}_{11}(0)(  \overline{F}_{11}(k_0)-a \overline{F}_{12}(k_0))
\\
-\frac12 k_0^2a^2+\frac12( \overline{F}_{21}(k_0)-a\overline{F}_{22}(k_0))^2
-a \overline{F}_{22}(0)(  \overline{F}_{21}(k_0)-a\overline{F}_{22}(k_0))
\end{pmatrix}
\\
&&
+\rho \nu_0^2
\begin{pmatrix}
-a \overline{F}_{21}(0)(  \overline{F}_{21}(k_0)-a \overline{F}_{22}(k_0))
\\
 -\overline{F}_{12}(0) (  \overline{F}_{11}(k_0)-a \overline{F}_{12}(k_0))
\end{pmatrix}.
\end{eqnarray*}
\end{props}

The following estimates for $\MM_\mu(\tilde{\bs\eta})$ and
$\langle \MM_\mu^\prime(\tilde{\bs\eta}), \tilde{\bs\eta} \rangle + 4\mu \tilde{\MM}_\mu(\tilde{\bs\eta})$
may now be derived from Corollary \ref{lot in fours} and Proposition \ref{lot in threes}.

\begin{lemma} \label{Weak ST MM estimate}
The estimates
\begin{eqnarray*}
& & \hspace{-6mm}\MM_{s^2\mu}(s\tilde{\bs\eta}) = -s^3 \nu_0^2 \LL_3(\tilde{\bs\eta})
+s^4\big(\KK_4(\tilde{\bs\eta}) - \nu_0^2 \LL_4(\tilde{\bs\eta})\big)
+ s^3 o(\mu^3),\\
& &  \hspace{-6mm}\langle \MM_{s^2\mu}^\prime(s\tilde{\bs\eta}), s\tilde{\bs\eta} \rangle + 4s^2\mu \tilde{\MM}_{s^2\mu}(s\tilde{\bs\eta})
=-3s^3\nu_0^2 \LL_3(\tilde{\bs\eta}) 
+4s^4\big(\KK_4(\tilde{\bs\eta}) - \nu_0^2 \LL_4(\tilde{\bs\eta}) \big) + s^3o(\mu^3)
\end{eqnarray*}
hold uniformly over $s \in [1,2]$.
\end{lemma}
\begin{proof}
Lemma \ref{Formulae for M} asserts that
\begin{eqnarray*}
\MM_{s^2\mu}(s\tilde{\bs\eta}) & = & -s^3 \nu_0^2 \LL_3(\tilde{\bs\eta})
+s^4\big(\KK_4(\tilde{\bs\eta})  - \nu_0^2 \LL_4(\tilde{\bs\eta})\big) \\
& & \mbox{}-\left(\left(\frac{\mu}{\LL_2(\tilde{\bs\eta})}\right)^2-\nu_0^2\right)(s^3\LL_3(\tilde{\bs\eta})+s^4\LL_4(\tilde{\bs\eta})) \\
& & \mbox{}+\frac{s^4\mu^2}{(\LL_2(\tilde{\bs\eta}))^3}(\LL_3(\tilde{\bs\eta}))^2
+ O(s^5\mu^\frac{3}{2}(\|\tilde{\bs\eta}\|_{1,\infty} + \|\tilde{\bs\eta}^{\prime\prime}+k_0^2 \tilde{\bs\eta}\|_0)^2)
\end{eqnarray*}
uniformly over $s \in [1,2]$. The first result follows by estimating
\[
\|\tilde{\bs\eta}\|_{1,\infty} + \|\tilde{\bs\eta}^{\prime\prime}+k_0^2 \tilde{\bs\eta}\|_0
\ \leq\ c(\mu^\frac{\alpha}{2}\nn \tilde{\bs\eta} \nn_\alpha+\|\tilde{\bs\eta}_3\|_2)
\ \leq\ c\mu^{\frac{1}{2}+\frac{\alpha}{2}}
\]
(see equation \eqn{General estimate pt 1}),
\[
 \LL_3(\tilde{\bs\eta})
= O(\mu^{2+\alpha}),
\quad
 \LL_4(\tilde{\bs\eta})
= O(\mu^{2+\alpha})
\]
(by Propositions \ref{Quartic terms 1}, \ref{Quartic terms 2} and \ref{lot in threes}) and
$$\left|\frac{\mu}{\LL_2(\tilde{\bs\eta})} - \nu_0\right| \leq c( \mu^\alpha \nn \tilde{\bs\eta}_1 \nn_\alpha^2 + \|\bs\eta_3\|_2 + \mu^{N-\frac{1}{2}}) \leq c \mu^{1+\alpha}.$$
The second result is derived in a similar fashion.
\end{proof}

\begin{lemma}
\label{M upper bound}
Any near minimiser satifies the inequality
\begin{equation*}
\MM_\mu(\bs{\tilde \eta})\le -c\mu^3. 
\end{equation*}
\end{lemma}

\begin{proof}
Note first that an arbitrary function $\bs \eta\in U\setminus\{0\}$ satisfies the inequality
\begin{align*}
\frac{\mu^2}{\LL_2(\bs \eta)}+\KK_2(\bs \eta)\ge  2\mu \sqrt{\frac{\KK_2(\bs \eta)}{\LL_2(\bs \eta)}}\ge 2\mu \nu_0
\end{align*}
where we have used that
\[
\KK_2(\bs \eta)- \nu_0^2 \LL_2(\bs \eta)=\frac12 \int_{\R} g(k) \hat{\bs \eta}\cdot \hat{\bs \eta}\dk\ge 0
\]
(cf.~\eqref{lower bound on g}).
The result now follows from the calculation
\[
\MM_\mu(\tilde{\bs \eta})=\JJ_\mu(\tilde{\bs \eta})-\frac{\mu^2}{\LL_2(\tilde{\bs \eta})}-\KK_2(\tilde{\bs \eta}) 
<2\nu_0 \mu-c \mu^3 -2\nu_0 \mu=-c\mu^3. \qedhere
\]
\end{proof}

\begin{corollary} \label{Weak ST final MM estimates}
The estimates
\begin{eqnarray*}
& & \MM_{s^2\mu}(s\tilde{\bs\eta}) = (s^3A_3+s^4A_4)\int_{\R}\tilde{\underline{\eta}}_1^4 \dx+ s^3 o(\mu^3),\\
& & \langle \MM_{s^2\mu}^\prime(s\tilde{\bs\eta}), s\tilde{\bs\eta} \rangle + 4s^2\mu \tilde{\MM}_{s^2\mu}(s\tilde{\bs\eta})
=(3s^3A_3+4s^4A_4)\int_{\R}\tilde{\underline{\eta}}_1^4 \dx+ s^3 o(\mu^3),\\
\end{eqnarray*}
hold uniformly over $s \in [1,2]$ and
$$\int_{\R}\tilde{\underline{\eta}}_1^4 \dx \geq c\mu^3.$$
\end{corollary}
\begin{proof}
The estimates follow by combining
Corollary \ref{lot in fours}, Proposition \ref{lot in threes} and Lemma \ref{Weak ST MM estimate},
while the inequality for $\underline{\tilde{\eta}}_1$ is a consequence of
the first estimate (with $s=1$) and Lemma \ref{M upper bound}. 
\end{proof}

\begin{props}\label{prop:decr}
There exists $s_0\in(1,2]$ and $q>2$ with the property that the function
\begin{align*}
s\mapsto s^{-q}\MM_{s^2\mu}(s\tilde{\bs \eta}),\quad s\in[1,s_0 ]
\end{align*}
is decreasing and strictly negative.
\end{props}

\begin{proof}
This result follows from the calculation
\begin{eqnarray*}
\lefteqn{\frac{d}{ds}\left(s^{-q}\MM_{s^2\mu}(s\tilde{\bs \eta})\right)}\qquad \\
& = & s^{-(q+1)}\left(
-q\MM_{s^2\mu}(s\tilde{\bs\eta})+\langle \MM^\prime_{s^2\mu}(s\tilde{\bs \eta}),s\tilde{\bs \eta}\rangle_0
+ 4s^2\mu\tilde{\MM}_{s^2\mu}(s\tilde{\bs \eta})\right) \\
& = & s^{-(q+1)}\left(\big(-q(s^3A_3+s^4A_4)+3s^3A_3+4s^4A_4\big)\int_{\R}\tilde{\underline{\eta}}_1^4  \dx+ s^3 o(\mu^3)\right) \\
& = & s^{2-q}\left(\big((3-q)A_3+s(4-q)A_4\big)\int_{\R}\tilde{\underline{\eta}}_1^4 \dx + o(\mu^3)\right) \\
& \leq & -c\mu^3 \\
& <  & 0, \hspace{2.5in}s \in (1,s_0),\ q \in (2,q_0),
\end{eqnarray*}
where we have used Corollary \ref{Weak ST final MM estimates} and chosen  $s_0>1$ and $q_0>2$ 
so that $(3-q)A_3+s(4-q)A_4$, which is negative for $s=1$ and $q=2$ (by Assumption \ref{assumption 2}), is also negative for $s \in (1,s_0]$ and $q \in (2,q_0]$.
\end{proof}

The sub-homogeneity of $I_\mu$ now follows (see Groves \& Wahl\'en \cite[Corollary 4.32]{GrovesWahlen15}).

\begin{corollary}\label{cor:subhom}
There exists $\mu_0>0$ such that the map $(0, \mu_0)\ni \mu \mapsto I_\mu$ is strictly sub-homogeneous, that is
\begin{equation}
\label{SSH}
I_{s\mu}<s I_\mu
\end{equation}
whenever $0<\mu<s\mu<\mu_0$.
\end{corollary}

Theorem 3.3 follows from Corollary \ref{cor:subhom} using the argument in Buffoni \cite[p.~48]{Buffoni04a}.

\appendix
\section{Test function}
In order to show that $C_\mu$ is non-empty we have to construct a special test-function. Here the eigenvector $\bs v_0=(1,-a)$ to the eigenvalue $\lambda_-(k_0)$ 
of the matrix $F(k_0)^{-1} P(k_0)$ plays an important role (see Section \ref{Heuristics}).

\begin{lemA}
Suppose that Assumptions \ref{assumption 1} and \ref{assumption 2} hold.
There exists a continuous invertible mapping $\mu \mapsto \varepsilon(\mu)$ such that
\begin{align*}
\JJ(\bs\eta_\star)=2\nu_0\mu+I_\mathrm{NLS} \mu^3 +o(\mu^3),
\end{align*}
where
\begin{align*}
\bs\eta_\star(x)&=\varepsilon \bs\phi(\varepsilon x)\cos(k_0 x)+ 
\varepsilon^2 \bs\psi(\varepsilon x)\cos(2k_0x)+\varepsilon^2\bs\zeta(\varepsilon x),\\
\bs\phi&=\phi_\mathrm{NLS}\bs v_0,\quad \bs\psi=-\frac12 \phi_\mathrm{NLS}^2 g(2k_0)^{-1}\bs A_3^1 \quad \text{and} \quad \bs\zeta=-\frac12 \phi_\mathrm{NLS}^2 g(0)^{-1} \bs A_3^2.
\end{align*}
\end{lemA}

\begin{proof}
We expand the functional $\KK(\bs\eta)-\nu_0^2\LL(\bs\eta)$
evaluated at $\bs \eta_\star$ in powers of $\varepsilon$. We begin by computing the contribution from $\KK(\bs \eta_\star)$.
We have
\begin{align*}
\mathcal{K}(\bs\eta)&=\mathcal{K}_2(\bs\eta)+\mathcal{K}_4(\bs\eta)+\mathcal{K}_\mathrm{r}(\bs\eta),
\end{align*}
where $\KK_2$ and $\KK_4$ are given by \eqref{K2 and K4}.
Using the formulas
\begin{align*}
\bs\eta_\star^\prime(x)&=\varepsilon^2\bs\phi^\prime(\varepsilon x)\cos(k_0 x)-\varepsilon k_0 \bs\phi(\varepsilon x)\sin(k_0 x)\\
&\quad
+\varepsilon^3\bs\psi^\prime(\varepsilon x)\cos(2k_0 x)-2k_0 \varepsilon^2\bs\psi(\varepsilon x)\sin(2k_0 x)
+\varepsilon^3\bs\zeta^\prime(\varepsilon x),\\
\bs\eta_\star''(x)&= \varepsilon^3 \bs\phi''(\varepsilon x)\cos(k_0 x) -2\varepsilon^2 k_0 \bs\phi'(\varepsilon x)\sin(k_0 x)-\varepsilon k_0^2\bs\phi(\varepsilon x)\cos(k_0 x) \\
&\quad +\varepsilon^4 \bs\psi''(\varepsilon x)\cos(2k_0x)-4\varepsilon^3k_0 \bs\psi'(\varepsilon x)\sin(2k_0x)-4\varepsilon^2k_0^2 \bs\psi(\varepsilon x)\cos(2k_0x)+\varepsilon^4\bs\zeta''(\varepsilon x),
\end{align*}
we find that
\begin{align*}
\mathcal{K}_2(\bs\eta_\star)&=
\frac{\varepsilon}{4}(1-\rho+\underline{\beta}k_0^2+a^2\rho(1+\overline{\beta}k_0^2)) \int_{\mathbb R} \phi_\mathrm{NLS}^2\dx +
\frac{\varepsilon^3}{4}(\underline{\beta}+\rho a^2\overline{\beta}) \int_{\mathbb R}  (\phi_\mathrm{NLS}')^2 \dx\\
&\quad + \frac{\varepsilon^3}{4}((1-\rho+4\underline{\beta}k_0^2)\int_{\mathbb R}\underline{\psi}^2\dx+\rho(1+4\overline{\beta}k_0^2)\int_{\mathbb R}\overline{\psi}^2\dx)\\
&\quad+\frac{\varepsilon^3}{2}\left(\int_{\mathbb R}(1-\rho)\underline{\zeta}^2\dx+\int_{\mathbb R} \rho\overline{\zeta}^2\dx\right)+O(\varepsilon^4)
\end{align*}
and
\[
\mathcal{K}_4(\bs\eta_\star)=
-\frac{\varepsilon^3}{64}(3\underline{\beta}k_0^4 +3\rho a^4\overline{\beta}k_0^4) \int_{\R}
\phi_\mathrm{NLS}^4 \dx+O(\varepsilon^4).
\]
Finally, the remainder term satisfies $\KK_\mathrm{r}(\bs\eta_\star)  = \ O(\varepsilon^{\frac{7}{2}})$ in view 
of Proposition \ref{Size of the functionals}

We next compute the contribution from $\LL(\bs \eta)=\underline{\LL}(\underline{\eta})+\rho \overline{\LL}(\bs \eta)$, 
beginning with the first term. Recall that
\[
\underline{\LL}(\underline{\eta})=\underline{\LL}_2(\underline{\eta})+\underline{\LL}_3(\underline{\eta})+\underline{\LL}_4(\underline{\eta})+\underline{\LL}_{\mathrm{r}}(\underline{\eta}),
\]
where $\underline{\LL}_2$, $\underline{\LL}_3$ and $\underline{\LL}_4$  are given in \eqref{lower L2}--\eqref{lower L4}.
In particular, $\underline{\LL}_2(\underline{\eta})= \frac12 \int_{\R} \underline{\eta} \underline{K^0} \underline{\eta}\dx$, 
where $\underline{K}^0$ is the Fourier multiplier operator with symbol $|k|$.
Straightforward calculations yield the formulas
\begin{align*}
\underline{K}^0(\varepsilon^2\bs\zeta(\varepsilon x)) &= \varepsilon^3(\underline{K}^0\bs\zeta)(\varepsilon x),\\
\underline{K}^0(\varepsilon\bs\phi(\varepsilon x)\cos(k_0 x)) 
& =  \varepsilon^2 \bs\phi^\prime(\varepsilon x)\sin(k_0 x) + \varepsilon k_0 \bs\phi(\varepsilon x)\cos(k_0 x) + O(\varepsilon^n),\\
\underline{K}^0(\varepsilon^2\bs\psi(\varepsilon x)\cos(2k_0 x))& = \varepsilon^3 \bs\psi^\prime(\varepsilon x)\sin(2k_0 x) + 2k_0\varepsilon^2 \bs\psi(\varepsilon x)\cos(2k_0 x) + O(\varepsilon^n),
\end{align*}
for each $n \in \N$ uniformly over $x \in \R$ (because $\hat{\phi} \in \SS(\R)$).
Using the formulas \eqref{lower L2}--\eqref{lower L4} we therefore find that
\begin{align*}
\underline{\LL}_2(\underline\eta_\star) &= \frac{\varepsilon k_0}{4}\int_\R \phi_\mathrm{NLS}^2\dx + \frac{\varepsilon^3 k_0}{2} \int_\R\underline{\psi}^2\dx + O(\varepsilon^4),\\
\underline{\LL}_3(\underline\eta_\star) &= -\frac{\varepsilon^3 k_0^2}{4}
\int_{\R} \phi_\mathrm{NLS}^2 \underline{\psi}\dx+O(\varepsilon^4),\\
\underline{\LL}_4(\underline\eta_\star) &= -\frac{\varepsilon^3 k_0^3}{16}\int_{\R} \phi_\mathrm{NLS}^4\dx+O(\varepsilon^4).
\end{align*}
Again, the remainder term satisfies $\underline{\LL}_\mathrm{r}(\underline{\eta}_\star)= O(\varepsilon^{\frac{7}{2}})$ by Proposition \ref{Size of the functionals}.

Similarly,
\[
\overline{\LL}(\bs \eta)=\overline{\LL}_2(\bs \eta)+\overline{\LL}_3(\bs \eta)+\overline{\LL}_4(\bs \eta)+\overline{\LL}_{\mathrm{r}}(\bs \eta),
\]
with $\overline{\LL}_2(\bs \eta)$, $\overline{\LL}_3(\bs \eta)$ and $\overline{\LL}_4(\bs \eta)$ given by 
\eqref{upper L2}--\eqref{upper L4}. In particular, $\overline{\LL}_2(\bs\eta)=
\frac12 \int_{\R} \bs\eta \overline{K}^0\bs\eta \dx$, where
$\overline{K}^0(\bs\eta)=\FF^{-1}[\overline F(k)\widehat{\bs\eta}]$ 
and $\overline F(k)$ is given in \eqref{definition of F}.
Writing
\begin{align*}
\FF[\varepsilon \bs\phi(\varepsilon x)\cos(k_0 x)]
&=\frac{1}{2}\widehat{\phi}_\mathrm{NLS}\Big(\frac{k-k_0}{\varepsilon}\Big)\bs v_0+\frac{1}{2}\widehat{\phi}_\mathrm{NLS}\Big(\frac{k+k_0}{\varepsilon}\Big)\bs v_0,
\end{align*}
we find that
\begin{align*}
\overline F(k)\FF[\varepsilon \bs\phi(\varepsilon x)\cos(k_0 x)]&= \frac{1}{2}\Big[\overline F(k_0)+\overline F'(k_0)(k-k_0)+\frac{1}{2}\overline F''(k_0)(k-k_0)^2\Big]\widehat{\phi}_\mathrm{NLS}\Big(\frac{k-k_0}{\varepsilon}\Big) \bs v_0\\
&\quad+\frac{1}{2}\Big[\overline F(k_0)-\overline F'(k_0)(k+k_0)+\frac{1}{2}\overline F''(k_0)(k+k_0)^2\Big]\widehat{\phi}_\mathrm{NLS}\Big(\frac{k+k_0}{\varepsilon}\Big) \bs v_0\\
&\quad+\hat S_1(k),
\end{align*}
and hence
\begin{align*}
\overline{K}^0(\varepsilon \bs\phi(\varepsilon x)\cos(k_0 x))&=\varepsilon \phi_\mathrm{NLS}(\varepsilon x) \cos(k_0 x)\overline F(k_0)\bs v_0+\varepsilon^2 \phi_\mathrm{NLS}'(\varepsilon x)\sin(k_0 x)\overline F'(k_0)\bs v_0\\
&\quad
-\frac12 \varepsilon^3 \phi_\mathrm{NLS}''(\varepsilon x)\cos(k_0 x)\overline F''(k_0) \bs v_0+S_1(x),
\end{align*}
where $S_1$ satisfies the estimates $\|S_1\|_\infty=O(\varepsilon^3)$, $\|S_1\|_1=O(\varepsilon^{\frac{7}{2}})$.
Similarly, we find that
\begin{align*}
\overline{K}^0(\varepsilon^2 \bs\psi(\varepsilon x)\cos(2k_0x))&
= \varepsilon^2 \cos(2k_0 x) \overline F(2k_0)\bs\psi(\varepsilon x)
+ \varepsilon^3 \sin(2k_0 x)\overline F'(2k_0)\bs\psi'(\varepsilon x)+S_2(x),\\
\overline{K}^0(\varepsilon^2\bs\zeta(\varepsilon x))
&= \varepsilon^2 \overline F(0) \bs\zeta(\varepsilon x)+S_0(x),
\end{align*}
where $S_2$ satisfies the same estimates as $S_1$ and $\|S_2^{(m)}\|_0=O(\varepsilon^{m+3/2})$.
Note also that
\[
\int_{\R}\varphi(\varepsilon x) \left\{\begin{array}{c}\ \sin \\ \cos \end{array}\right\} (m_1 x)
\cdots
\left\{\begin{array}{c}\ \sin \\ \cos \end{array}\right\} (m_\ell x)
\dx
= O(\varepsilon^n)
\]
for all $\varphi\in \SS(\R)$, $n\in \N$ and $m_1,\ldots,m_\ell \in \N$ with $m_1 \pm \ldots \pm m_\ell \neq 0$ (write the product of the trigonometric functions as a linear combination of sine and cosine functions and integrate by parts).
Using the above rules and the formulas \eqref{upper L2}--\eqref{upper L4}  we find that
\begin{align*}
\overline{\LL}_2(\bs\eta_\star)&=
 \frac{\varepsilon}{4} \int_{\mathbb R}\phi_\mathrm{NLS}^2 \overline F(k_0)\bs v_0\cdot\bs v_0\dx +
\frac{\varepsilon^3}{8} \int_{\mathbb R}  
(\phi_\mathrm{NLS}')^2 \overline F''(k_0)\bs v_0 \cdot\bs v_0 \dx\\
&\quad+\frac{\varepsilon^3}{4} \int_{\R}\overline F(2k_0)\bs\psi \cdot \bs\psi\dx
+\frac{\varepsilon^3}{2}\int_{\R} \overline F(0)\bs\zeta \cdot \bs\zeta \dx+
O(\varepsilon^4),
\end{align*}
\begin{align*}
\overline{\LL}_3(\bs \eta_\star)
&=-\frac{3\varepsilon^3k_0^2}{8}\int_{\R} \phi_\mathrm{NLS}^2 \underline{\psi}\dx-\frac{\varepsilon^3k_0^2}4\int_{\R} \phi_\mathrm{NLS}^2 \underline{\zeta}\dx \\
&\quad+ \frac{\varepsilon^3}{4}\big(\overline F_{11}(k_0)-a\,\overline F_{12}(k_0)\big)\left(\overline F_{11}(2k_0)
\int_{\R} \phi_\mathrm{NLS}^2\underline{\psi}\dx+\,\overline F_{12}(2k_0)\int_{\R} \phi_\mathrm{NLS}^2\overline{\psi}\dx\right)\\
&\quad+\frac{\varepsilon^3}{8}(\overline F_{11}(k_0)-a\,\overline F_{12}(k_0))^2\int_{\R} \phi_\mathrm{NLS}^2\underline{\psi}\dx
+\frac{\varepsilon^3}{4}\big(\overline F_{11}(k_0)-a\,\overline F_{12}(k_0))^2\int_{\R} \phi_\mathrm{NLS}^2\underline{\zeta}\dx\\
&\quad+\frac{\varepsilon^3}{2}\big(\overline F_{11}(k_0)-a\,\overline F_{12}(k_0)\big)\left(\overline F_{11}(0)\int_{\R} \phi_\mathrm{NLS}^2\underline{\zeta}\dx+\overline F_{12}(0)\int_{\R} \phi_\mathrm{NLS}^2\overline{\zeta}\dx\right)\\
&\quad+
\frac{3\varepsilon^3k_0^2a^2}{8}\int_{\R} \phi_\mathrm{NLS}^2 \overline{\psi}\dx+\frac{\varepsilon^3k_0^2a^2}{4}\int_{\R} \phi_\mathrm{NLS}^2 \overline{\zeta}\dx\\
&\quad+\frac{\varepsilon^3a}{4}\big(\overline F_{21}(k_0)-a\,\overline F_{22}(k_0)\big)\left(\overline F_{21}(2k_0)\int_{\R} \phi_\mathrm{NLS}^2\underline{\psi}\dx+\,\overline F_{22}(2k_0)\int_{\R} \phi_\mathrm{NLS}^2\overline{\psi}\dx\right)\\
&\quad-\frac{\varepsilon^3}{8}\big(\overline F_{21}(k_0)-a\,\overline F_{22}(k_0)\big)^2\int_{\R} \phi_\mathrm{NLS}^2\overline{\psi}\dx
-\frac{\varepsilon^3}{4}\big(\overline F_{21}(k_0)-a\,\overline F_{22}(k_0)\big)^2\int_{\R} \phi_\mathrm{NLS}^2\overline{\zeta}\dx\\
&\quad+\frac{\varepsilon^3a}{2}\big(\overline F_{21}(k_0)-a\,\overline F_{22}(k_0)\big)\left(\overline F_{21}(0)\int_{\R} \phi_\mathrm{NLS}^2\underline{\zeta}\dx+\overline F_{22}(0)\int_{\R} \phi_\mathrm{NLS}^2\overline{\zeta}\dx\right)+O(\varepsilon^4),
\end{align*}
and
\begin{align*}
 \overline{\LL}_4(\bs\eta_\star)
&=\varepsilon^3\left[-\frac{3k_0^2}{16}(\overline F_{11}(k_0)-a\,\overline F_{12}(k_0)) +\frac{3a^3k_0^2}{16}(\overline F_{21}(k_0)-a\,\overline F_{22}(k_0)) \right.\\
&\qquad\quad+
\frac{1}{16}(\overline F_{11}(k_0)-a\,\overline F_{12}(k_0) )^2(\overline F_{11}(2k_0)+2\overline F_{11}(0)) \\
&\qquad \quad+\frac{a}{8}(\overline F_{11}(k_0)-a\,\overline F_{12}(k_0) )(\overline F_{21}(2k_0)+2\overline F_{21}(0)) (\overline F_{21}(k_0)-a\,\overline F_{22}(k_0) ) \\
&\left.\qquad \quad +\frac{a^2}{16}(\overline F_{21}(k_0)-a\,\overline F_{22}(k_0) )^2(\overline F_{22}(2k_0)+2\overline F_{22}(0)) \right] \int_{\R}
\phi_\mathrm{NLS}^4\dx+O(\varepsilon^4).
\end{align*}
Again, $\overline{\LL}_\mathrm{r}(\bs\eta_\star) = \ O(\varepsilon^{\frac{7}{2}})$.

Combining the above expansions, we find that
\begin{align*}
\varepsilon^{-3}(\KK(\bs\eta_\star)-\nu_0^2 \LL(\bs\eta_\star))
&=\frac{1}8 g''(k_0)\bs v_0\cdot\bs v_0 \int_{\R}  (\phi_\mathrm{NLS}')^2\dx\\
&\qquad+\left(\frac{3}{8}A_4-\frac1{16} g(2k_0)^{-1}\bs A_3^1 \cdot \bs A_3^1-
\frac1{8} g(0)^{-1}\bs A_3^2 \cdot \bs A_3^2\right)\int_{\R}\phi_\mathrm{NLS}^4 \dx\\
&\qquad
+\int_{\R} \frac14 g(2k_0)(\bs\psi+\tfrac12 \phi_\mathrm{NLS}^2 g(2k_0)^{-1} \bs A_3^1)\cdot 
(\bs\psi+\tfrac12 \phi_\mathrm{NLS}^2 g(2k_0)^{-1}\bs A_3^1)\dx\\
&\qquad 
+\int_{\R} \frac12 g(0)(\bs\zeta+\tfrac12 \phi_\mathrm{NLS}^2 g(0)^{-1}\bs A_3^2)\cdot 
(\bs\zeta+\tfrac12 \phi_\mathrm{NLS}^2 g(0)^{-1}\bs A_3^2)\dx\\
&\qquad +O(\varepsilon^{\frac12}).
\end{align*}
Our choice of $\bs \psi$ and $\bs \zeta$ minimises the above expression up to $O(\varepsilon^{\frac{7}{2}})$ and we obtain that
\begin{align*}
\KK(\bs\eta_\star)-\nu_0^2 \LL(\bs\eta_\star)&=\varepsilon^3\left(A_2\int_{\R}  (\phi_\mathrm{NLS}')^2\dx+A_3\int_{\R}\phi_\mathrm{NLS}^4 \dx\right) +O(\varepsilon^{\frac{7}{2}})\\
&=\varepsilon^3 E_\mathrm{NLS}(\phi_\mathrm{NLS})+O(\varepsilon^{\frac{7}{2}})\\
&= \varepsilon^3  I_\mathrm{NLS}+O(\varepsilon^{\frac{7}{2}}).
\end{align*}

The mapping
\[
\varepsilon \mapsto \nu_0 \LL(\bs\eta_\star)=\varepsilon\frac{\nu_0F(k_0)\bs v_0\cdot \bs v_0}{4}\int_{\R} \phi_\mathrm{NLS}^2 \dx
+O(\varepsilon^3)
\]
is continuous and strictly increasing and therefore has a continuous inverse 
$\mu \mapsto \varepsilon(\mu)$, such that $\varepsilon(\mu)=\mu+o(\mu)$. Furthermore,
\[
\JJ_\mu(\bs \eta_\star)-2\nu_0\mu= \KK(\bs \eta_\star)-\nu_0^2 \LL(\bs \eta_\star)=
I_\mathrm{NLS} \mu^3+o(\mu^3),
\]
which concludes the proof.
\end{proof}

\noindent\\
{\bf Acknowledgement.} E. Wahl\'{e}n was supported by the Swedish Research Council (grant no. 621-2012-3753).

\bibliographystyle{siam} 
\bibliography{wellen}

\end{document}